\documentclass[leqno]{article}
\usepackage{amsmath, amscd, amsthm, amssymb, graphics, xypic, mathrsfs, setspace, fancyhdr, times, bm, enumitem}
\usepackage[colorlinks=true,pagebackref=true]{hyperref} 

\newcommand{\tensor}{\otimes}

\newcommand{\Spec}{\operatorname{Spec}}

\newcommand{\isomto}{{\stackrel{\sim}{\;\longrightarrow\;}}}
\newcommand{\isomt}{{\stackrel{{\scriptscriptstyle{\sim}}}{\;\rightarrow\;}}}

\newcommand{\sma}{{\scriptstyle{\wedge}}}
\newcommand{\coker}{\operatorname{coker}}
\newcommand{\longhookrightarrow}{\lhook\joinrel\longrightarrow}

\renewcommand{\O}{{\mathcal O}}
\renewcommand{\hom}{\operatorname{Hom}}
\newcommand{\real}{{\mathbb R}}
\newcommand{\cplx}{{\mathbb C}}
\newcommand{\Q}{{\mathbb Q}}
\newcommand{\Z}{{\mathbb Z}}

\newcommand{\A}{{\mathbb A}}

\newcommand{\aone}{{\mathbb A}^1}

\newcommand{\gm}[1]{{\mathbb{G}_{\op{m}}^{#1}}}

\newcommand{\MW}{\mathrm{MW}}

\newcommand{\et}{\mathrm{\acute{e}t}}
\newcommand{\ho}[1]{\mathscr{H}({#1})}
\newcommand{\hop}[1]{\mathscr{H}_{\bullet}({#1})}

\newcommand{\bpi}{\bm{\pi}}

\newcommand{\Nis}{\operatorname{Nis}}
\newcommand{\Zar}{\operatorname{Zar}} 

\newcommand{\Sm}{\mathrm{Sm}}
\newcommand{\Spc}{\mathrm{Spc}}
\newcommand{\K}{{{\mathbf K}}}

\newcommand{\op}{\operatorname}

\newcommand{\Singaone}{\operatorname{Sing}^{\aone}\!\!}

\newcommand{\Addresses}{{
  \bigskip
  \footnotesize

  A.~Asok, \textsc{Department of Mathematics, University of Southern California,
    Los Angeles, CA 90089-2532, United States;} \textit{E-mail address:} \url{asok@usc.edu}

  \medskip

  M.~Hoyois, \textsc{Department of Mathematics, Massachusetts Institute of Technology,
    Cambridge, MA 02139-4307, United States;} \textit{E-mail address:} \url{hoyois@mit.edu}

  \medskip

  M.~Wendt, \textsc{} \textit{E-mail address:} \url{m.wendt.c@gmail.com}

}}

\newcounter{intro}
\theoremstyle{plain}
\newtheorem{thm}{Theorem}[subsection]

\newtheorem{lem}[thm]{Lemma}
\newtheorem{cor}[thm]{Corollary}
\newtheorem{prop}[thm]{Proposition}
\newtheorem*{claim*}{Claim}  

\newtheorem*{thm*}{Theorem}
\newtheorem*{problem*}{Problem}

\newtheorem{thmintro}{Theorem}

\theoremstyle{definition}
\newtheorem{defn}[thm]{Definition}

\theoremstyle{remark}
\newtheorem{rem}[thm]{Remark}
\newtheorem{remintro}[thmintro]{Remark}

\newtheorem{ex}[thm]{Example}

\numberwithin{equation}{section}

\begin{document}
\pagestyle{fancy}
\renewcommand{\sectionmark}[1]{\markright{\thesection\ #1}}
\fancyhead{}
\fancyhead[LO,R]{\bfseries\footnotesize\thepage}
\fancyhead[RO]{\bfseries\footnotesize\rightmark}
\chead[]{}
\cfoot[]{}
\setlength{\headheight}{1cm}

\author{Aravind Asok\thanks{Aravind Asok was partially supported by National Science Foundation Award DMS-1254892.} \and Marc Hoyois\thanks{Marc Hoyois was partially supported by National Science Foundation Award DMS-1508096}\and Matthias Wendt\thanks{Matthias Wendt was partially supported by EPSRC grant EP/M001113/1 and DFG SPP 1786.}}

\title{{\bf Generically split octonion algebras \\ and $\aone$-homotopy theory}}
\date{}
\maketitle

\begin{abstract}
We study generically split octonion algebras over schemes using techniques of ${\mathbb A}^1$-homotopy theory.  By combining affine representability results with techniques of obstruction theory, we establish classification results over smooth affine schemes of small dimension.  In particular, for smooth affine schemes over algebraically closed fields, we show that generically split octonion algebras may be classified by characteristic classes including the second Chern class and another ``mod $3$" invariant.  We review Zorn's ``vector matrix" construction of octonion algebras, generalized to rings by various authors, and show that generically split octonion algebras are always obtained from this construction over smooth affine schemes of low dimension.  Finally, generalizing P. Gille's analysis of octonion algebras with trivial norm form, we observe that generically split octonion algebras with trivial associated spinor bundle are automatically split in low dimensions.
\end{abstract}

\begin{footnotesize}
\tableofcontents
\end{footnotesize}

\section{Introduction}
In this paper, we analyze problems related to ``generically split" octonion algebras using techniques of $\aone$-homotopy theory \cite{MV}.  In particular we: (i) study classification of generically split octonion algebras over schemes, (ii) analyze when generically split octonion algebras may be realized in terms of a generalization of M. Zorn's construction of the octonions, and (iii) study when generically split octonion algebras are determined by their norm forms.

Recall that the group scheme $\op{G}_2$ may be identified as the automorphism group scheme of a split octonion algebra (see, e.g., \cite[Chapter 2]{SpringerVeldkamp}).  More generally, the set of isomorphism classes of octonion algebras over a scheme $X$, pointed by the split octonion algebra, is in natural (pointed) bijection with the pointed set $\op{H}^1_{\et}(X;\op{G}_2)$ parameterizing \'etale locally trivial torsors under the split group scheme $\op{G}_2$; we refer the reader to, e.g., \cite[\S 2.4]{Giraud} for more details about non-abelian cohomology.  Generically split octonion algebras are those lying in the kernel of the restriction map $\op{H}^1_{\et}(X;\op{G}_2)\to \op{H}^1_{\et}(k(X);\op{G}_2)$ and, by a result of Nisnevich, correspond precisely to Nisnevich locally trivial $\op{G}_2$-torsors (at least over regular schemes).  We will briefly recall definitions regarding octonion algebras over schemes in Section \ref{ss:octonionalgebrasdefs} and recall a suitably categorified version of the dictionary between octonion algebras and torsors under $\op{G}_2$.

Granted the identification of generically split octonion algebras in terms of Nisnevich locally trivial $\op{G}_2$-torsors, they may be investigated by appeal to the techniques of \cite{AffineRepresentabilityI,AffineRepresentabilityII}.  In particular, the pointed set of Nisnevich locally trivial torsors under $\op{G}_2$ over a smooth affine scheme $X$ (over an infinite field) is naturally in bijection with the set of maps in the $\aone$-homotopy category from $X$ to a suitable classifying space ${\op{B}_{\Nis}}\op{G}_2$.  We may then analyze the set of $\aone$-homotopy classes of maps $[X,{\op{B}_{\Nis}}\op{G}_2]_{\aone}$ using techniques of obstruction theory.

\subsubsection*{Classification of generically split octonion algebras}
Our initial goal is to discuss classification results for generically split octonion algebras over schemes.  Recall that the classification of octonion algebras over fields is completely controlled by a single cohomological invariant living in the Galois cohomology group $\op{H}^3_{\et}(k;\mathbb{Z}/2\mathbb{Z})$ (see, e.g., \cite[Appendix 2.3.3]{Serre}).  Much less is known about the classification of octonion algebras over schemes.

Since an octonion algebra over a scheme $X$ consists of a finite rank vector bundle over $X$ equipped with a suitable multiplication, it is natural to expect that any classification will depend on invariants of the underlying vector bundle, e.g., Chern classes. For example, one may define a second Chern class of an octonion algebra (this is related but unequal to the second Chern classes of the underlying vector bundle; see Proposition~\ref{prop:secondstageBG2} and Remark~\ref{rem:secondchernclass} for more details).  In our context, the natural invariants that appear arise from $k$-invariants in the $\aone$-Postnikov tower of ${\op{B}_{\Nis}}\op{G}_2$.  These $k$-invariants may be described using $\aone$-homotopy groups of ${\op{B}_{\Nis}}\op{G}_2$, which we compute in low degrees.  For example, we are able to establish the following classification result.

\begin{thmintro}[{See \ref{thm:g2classification}}]
\label{thm:main1}
Assume $k$ is an infinite field and $X$ is a smooth affine $k$-scheme.
\begin{enumerate}[noitemsep,topsep=1pt]
\item If $X$ has dimension $\leq 2$, then the map $c_2: \op{H}^1_{\Nis}(X;\op{G}_2) \to \op{CH}^2(X)$ is a bijection, i.e., generically split octonion algebras are determined by their second Chern class.
\item If $X$ has dimension $\leq 3$ and $k$ has characteristic unequal to $2$, then there is an exact sequence of the form:
\[
\op{coker}\left(\Omega k_3:\op{H}^1_{\Nis}(X;\K^{\op{M}}_2)\to\op{H}^3_{\Nis}(X;\K^{\op{M}}_4/3)\right) \longrightarrow \op{H}^1_{\Nis}(X;\op{G}_2) \stackrel{c_2}{\longrightarrow} \op{CH}^2(X) \longrightarrow 0.
\]
In particular, if $k$ is algebraically closed and has characteristic unequal to $2$, then $c_2$ is bijective.
\end{enumerate}
\end{thmintro}

\begin{remintro}
It follows from the classification of octonion algebras over fields that if $k(X)$ has \'etale $2$-cohomological dimension $\leq 2$, then all $\op{G}_2$-torsors over $X$ are generically trivial and the evident inclusion $\op{H}^1_{\Nis}(X;\op{G}_2) \to \op{H}^1_{\et}(X;\op{G}_2)$ is a bijection.  Thus, for example if $k$ is algebraically closed and $X$ has dimension $\leq 2$, then any octonion algebra over $X$ is determined by its second Chern class.
\end{remintro}

\begin{remintro}
Our appeal to techniques of $\aone$-homotopy theory imposes restrictions that, at this stage, seem unavoidable: our classification results will only work for smooth affine schemes over a field, which at the moment must be infinite.  The restriction that the base field be infinite arises from the fact that our classification results in \cite{AffineRepresentabilityII} require an infinite base field (the current proof of $\aone$-invariance of $\op{G}_2$-torsors requires this assumption).
\end{remintro}

\begin{remintro}
If the base field $k$ has characteristic $0$, then we may augment our analysis by appeal to topological realization functors.  More precisely, if $X$ is a smooth affine scheme, then we may compare the algebraic classification problem studied above with the classification of principal bundles under the groups $\op{G}_2(\cplx)$ or $\op{G}_2(\real)$ over $X(\cplx)$ or $X(\real)$ (these have the homotopy type of finite CW complexes). Example~\ref{ex:c2examples} provides smooth affine surfaces with uncountably many octonion algebras that become holomorphically trivial after complex realization (this furthermore uses a version of the Grauert--Oka principle asserting that on Stein manifolds, the holomorphic and topological classification of principal bundles under the complex reductive group $\op{G}_2(\cplx)$ coincide).
\end{remintro}

\subsubsection*{Comparison of octonion algebras and Zorn algebras}
M. Zorn \cite{Zorn} gave a construction of the split octonions using what one now calls ``vector matrices".  Loosely speaking, one may realize the split octonions using ``matrices" built using the classical cross product on a $3$-dimensional vector space.  Zorn's construction can be globalized; our treatment follows \cite[\S 3]{Petersson}, though related constructions appear elsewhere in the literature (see, e.g., \cite{KnusParimalaSridharan}).  In the end, one obtains a procedure that associates with an oriented rank $3$ vector bundle over a scheme $X$ (i.e., a vector bundle equipped with a trivialization of its determinant bundle) an octonion algebra over $X$; we will refer to octonion algebras that arise in this way as Zorn algebras (see Section~\ref{sec:octonionprelims} for a review of these constructions).

Another interpretation of Zorn's construction is as follows.  There is a homomorphism of algebraic groups $\op{SL}_3 \hookrightarrow \op{G}_2$ corresponding to the ``embedding of the long roots" and this homomorphism yields a natural transformation of functors
\[
\op{Zorn}(-):\op{H}^1_{\Nis}(-;\op{SL}_3) \longrightarrow \op{H}^1_{\Nis}(-;\op{G}_2).
\]
Zorn algebras are precisely those octonion algebras associated with $\op{G}_2$-torsors in the image of this map. Two natural questions arise as regards the map $\op{Zorn}$: when is it surjective, and when is it injective?  The former question amounts to asking: when is every generically split octonion algebra a Zorn algebra, and the latter amounts to asking: if an octonion algebra is a Zorn algebra, does it admit a unique presentation as such?  Our next main result answers this question in various situations.

\begin{thmintro}[{See Theorem~\ref{thm:primaryobstructiontobeingazornalgebra} and Corollary~\ref{cor:zornsecond}}]
\label{thm:main2}
Assume $k$ is an infinite field.  If $X$ is a smooth affine scheme over $k$, then the morphism $\op{Zorn}$ is surjective under any of the following hypotheses:
\begin{enumerate}[noitemsep,topsep=1pt]
\item $\dim X\leq 3$, or
\item $k$ algebraically closed and $\dim X \leq 4$, or
\item $k$ algebraically closed having characteristic $0$, $\dim X\leq 5$ and $\op{CH}^4(X)/2 = 0$.
\end{enumerate}
Furthermore, the morphism $\op{Zorn}$ is bijective if $\dim X\leq 2$.
\end{thmintro}

\begin{remintro}
The dimension bound in Theorem~\ref{thm:main2} over general base fields is sharp in a homotopical sense: Proposition~\ref{prop:zornreductionex} provides an exact sequence
\[
\op{H}^1_{\Nis}(\op{Q}_8;\op{SL}_3) \longrightarrow \op{H}^1_{\Nis}(\op{Q}_8;\op{G}_2) \longrightarrow  \mathbf{I}(k)\to 0
\]
which shows that for fields with non-trivial fundamental ideal $\mathbf{I}(k)$ there are octonion algebras over $\op{Q}_8$ that are not isomorphic to Zorn algebras.  Note also that if $X$ is a smooth affine scheme with $\dim X\geq 3$, even if a generically split octonion algebra is a Zorn algebra, its presentation as such may be non-unique.  This follows from, e.g., Theorem~\ref{thm:main1} since the third Chern class of the oriented vector bundle is not visible in the classification of octonion algebras.
\end{remintro}

\subsubsection*{Octonion algebras and their norm forms}
Generalizing the fact that octonion algebras over a field come equipped with a norm, the vector bundle underlying an octonion algebra over a scheme comes equipped with a norm form giving it the structure of a quadratic space of rank $8$.  As above this operation can be reinterpreted using the theory of algebraic groups.  There is a homomorphism $\op{G}_2 \to \op{O}_8$ such that the induced natural transformation of functors
\[
\op{H}^1_{\et}(-;\op{G}_2) \longrightarrow \op{H}^1_{\et}(-;\op{O}_8)
\]
is the one sending an octonion algebra to its underlying norm form.

It is a classical result of Springer--Veldkamp that octonion algebras over fields are determined by their associated norm forms (see, e.g., \cite[Theorem 1.7.1]{SpringerVeldkamp}).  R. Bix extended this result to octonion algebras over local rings in which $2$ is invertible \cite[Lemma 1.1]{Bix}.  In a 2012 lecture in Lens, H. Petersson posed the question of whether this same fact remained true over more general bases.  Petersson's question may be recast, at least in the context of generically split octonion algebras, as asking whether the natural transformation of the previous paragraph is injective.

The inclusion $\op{G}_2 \to \op{Spin}_8$ naturally factors through an embedding $\op{Spin}_7 \to \op{Spin}_8$ (corresponding to restricting the norm form to ``purely imaginary" octonions).  Therefore, one may recast the question about determination by norm forms in terms of a map analogous to that above, but associated with the embedding $\op{G}_2 \to \op{Spin}_7$ (we will refer to the $\op{Spin}_7$-torsor associated with a $\op{G}_2$-bundle as the associated spinor bundle).  Consequently, a first step towards understanding injectivity is understanding the kernel of the above natural transformation.  We analyze this question in the context of generically split octonion algebras, i.e., we would like to understand those generically split octonion algebra with trivial associated spinor bundle.

\begin{thmintro}[{See Theorem~\ref{thm:trivialnormform}}]
\label{thm:main3}
Assume $k$ is an infinite field.  If $X$ is a smooth affine $k$-variety, then generically split octonion algebras with trivial spinor bundles are split if all oriented rank $3$ vector bundles which become free after addition of a  free rank $1$ summand are already free.  In particular, generically split octonion algebras with trivial spinor bundle are always split if either
\begin{enumerate}[noitemsep,topsep=1pt]
\item $X$ has dimension $\leq 2$, or
\item $X$ has dimension $3$, and the base field satisfies one of the following hypotheses: (a) $k$ has characteristic unequal to $2$ or $3$ and has \'etale $2$ and $3$-cohomological dimension $\leq 2$, (b) $k$ is perfect of characteristic $2$ and has \'etale $3$-cohomological dimension $\leq 1$, (c) $k$ is perfect of characteristic $3$ and has \'etale $2$-cohomological dimension $\leq 1$, or
\item $k$ is algebraically closed, has characteristic unequal to $2$ or $3$ and $X$ has dimension $\leq 4$.
\end{enumerate}
\end{thmintro}

In \cite{GilleOctonion}, P. Gille gave examples of octonion algebras over seven-dimensional schemes that are not determined by their norm forms. Gille's argument was topologically inspired, detecting the examples by appeal to the computation $\pi_7(BG_2) \cong\mathbb{Z}/3\Z$.  Our analysis is, in spirit, similar.  In fact, our results may be viewed as an algebro-geometric ``lift" of the topological computations: we appeal to obstruction theory in the $\aone$-homotopy category associated with the map ${\op{B}_{\Nis}}\op{G}_2\to {\op{B}_{\Nis}}\op{Spin}_7$.  Moreover, the proof shows Gille's examples are ``universal examples" of octonion algebras with trivial norm form.

Gille also asked for minimal dimensional varieties carrying non-trivial octonion algebras with trivial associated norm form \cite[Concluding Remarks (1)]{GilleOctonion}.  We complement Theorem~\ref{thm:main3} by showing our dimension estimates are ``best-possible".  Indeed, using constructions of stably free non-free bundles by Mohan Kumar \cite{MohanKumarstablyfree}, we observe in Theorem~\ref{thm:mohankumar} that there exist smooth affine varieties of dimension $4$ over a $C_1$-field or dimension $5$ over an algebraically closed field carrying a generically split octonion algebra with trivial associated norm form.

\begin{remintro}
For special varieties, the computations can be made extremely explicit.  We refer the reader to Examples~\ref{ex:q6}, \ref{ex:nontrivialiftingclass} and \ref{ex:q6q7} for analysis of octonion algebras over low-dimensional split affine quadrics.  In fact, we may even classify generically split octonion algebras with trivial spinor bundles in the revelant examples.
\end{remintro}

Furthermore, using our techniques, we are able to show that octonion algebras are determined by their norm forms in low-dimensional situations, which shows that an analog of the result of Springer--Veldkamp does indeed hold in some cases (there is a variant of the second point that holds in case the base field has characteristic $2$ as well).

\begin{thmintro}[{See Theorem \ref{thm:octonionalgebrasaredeterminedbytheirnormforms}}]
\label{thmintro:main4}
Assume $k$ is an infinite field and $X$ is an irreducible smooth affine $k$-scheme of dimension $d \leq 2$.
\begin{enumerate}[noitemsep,topsep=1pt]
\item Any generically split octonion algebra on $X$ is determined by its norm form.
\item If, furthermore, $k$ has characterstic unequal to $2$ and $k(X)$ has $2$-cohomological dimension $\leq 2$, then {\em any} octonion algebra on $X$ is determined by its norm form.
\end{enumerate}
\end{thmintro}

\begin{remintro}
\label{rem:alsaodygilleremark}
If $R$ is a commutative unital ring, then S. Alsaody and P. Gille \cite{AlsaodyGille} have recently studied the isomorphism classes of octonion algebras with isometric norm forms, i.e., the fibers of the map $\op{H}^1_{\et}(\Spec R,\op{G}_2) \to \op{H}^1_{\et}(\Spec R,\op{SO}_8)$.  Their perspective is completely different from ours and, for example, they identify the fibers of the above map using the notion of {\em isotope} of an octonion algebras \cite[Corollary 6.7]{AlsaodyGille}.
\end{remintro}

\subsubsection*{Overview of sections}
In Section~\ref{sec:octonionprelims}, we recall basic properties of octonion algebras, Zorn algebras and relate these notions to torsors for algebraic groups, associated vector bundles and associated homogeneous spaces.  In Section~\ref{sec:a1computations}, we review relevant facts from $\aone$-homotopy theory, compute homotopy sheaves of $\op{B}_{\Nis}\op{G}_2$ and compare these computations with corresponding topological analogs.  Finally, Section~\ref{sec:octonionclassification} puts everything together to establish the results described above.   We refer the reader to the beginning of each section for a more precise description of its contents.

\subsubsection*{Acknowledgements}
We would like to thank Philippe Gille for email exchanges on classification results for octonion algebras long ago.  We also thank him for sharing an early draft of \cite{AlsaodyGille}.  We would also like to thank Brian Conrad for discussions about $\op{G}_2$ and related homogeneous spaces over $\Z$; in particular, he kindly provided Lemma~\ref{lem:reductiontogeometricfibers} and its proof to us.  Finally, we thank Mike Hopkins for helpful discussions.

\subsubsection*{Preliminaries/Notation}
Throughout, $k$ will denote a fixed commutative unital base ring; further restrictions will be imposed along the way.  We write $\Sm_k$ for the category of schemes that are smooth and have finite type over $k$.  We write $\op{Q}_n$ for the affine quadric hypersurfaces defined for $n = 2m-1$ by $\sum_{i=1}^m x_iy_i = 1$, and for $n = 2m$ by $\sum_i x_iy_i = z(1+z)$; these quadrics appear throughout the paper and are smooth over $\Spec \Z$.

We write $\Spc_k$ for the category of simplicial presheaves on $\Sm_k$ and $\Spc_{k,\bullet}$ for the corresponding pointed variant.  We write $\ho{k}$  (resp. $\hop{k}$) for the Morel-Voevodsky $\aone$-homotopy category \cite{MV}; this category is obtained by performing a Bousfield localization of $\Spc_k$ (resp. $\Spc_{k,\bullet}$).  For the sake of convenience, this is done in two steps: first, one Nisenevich localizes (see \cite[\S 3.1-2]{AffineRepresentabilityI}), and then one $\aone$-localizes (see \cite[\S 5.1]{AffineRepresentabilityI}).  This construction differs slighty from that of Morel--Voevodsky (who use simplicial sheaves instead of presheaves), but has an equivalent outcome.

If $\mathscr{X}$ and $\mathscr{Y}$ are two spaces, then we set
\[
[\mathscr{X},\mathscr{Y}]_{\A^1}:= \op{Hom}_{\ho{k}}(\mathscr{X},\mathscr{Y});
\]
similar notation will be used for maps in the pointed homotopy categories.  We write $\op{S}^i$ for the simplicial $i$-sphere, i.e., the constant presheaf associated with the simplicial $i$-sphere.  Then, given a pointed space $(\mathscr{X},x)$, we define $\aone$-homotopy sheaves $\bpi_i^{\aone}(\mathscr{X},x)$ as the Nisnevich sheaves on $\Sm_k$ associated with the presheaves
\[
U\mapsto [\op{S}^i \wedge U_+,(\mathscr{X},x)]_{\A^1}.
\]
The sheaves $\bpi_i^{\aone}(\mathscr{X},x)$ are sheaves of groups for $i \geq 1$ and sheaves of abelian groups for $i \geq 2$.  Isomorphisms in $\ho{k}$ will be called $\aone$-weak equivalences.

All group schemes that appear will be assumed pointed by their identity section.  In particular, we write $\gm{}$ for the usual multiplicative group, viewed as a pointed space.  In that case, for $j \geq 0$ we may form the motivic spheres $\op{S}^i \wedge \gm{\sma j}$.  We will routinely use the fact \cite[Theorem 2]{AsokDoranFasel} that the scheme $\op{Q}_{2n-1}$ is $\aone$-weakly equivalent to $\op{S}^{n-1} \wedge \gm{\sma n}$ while the scheme $\op{Q}_{2n}$ is $\aone$-weakly equivalent to $\op{S}^n \wedge \gm{\sma n}$.  Analogous to the homotopy sheaves considered above, we write $\bpi_{i,j}^{\aone}(\mathscr{X},x)$ for the sheaves obtained by replacing $\op{S}^i$ by $\op{S}^i \wedge \gm{\sma j}$ in the definition of the previous paragraph.

If $\op{G}$ is a Nisevich sheaf of groups, then we write ${\op{B}}\op{G}$ for the usual simplicial bar construction on $\op{G}$.  We write ${\op{B}_{\Nis}}\op{G}$ for a Nisnevich local replacement of ${\op{B}}\op{G}$; this space classifies Nisnevich locally trivial torsors in a suitable sense (see \cite[\S 2.2]{AffineRepresentabilityII} for more details about this construction).  However, the space ${\op{B}_{\Nis}}\op{G}$ is not typically $\aone$-local; if it is we will say that $\op{G}$ is a strongly $\aone$-invariant sheaf of groups.

If $\op{G}$ is furthermore a smooth group scheme, then one may form the geometric classifying space ${\op{B}_{\op{gm}}}\op{G}$; this classifying space is called geometric because it may be realized as an explicit colimit of smooth schemes and was studied independently by Morel--Voevodsky \cite[\S 4]{MV} and Totaro \cite{Totaro}.  In general, i.e., for group-schemes that are not ``special" in the sense of Serre, ${\op{B}_{\Nis}}\op{G}$ is not $\aone$-weakly equivalent to ${\op{B}_{\op{gm}}}\op{G}$.

If $\mathbf{A}$ is a Nisnevich sheaf of abelian groups on $\Sm_k$, and $i \geq 0$ is an integer, then we write $\op{K}(\mathbf{A},i)$ for the (Nisnevich local) Eilenberg--Mac Lane space.  We will also write $\op{H}^i(\op{X},\mathbf{A})$ for cohomology of $\mathbf{A}$ restricted to the small Nisnevich site of $\op{X}$.  Sometimes we will need to mention Zariski or \'etale cohomology as well, and in those cases, cohomology groups will be decorated with a subscript of $\Zar$ or $\et$ as necessary.  In the situations we consider $\mathbf{A}$ will frequently be a {\em strictly $\aone$-invariant sheaf}; equivalently, $\op{K}(\mathbf{A},i)$ will already be $\aone$-local for every $i \geq 0$.  In that case, $\op{K}(\mathbf{A},i)$ has precisely one non-trivial $\aone$-homotopy sheaf, which appears in degree $i$ and is isomorphic to $\mathbf{A}$.  Moreover, we freely use the fact that if $X$ is a smooth $k$-scheme, and $\mathbf{A}$ is strictly $\aone$-invariant then $[X,\op{K}(\mathbf{A},i)]_{\aone} = \op{H}^i(X,\mathbf{A})$.

When we study obstruction theory, we will use additional facts about strongly and strictly $\aone$-invariant sheaves, all due to F. Morel in \cite{MField} (though see \cite[\S 2.2-3]{AsokWickelgrenWilliams} for a convenient summary). In particular, if $(\mathscr{X},x)$ is a pointed space that is pulled back from one defined over a perfect field, then the sheaf $\bpi_i^{\aone}(\mathscr{X},x)$ is strongly $\aone$-invariant for $i \geq 1$, and strictly $\aone$-invariant for $i \geq 2$ (or $i \geq 1$ if it is already abelian).

We also use the contraction construction: if $\mathbf{G}$ is a strongly $\aone$-invariant sheaf, we write $\mathbf{G}_{-1}$ for the sheaf defined by
\[
\mathbf{G}_{-1}(U) := \ker(\mathbf{G}(\gm{} \times U) \stackrel{(1 \times \op{id})^*}{\longrightarrow} \mathbf{G}(U));
\]
the sheaves $\mathbf{G}_{-i}$ are defined by iterating the contraction construction.  By \cite[Lemmas 2.32 and 7.33]{MField}, the assignment $(-)_{-1}$ defines an endo-functor of the category of strictly (or strongly) $\aone$-invariant sheaves that preserves exact sequences.  Moreover, by \cite[Theorem 6.13]{MField}, we know  $\bpi_{i,j}^{\aone}(\mathscr{X},x) = \bpi_i^{\aone}(\mathscr{X})_{-j}$.

\section{Octonion algebras, algebra and geometry}
\label{sec:octonionprelims}
In this section, we recall basic definitions and constructions related to octonion algebras over schemes.  Definitions of such notions over general rings go back to (among others) McCrimmon, Knus and Petersson, but we follow \cite{Petersson} and \cite{Conrad}.  In particular, we highlight the key theorems that we use over general base rings; our main sources for such results are \cite{LoosPeterssonRacine} and \cite{Conrad}.  There are a number of classical geometric facts related to subgroups, embeddings and homogeneous spaces of $\op{G}_2$: the existence of a closed immersion group homomorphisms $\op{SL}_3 \to \op{G}_2$ going back to \cite{Jacobson}, the embedding $\op{G}_2 \to \op{Spin}_7$ (resp. $\op{G}_2 \to \op{Spin}_8$) mentioned in the introduction, and explicit descriptions of the homogeneous spaces $\op{G}_2/\op{SL}_3$ and $\op{Spin}_7/\op{G}_2$ as explicit quadric hypersurfaces.  Many of these results are well-known over fields (see, e.g,. \cite{SpringerVeldkamp}), but they hold over more general rings too, and our main contribution here is to establish the results in their natural generality and formulate them in a way convenient for later homotopy-theoretic applications; related results have also been obtained by Alsaody and Gille \cite{AlsaodyGille}.

\subsection{Octonion algebras and Zorn's vector matrices}
\label{ss:octonionalgebrasdefs}
We begin by recalling some definitions related to octonion algebras.  There are many sources for the results we need over fields, but it is harder to find the analogous results over more general bases without additional hypotheses (which are frequently unnecessary).  In particular, we review the Zorn algebra construction (Definition \ref{defn:zorn}), the relationship between octonion algebras and torsors under the split simply-connected simply group scheme of type $\op{G}_2$ (Theorem \ref{thm:g2torsorsandoctonionalgebras}), and the so-called ``long root" embedding $\op{SL}_3 \to \op{G}_2$ (Theorem \ref{thm:longrootembedding}).

\subsubsection*{Octonion algebras}
We refer the reader to \cite{Knus} for discussion of quadratic spaces.

\begin{defn}
\label{defn:octonionalgebra}
Let $R$ be a commutative ring with unit. An \emph{octonion algebra} over $R$ is a quadruple $(O,\circ,1_0,\op{N}_O)$ where (i) $O$ is a rank $8$ projective $R$-module, (ii) $\circ: O \times O \to O$ is a binary composition, (iii) $1_0: R \to O$ is an $R$-module homomorphism acting as a two-sided unit for $\circ$, and (iv) $\op{N}_O: O \to R$ is an $R$-module homomorphism making the pair $(O,\op{N}_O)$ into a non-degenerate quadratic space such that the composition identity:
\[
\op{N}_O(x\circ y)=\op{N}_O(x)\cdot \op{N}_O(y)
\]
is satisfied for all $x,y\in O$.  The pair $(O,\op{N}_O)$ is called the \emph{associated norm form} of the octonion algebra.
\end{defn}

\begin{rem}
What we have called octonion algebras following Springer--Veldkamp \cite[1.11]{SpringerVeldkamp} are referred to in the literature also as Cayley algebras \cite{KnusParimalaSridharan} or \cite[\S 33]{BookofInvolutions}.
\end{rem}

Suppose we are given an octonion algebra $O$ (as usual, we will suppress the additional data).  Write $\langle-,-\rangle$ for the symmetric bilinear form associated with the norm form $\op{N}_O$.  Given this notation, one introduces the following notions:
\begin{enumerate}[noitemsep,topsep=1pt]
\item a {\em conjugation} involution:
\[
\begin{split}
(-)^\vee:O &\longrightarrow O \\
x &\longmapsto x^\vee:=\langle x,1_O\rangle-x;
\end{split}
\]
one checks that conjugation preserves $1_O$.
\item An $R$-linear {\em trace map}:
\[
\begin{split}
\op{Tr}_O:O &\longrightarrow R \\
 x &\longmapsto \op{Tr}_O(x):=\langle x,1_O\rangle.
\end{split}
\]
Alternatively, one checks that the formula $\op{Tr}_O(x)=x+x^\vee$ holds.
\end{enumerate}

\begin{ex}
\label{ex:splitoctonion}
The basic example of an octonion algebra is given by the \emph{split octonion algebra}.  If $R$ is any commutative unital ring, consider the free rank $8$ module $\op{M}_2(R) \oplus \op{M}_2(R)$.  If $x \in \op{M}_2(R)$ is a $2 \times 2$-matrix, then write $x^* \in \op{M}_2(R)$ for the matrix obtained by conjugation with the matrix
\begin{footnotesize}$\begin{pmatrix}0 & 1 \\ -1 & 0\end{pmatrix}$\end{footnotesize}.  Define a composition on $\op{M}_2(R) \oplus \op{M}_2(R)$ by means of the formula
\[
(x,y) \circ (z,w) = (xz + wy^*,x^*w + zy)
\]
and define $N := \det x - \det y$.  Using the fact that $xx^* = \det x$ on $\op{M}_2(R)$ and $x^* = -x$ on the subset of $\op{M}_2(R)$ consisting of trace zero matrices, one checks that setting $1_O := (\op{Id}_2,0) \in \op{M}_2(R) \oplus \op{M}_2(R)$ defines an octonion algebra.  The conjugation operation in $O$ is given by $(x,y)^{\vee} = (x^*,-y)$ and the trace map is defined by $\op{Tr}_O(x,y) = \op{Tr}(x)$.
\end{ex}

\begin{rem}
\label{rem:split}
We will call an octonion algebra {\em split} if it is isomorphic as an octonion algebra to the algebra in Example \ref{ex:splitoctonion}.  This definition conflicts with the notion of {\em split} used in \cite[1.8]{Petersson} (where it means ``contains a split complex sub-algebra").  Our terminology is, however, consistent with usage of the term ``split" in the theory of algebraic groups; see Theorems~\ref{thm:octonionsandg2} and \ref{thm:g2torsorsandoctonionalgebras} and Example \ref{ex:splitoctonionsaszorn} for further discussion of this point.
\end{rem}

\begin{rem}
\label{rem:homomtoO8}
For later use, we also require the following fact.  If $x$ is an octonion, then $xx^{\vee} = \op{N}_O(x) = x^{\vee}x$.  Moreover, for any octonion $y$, the identities $x(x^{\vee}y) = \op{N}_O(x)y = (yx)x^{\vee}$ hold.  In particular, we conclude that if $x$ is a unit norm octonion, then $x^{\vee}$ provides both a left and right inverse for $x$.  Thus, the operations of left or right multiplication by a unit norm octonion provide invertible $R$-linear maps $O \to O$.  Such automorphisms automatically preserve $\op{N}_O$ by virtue of the fact that $\op{N}_O$ satisfies the composition identity.
\end{rem}

As usual, one defines a category $\mathrm{Oct}(R)$ of octonion algebras over $R$.  If $R \to S$ is a morphism of commutative unital rings, then there are ``extension-of-scalars" functors $(-) \tensor_R S: \mathrm{Oct}(R) \to \mathrm{Oct}(S)$.  If $R$ is an integral domain with fraction field $K$, then an octonion algebra $O$ over $R$ is {\em generically split} if the octonion algebra obtained by extension of scalars $O \otimes_R K$ is isomorphic to the split octonion algebra over $K$.  Since the norm form of a split octonion algebra is a hyperbolic form, if $O$ is a generically split octonion algebra, its norm form is generically hyperbolic.  The following key result shows that octonion algebras are locally split in the \'etale topology.  Moreover, it also shows that if an octonion algebra has a norm form that is generically hyperbolic, then it is in fact generically split.

\begin{ex}
\label{ex:holomorphicsetting}
We may speak of octonion algebras in the complex analytic setting using this terminology.  In particular, if $X$ is a Stein manifold, then by a holomorphic octonion algebra over $X$, we mean an octonion algebra over the ring $R^{\op{hol}}$ of holomorphic functions on $X$.  In particular if $X = \Spec R$ is a smooth affine $\cplx$-scheme and $R^{\op{hol}}$ is the ring of holomorphic functions on $X$, then extension-of-scalars along the map $R \to R^{\op{hol}}$ yields a functor from octonion algebras over $R$ to holomorphic octonion algebras.
\end{ex}

\begin{thm}[See {\cite[Theorem B.5]{Conrad}}]
\label{thm:g2zarloctriv}
If an octonion algebra over a local ring $R$ has underlying quadratic space that is split, then it is isomorphic to a split octonion algebra.  Every octonion algebra over a ring $S$ becomes split after extension of scalars along a ring map $S \to S'$ such that $\Spec S' \to \Spec S$ is an affine \'etale cover.
\end{thm}

\begin{rem}
The first part of the above result for local rings in which $2$ is invertible follows from results of Bix \cite[Theorem 1.3]{Bix}.
\end{rem}

\subsubsection*{Zorn's vector matrices}
We now recall another standard construction of octonion algebras going back to Zorn \cite{Zorn}.  Roughly speaking, Zorn constructed the ``classical" octonions using the standard cross product on a $3$-dimensional real vector space.  This construction was given in much greater generality in \cite[\S 3]{Petersson} or \cite[\S 4.2]{LoosPeterssonRacine}.  In fact, the construction works for locally ringed spaces and thus over schemes, but we use it in this generality only in passing.

Recall that an oriented projective module $P$ over a commutative unital ring $R$ is a pair $(P,\varphi)$ consisting of a projective $R$-module $P$ and an isomorphism $\varphi: \det P \isomt R$. Write $P^*$ for the $R$-module dual of $P$.  Using the orientation, one may construct analogs of the usual cross product on $P$ and $P^*$:
\[
\begin{split}
\times_{\varphi}: P \times P &\longrightarrow P^{*}, \text{ and } \\
\times_{\varphi}: P^* \times P^* &\longrightarrow P.
\end{split}
\]
Abusing notation, write $\langle -,- \rangle$ for the usual evaluation bilinear forms $P \times P^* \to R$ and $P^* \times P \to R$.  Using the operations just described, one may equip the rank $8$ projective $R$-module $R^{\oplus 2} \oplus P \oplus P^*$ with an octonion algebra structure.

\begin{defn}
\label{defn:zorn}
Suppose $R$ is a commutative unital ring, and $(P,\varphi)$ is an oriented rank $3$ projective $R$-module.  Write $\op{Zorn}((P,\varphi))$ for the projective $R$-module
\[
\begin{bmatrix}
R & P \\
P^* & R
\end{bmatrix}
\]
equipped with the product
\[
\begin{pmatrix}
a_1 & x^+ \\
x^- & a_2
\end{pmatrix}\circ\begin{pmatrix}
b_1 & y^+ \\
y^- & b_2
\end{pmatrix}
= \begin{pmatrix}
a_1b_1 - \langle x^+,y^- \rangle & a_1y^+ + b_2x^+ + x^- \times_{\varphi} y^- \\
b_1x^- + a_2y^- + x^+ \times_{\varphi} y^+ & -\langle x^-,y^-\rangle + a_2b_2
\end{pmatrix};
\]
the unit given by
\begin{footnotesize}$\begin{pmatrix}
1 & 0\\
0 & 1
\end{pmatrix}$
\end{footnotesize} and norm map given by:
\[
\op{N}\left(\begin{pmatrix}
a_1 & x^+ \\
x^- & a_2
\end{pmatrix}\right) = a_1a_2 + \langle x^-,x^+ \rangle.
\]
\end{defn}

\begin{rem}
\label{rem:traceformzorn}
The trace map for $\op{Zorn}((P,\varphi))$ admits a very simple presentation: it is given by the usual matrix trace, i.e.,
\[
\op{Tr}\left(\begin{pmatrix}
a_1 & x^+ \\
x^- & a_2
\end{pmatrix}\right) = a_1 + a_2.
\]
\end{rem}

\begin{rem}
\label{rem:zornisafunctor}
The construction of the Zorn algebra of an oriented projective module yields a functor from the category of oriented projective modules of rank $3$ over $R$ to the category $\mathrm{Oct}(R)$.  This functor is compatible with extension of scalars in the following sense: if $R \to S$ is a ring homomorphism, then there is a natural isomorphism of the form:
\[
\op{Zorn}((P,\varphi) \tensor_R S) \cong \op{Zorn}((P,\varphi)) \tensor_R S.
\]
\end{rem}

\begin{ex}
\label{ex:splitoctonionsaszorn}
If we equip $R^{\oplus 3}$ with the standard orientation, then $\op{Zorn}(R^{\oplus 3})$ is a split octonion algebra (see Example~\ref{ex:splitoctonion}).  We caution the reader that terminology varies from source to source: for example, in Petersson's terminology any Zorn algebra is split \cite[\S 1.8 and Theorem 3.5]{Petersson}, while Loos--Racine--Petersson refer to algebras of the form $\op{Zorn}((P,\varphi))$ as reduced.  We will simply refer to octonion algebras in the image of $\op{Zorn}$ as {\em Zorn algebras}.
\end{ex}

\subsection{Octonion algebras and \texorpdfstring{$\op{G}_2$}{G2}-torsors}
Over a field, it is well-known that the automorphism group of a split octonion algebra is isomorphic to the split semi-simple algebraic group $\op{G}_2$ (see, e.g., \cite[Theorem 2.3.5]{SpringerVeldkamp}).  Over a field having characteristic unequal to $2$, there is a well-known embedding $\op{SL}_3 \to \op{G}_2$, see \cite[Theorem 4]{Jacobson}.  In terms of Dynkin diagrams, this embedding corresponds to the observation that the Dynkin diagram $\op{A}_2$ is a sub-diagram of that of $\op{G}_2$ and, in terms of root systems, corresponds to the inclusion of the long roots.  We review these constructions here in greater generality.

\subsubsection*{Octonion algebras and $\op{G}_2$-torsors}
The relationship between octonion algebras and $\op{G}_2$-torsors holds over an arbitrary base ring (see \cite[Appendix B]{Conrad} or \cite[\S 3]{GilleOctonion}).  If $O$ and $O'$ are two octonion algebras over a ring $R$, then for any $R$-algebra $S$, we can consider the assignment
\[
S \longmapsto \underline{\op{Isom}}(O,O')(S) := \op{Isom}(O \tensor_R S,O' \tensor_R S).
\]
We write $\op{Aut}_{O/R}$ for $\underline{\op{Isom}}(O,O)$. 

\begin{thm}[See {\cite[Theorem B.14 and Corollary B.15]{Conrad}}]
\label{thm:octonionsandg2}
For any octonion algebra $O$ over a commutative unital ring $R$, the affine finite-type automorphism scheme $\op{Aut}_{O/R}$ of the algebra is a semi-simple $R$-group scheme of type $\op{G}_2$; this $R$-group scheme is the Chevalley group of type $\op{G}_2$ (i.e., split and simply connected) if $O$ is isomorphic to a split octonion algebra. Moreover, the assignment $O \mapsto \op{Aut}_{O/R}$ determines a bijection between the set of isomorphism classes of octonion algebras over $R$ and the set of isomorphism classes of semi-simple $R$-group schemes of type $\op{G}_2$.
\end{thm}

Theorem \ref{thm:octonionsandg2} in conjunction with \cite[Theorem 4.10]{LoosPeterssonRacine}, yields the following result.

\begin{thm}
\label{thm:g2torsorsandoctonionalgebras}
If $Z$ is a split octonion algebra, then the assignment $O \mapsto \underline{\op{Isom}}(Z,O)$ determines an equivalence between the groupoid $\op{Oct}(R)$ and the groupoid of \'etale locally trivial torsors under the split group scheme $\op{G}_2$.
\end{thm}

\begin{proof}
\cite[Theorem 4.10]{LoosPeterssonRacine} shows that $\underline{\op{Isom}}(Z,O)$ is an \'etale locally trivial torsor under $\op{Aut}_{Z/R}$.  By Theorem \ref{thm:octonionsandg2}, $\op{Aut}_{Z/R}$ is the split form of $\mathrm{G}_2$.  It remains to observe that the assignment of the statement is functorial.  Indeed, if $O \to O'$ is an isomorphism of octonion algebras, the induced map of functors is a natural equivalence and thus the associated map of torsors is an isomorphism as well.
\end{proof}

In the sequel, we use the above result to pass back and forth between analysis of \'etale locally trivial $\mathrm{G}_2$-torsors and octonion algebras.  The following example shows how this observation will be used.

\begin{ex}
Suppose $R$ is a regular domain.  Since $\op{G}_2$ is a semi-simple group scheme, it follows from \cite{Nisnevich} that $\op{G}_2$-torsors over $R$ that are generically trivial are automatically Nisnevich locally trivial (the converse is immediate).  As a consequence, it follows from the identifications above that the groupoid of Nisnevich locally trivial $\mathrm{G}_2$-torsors over $R$ is, under the equivalence of Theorem \ref{thm:g2torsorsandoctonionalgebras}, identified with the subcategory of $\op{Oct}(R)$ consisting of generically split octonion algebras.
\end{ex}

\subsubsection*{$\mathrm{G}_2$-torsors over fields}
The classification of $\mathrm{G}_2$-torsors over a field can be given in explicit cohomological terms, see \cite[Appendix 2.3.3]{Serre}.  Indeed, a result of Springer--Veldkamp \cite[Theorem 1.7.1]{SpringerVeldkamp} shows that two octonion algebras over a field are isomorphic if and only if their norm forms are isometric.  Since octonion algebras are composition algebras of dimension $8$, their associated norm forms are so-called $3$-fold Pfister forms (in characteristic $2$, they are quadratic Pfister forms).  One then observes that $3$-fold Pfister forms over a field may be classified in terms of Galois cohomology.

If one works over a field $k$ having characteristic unequal to $2$, the isometry classes of $3$-fold Pfister forms are determined by decomposable elements in $\op{H}^3_{\et}(k,\mu_2^{\tensor 3})$, i.e., elements of the form $\alpha_1 \cup \alpha_2 \cup \alpha_3$ where $\alpha_i \in \op{H}^1_{\et}(k,\mu_2) = k^{\times}/k^{\times 2}$ (see, e.g., \cite[Theorem 33.25]{BookofInvolutions}).  In other words, there is an injective map
\[
\op{H}^1_{\et}(K,\op{G}_2) \longhookrightarrow \op{H}^3_{\et}(k,\mu_2^{\tensor 3})
\]
whose image is contained in the subset of decomposable elements.  In fact, the image coincides with the subset of decomposable elements; see \cite[Appendix 2.3.3]{Serre} for further discussion.  A similar result holds in characteristic $2$.  As before, classifying $\op{G}_2$-torsors over a field $k$ can be achieved in terms of quadratic Pfister forms \cite[Th\'eor\`eme 11]{SerreUpdate}, though one must replace $\op{H}^3_{\et}(k,\mu_2^{\tensor 3})$ by a group defined in terms of differential forms.  As a consequence, one deduces the following result.

\begin{lem}
\label{lem:generictrivialitylowdimensions}
Assume $k$ is a field having characteristic unequal to $2$.  If $X = \Spec R$ is a smooth affine $k$-scheme, and $k(X)$ has \'etale $2$-cohomological dimension $\leq 2$, then all octonion algebras over $X$ are generically split.
\end{lem}

\begin{proof}
If $K$ is a field having characteristic unequal to $2$ and \'etale $2$-cohomological dimension $\leq 2$ (e.g., if $K$ is a function field of a smooth variety of dimension $\leq 2$ over a separably closed field having characteristic unequal to $2$), then all octonion $K$-algebras are trivial by appeal to the classification of $\op{G}_2$-torsors over fields.  Now, apply this to $K = k(X)$.
\end{proof}

\begin{rem}
\label{rem:characteristic2generictriviality}
See \cite[\S 10.3 Corollaire]{SerreUpdate} for a suitable variant of this statement over fields having characteristic $2$.
\end{rem}

\subsubsection*{The long root embedding}
The ``long root embedding" may be generalized using octonion algebras that arise via Zorn's vector matrix construction.  Over an arbitrary commutative unital ground ring, a construction was given in \cite{LoosPeterssonRacine}; we recall this construction here.  To begin, consider the presentation of $O$ as the Zorn algebra of $R^{\oplus 3}$ as in  Example~\ref{ex:splitoctonionsaszorn}.  There is an embedding of the algebra $C := R \oplus R$ in $O$ as diagonal matrices (the notation is meant to be suggestive of the Cayley--Dickson doubling process: the algebra $C$ is a copy of the ``split complex numbers" sitting inside $O$).  We write $\op{Aut}_{O/C}$ for the sub-functor of $\op{Aut}_{O/R}$ consisting of automorphisms of $O$ that restrict to the identity on $C$.

\begin{thm}
\label{thm:longrootembedding}
Suppose $R$ is a commutative unital ring.  There is an isomorphism $\Phi: \op{SL}_3 \isomt \op{Aut}_{O/C}$ of $R$-group schemes where for any $R$-algebra $S$, and $g \in \op{SL}_3(S)$, $\Phi(g)$ is given by the formula
\[
\Phi(g)\begin{pmatrix}
a_1 & x^+ \\
x^- & a_2
\end{pmatrix} = \begin{pmatrix}
a_1 & gx^+ \\
{}^{\op{t}}g^{-1}x^- & a_2
\end{pmatrix}.
\]
In particular, the immersion $\op{Aut}_{O/C} \hookrightarrow \op{Aut}_{O/R}$ yields a closed immersion group homomorphism $\op{SL}_3 \to \op{G}_2$.
\end{thm}

\begin{proof}
This is a specialization of \cite[Theorem 5.9]{LoosPeterssonRacine} to the case of the split octonion algebra; the second part of the statement is then an immediate consequence of Theorem \ref{thm:octonionsandg2}.
\end{proof}

\begin{rem}
\label{rem:orthogonalembedding}
For later use, we will recall some facts about $\op{Aut}_{O/C}$.  Consider
\[
C^{\perp} := \{ x \in O \mid \op{Tr}_O( x\bar{y}) = 0 \;\forall y \in O\}.
\]
Because $O$ is an alternative algebra, $C^{\perp}$ admits the structure of a (left) $C$-module.   There is a $C$-valued Hermitian form on $C^{\perp}$ preserved by every element of $\op{Aut}_{O/C}$.  In more detail, conjugation restricted to $z \in C$ is given by the formula $\bar{z} := (a_2,a_1)$.  If we identify $C^{\perp}$ with $C^{\oplus 3}$, then given $u = (u_1,u_2,u_3)$ and $w = (w_1,w_2,w_3)$ in $C^{\oplus 3}$, the Hermitian form is given by the explicit formula
\[
\Phi(z,w) = z_1\bar{w}_1 + z_2 \bar{w}_2 + z_3 \bar{w}_3.
\]
In fact, Theorem \ref{thm:longrootembedding} can be phrased as saying that $\op{Aut}_{O/C}$ is identified with the special unitary group of this $C$-valued Hermitian form.
\end{rem}

\subsection{Homogeneous spaces related to \texorpdfstring{$\op{G}_2$}{G2} and octonions}
\label{ss:homogeneousspaces}
We now analyze various geometric structures related to the group scheme $\op{G}_2$ and the long root embedding $\op{SL}_3 \to \op{G}_2$.  Then we study the structure of homogeneous spaces associated with these embeddings.  The main results of this section are Proposition \ref{prop:spin7g2}, Theorem \ref{thm:g2sl3andspin7} and Proposition \ref{prop:spin8g2}, which are algebro-geometric versions of well-known results for compact Lie groups.

\subsubsection*{Geometry of some group homomorphisms}
By Theorem \ref{thm:octonionsandg2}, we may identify $\op{G}_2$ as the automorphism group scheme of a split octonion algebra $O$ over $\Z$ and thus over any base ring.  The norm form $\op{N}_O$ of the split octonion algebra is the split quadratic form of rank $8$.  Since any automorphism of an octonion algebra preserves the norm \cite[Proposition B.1]{Conrad}, it follows that over any base ring there is an embedding
\[
\op{G}_2 \longrightarrow \op{O}_8,
\]
where we write $\op{O}_8$ for the orthogonal group of the split quadratic form $\op{N}_O$.

We write $\op{SO}_n$ for the special orthogonal group of a split form of rank $n$, see \cite[Appendix C]{ConradReductive} or \cite[Chapter IV]{Knus} for discussion of special orthogonal groups over rings.  By \cite[Theorem C.2.11]{ConradReductive}, the group scheme $\op{SO}_n$ is smooth with connected fibers.  Since $\op{G}_2$ is smooth with connected fibers by construction, it follows that the above inclusion factors through $\op{SO}_8$. Note that in any characteristic, the even special orthogonal group is the connected component containing the identity section of $\op{O}_8$ by \cite[Corollary C.3.1]{ConradReductive}.

Now the split form of rank $8$ is a non-degenerate quadratic space.  As mentioned above, any algebra automorphism of $O$ preserves the trace form as well.  The subscheme $O' \subset O$ consisting of matrices with trace $0$ is a trivial rank $7$ vector bundle over the base.  In the Zorn algebra presentation of $O$, the trace form is given by the matrix trace, see Remark \ref{rem:traceformzorn}.  Therefore, in suitable coordinates, the formula from Definition \ref{defn:zorn} guarantees that the restriction of the trace form to the locus of octonions of trace zero is the standard split form on $O'$.  Combining these observations, there is an embedding $\op{G}_2 \hookrightarrow \op{O}_7$ (even in characteristic $2$).  Once again, since $\op{G}_2$ has connected fibers, the image of the above group homomorphism must be contained in the connected component of the identity in $\op{O}_7$, which by \cite[Proposition C.3.5]{Conrad} coincides with $\op{SO}_7$.  Thus, combining Theorem \ref{thm:longrootembedding} with Remark \ref{rem:orthogonalembedding} and the observations just made, one obtains the following result.

\begin{prop}
\label{prop:cartesiang2sl3}
If $R$ is any commutative unital ring, then there is a Cartesian square of closed immersion group homomorphisms of the form:
\[
\xymatrix{
\op{SL}_3 \ar[r]\ar[d] & \op{SO}_6 \ar[d] \\
\op{G}_2  \ar[r] & \op{SO}_7;
}
\]
the top horizontal map is conjugate to the standard ``hyperbolic" map $\op{SL}_3 \to \op{SO}_6$.
\end{prop}

\begin{proof}
The existence of a commutative square is immediate from Theorem~\ref{thm:longrootembedding} combined with the discussion preceding the statement.  To establish the fact that $\op{SL}_3$ may be realized as the fiber product of $\op{G}_2$ and $\op{SO}_6$ over $\op{SO}_7$, we appeal to Remark~\ref{rem:orthogonalembedding}.  The final statement is immediate from the formula defining $\op{SL}_3 \to \op{G}_2$, which is given in Theorem~\ref{thm:longrootembedding} combined with Remark~\ref{rem:orthogonalembedding}.
\end{proof}

The groups $\op{G}_2$ and $\op{SL}_3$ are split simply-connected semi-simple group schemes.  On the other hand, the special orthogonal groups $\op{SO}_n$ are not simply-connected.  The simply-connected covering groups of $\op{SO}_n$ are the spin groups $\op{Spin}_n$; we refer the reader to \cite[\S 4.5]{GroupesClassiques} for a discussion of spin groups over an arbitrary base.  Over a field, one may lift the horizontal embeddings in Proposition \ref{prop:cartesiang2sl3} to embeddings of $\op{SL}_3$ in $\op{Spin}_6$ and $\op{G}_2 \to \op{Spin}_7$ \cite[\S 3.6]{SpringerVeldkamp}; in essence this follows from the fact that $\op{SL}_3$ and $\op{G}_2$ are simply-connected in the appropriate sense.  The same result also holds over an arbitrary base; indeed, the following result refines Proposition \ref{prop:cartesiang2sl3}.

\begin{prop}
\label{prop:spin7g2}
If $R$ is a commutative unital ring, then there is a Cartesian diagram of closed immersion group homomorphisms
\[
\xymatrix{
\op{SL}_3 \ar[r]\ar[d] & \op{Spin}_6 \ar[d] \\
\op{G}_2  \ar[r] & \op{Spin}_7.
}
\]
The map $\op{SL}_3 \to \op{Spin}_6$ is given by the hyperbolic map.
\end{prop}

\begin{proof}
Granted Proposition \ref{prop:cartesiang2sl3}, it remains to show that the closed immersion group homomorphisms $\op{SL}_3 \to \op{SO}_6$ and $\op{G}_2 \to \op{SO}_7$ lift along the universal covering morphisms $\op{Spin}_n \to \op{SO}_n$.   If we pull back all the closed immersion group homomorphisms in Proposition \ref{prop:cartesiang2sl3} along $\op{Spin}_7 \to \op{SO}_7$, then we obtain a corresponding Cartesian square. However, given a connected semi-simple group-scheme $\op{G}$ over an arbitrary base, any central extension of $\op{G}$ by $\mu_2$ is uniquely split and by \cite[Exercise 6.5.2(iii)]{ConradReductive} over an arbitrary base, any homomorphism of semi-simple group schemes uniquely lifts to simply-connected covering groups.  In particular, the stabilization map $\op{SO}_6 \to \op{SO}_7$ lifts uniquely to a morphism $\op{Spin}_6 \to \op{Spin}_7$.  Moreover, since $\op{SL}_3$ and $\op{G}_2$ are already simply-connected, it follows that we have unique lifts of the morphisms in the first sentence of this paragraph to morphisms $\op{SL}_3 \to \op{Spin}_6$ and $\op{G}_2 \to \op{Spin}_7$ and uniqueness guarantees that we obtain a Cartesian diagram as in the statement.
\end{proof}

\begin{rem}
\label{rem:leftmultinspin8}
If $O$ is the split octonion algebra over a base, then the locus of octonions of norm $1$ is isomorphic to the quadric $\op{Q}_7$.  By Remark \ref{rem:homomtoO8}, left multiplication by a unit norm octonion defines a morphism $\op{Q}_7 \longrightarrow \op{O}_8$.  Because $\op{Q}_7$ has connected fibers, this morphism automatically lifts through $\op{SO}_8$.  In fact, using the explicit formulas defining the spin group, left multiplication by a unit norm octonion actually defines a morphism $\op{Q}_7 \to \op{Spin}_8$.
\end{rem}

\subsubsection*{Compatibility of the hyperbolic map and stabilization}
Above, we studied a homomorphism $\op{SL}_3 \to \op{Spin}_6$.  There is a well-known exceptional isomorphism $\op{Spin}_6 \isomt \op{SL}_4$ induced by a ``half-spin" representation.  Details of this construction may be found in \cite[C.6.6]{ConradReductive}.  It suffices to construct this isomorphism over $\Z$.  If $\op{V}$ is a free $\Z$-module of rank $4$ equipped with the standard trivialization of its determinant, then there is an induced $\op{SL}(\op{V})$ action on $\Lambda^2 \op{V}$.  If we fix an identification $\det \op{V} \cong \Z$, then $\Lambda^2 \op{V}$ admits the structure of a quadratic space with the quadratic form defined by $q(w) = \frac{1}{2}(w \wedge w)$.  The resulting assignment factors through an isomorphism $\op{SL}_4/\mu_2 \isomt \op{SO}(q)$.  By \cite[Exercise 6.5.2(iii)]{ConradReductive}, this isomorphism lifts uniquely through a map $\op{SL}_4 \isomt \op{Spin}_6$, which is known to be an isomorphism.  There is a standard embedding $\op{SL}_3 \to \op{SL}_4$ corresponding to sending $X \in \op{SL}_3(R)$ to the block matrix $\op{diag}(X,1)$.  One obtains the following result by direct computation.

\begin{prop}
\label{prop:hyperbolicvsstabilization}
The composite map
\[
\op{SL}_3 \longrightarrow \op{Spin}_6 \isomto \op{SL}_4
\]
is conjugate to the standard embedding described above.
\end{prop}

\subsubsection*{The homogeneous spaces $\op{G}_2/\op{SL}_3$ and $\op{Spin}_7/\op{G}_2$}
The compact Lie group $G_2$ contains the subgroup $SU(3)$ and quotient space $G_2/SU(3)$ and Borel showed that the quotient is diffeomorphic to the $6$-sphere $S^6$.  Borel also observed that $Spin(7)$ acts transitively on $S^7$ and the stabilizer of a point is isomorphic to $G_2$ (we refer the reader to \cite[Th{\'e}or{\`e}me 3]{BorelSpheres} or \cite[Theorem 5.5]{Adams} for these facts).  We now provide algebro-geometric analogs of these results, valid over arbitrary base rings.

\begin{thm}
\label{thm:g2sl3andspin7}
Suppose $R$ is a commutative unital base ring and consider the Cartesian square of closed immersion homomorphisms:
\[
\xymatrix{
\op{SL}_3 \ar[r]\ar[d] & \op{Spin}_6 \ar[d] \\
\op{G}_2  \ar[r] & \op{Spin}_7.
}
\]
The following statements hold.
\begin{enumerate}[noitemsep,topsep=1pt]
\item The quotients of the horizontal and vertical homomorphisms exist as smooth affine $R$-schemes, i.e., the quotients $\op{G}_2/\op{SL}_3$, $\op{Spin}_7/\op{Spin}_6$, $\op{Spin}_6/\op{SL}_3$ and $\op{Spin}_7/\op{G}_2$ all exist as smooth affine $R$-schemes.
\item The induced map
\[
\op{G}_2/\op{SL}_3 \longrightarrow \op{Spin}_7/\op{Spin}_6
\]
is an isomorphism of $R$-schemes.  In particular, over any field, $\op{G}_2/\op{SL}_3 \cong \op{Q}_6$.
\item The induced map
\[
\op{SL}_4/\op{SL}_3 \cong \op{Spin}_6/\op{SL}_3 \longrightarrow \op{Spin}_7/\op{G}_2
\]
is an isomorphism of $R$-schemes.  In particular, over any field, $\op{Spin}_7/\op{G}_2 \cong \op{Q}_7$.
\item The torsors $\op{G}_2 \to \op{G}_2/\op{SL}_3$, $\op{Spin}_6 \to \op{Spin}_6/\op{SL}_3$, $\op{Spin}_7 \to \op{Spin}_7/\op{Spin}_6$ and $\op{Spin}_7 \to \op{Spin}_7/\op{G}_2$ are Nisnevich locally trivial.
\end{enumerate}
\end{thm}

\begin{rem}
Point (3) of the above result and Zariski local triviality of the torsor $\op{Spin}_7 \to \op{Spin}_7/\op{G}_2$  are also established by Alsaody and Gille \cite[Theorem 6.3]{AlsaodyGille}; we refer the reader there for an alternative treatment.
\end{rem}

Before establishing this result, we state one general result about quotients and homogeneous spaces; we thank Brian Conrad for suggesting this result and giving us its proof.

\begin{lem}
\label{lem:reductiontogeometricfibers}
Let $S$ be a scheme, and $\op{G}$ a smooth $S$-group scheme with connected fibers.  Assume $\op{H} \to \op{G}$ is a homomorphism of $S$-group schemes, and $\op{G}' \to \op{G}$ is a closed immersion homomorphism of $S$-group schemes with $\op{G}’$ fppf and $\op{H}$ smooth over S with connected fibers.  Assume the fiber product $\op{H}' := \op{G}' \times_{\op{G}} \op{H}$ is $S$-flat.  If the quotients $\op{H}/\op{H}'$ and $\op{G}/\op{G}'$ exist as schemes, then the induced morphism of schemes
\[
f:\op{H}/\op{H}' \longrightarrow \op{G}/\op{G}'
\]
is an open immersion if the fibral dimensions agree.  In particular, $f$ is an isomorphism if it induces isomorphisms of geometric fibers over $S$.
\end{lem}

\begin{proof}
By construction the formation of the quotients appearing in the statement commutes with base change on $S$.  Moreover, since smoothness may be checked locally in the fppf-topology, it follows that -- assuming the quotients exist -- they are necessarily smooth over $S$ since $\op{H}$ and $\op{G}$ are smooth over $S$ by assumption.  By appeal to  \cite[Corollaire 17.9.5]{EGAIV.4} to prove the map $f$ is an open immersion, it suffices to prove this fiberwise.  Therefore, we may and do reduce to the case where $S = \Spec k$, for $k$ a field.  Furthermore, since the property of being an open immersion is fppf-local on the base, we can assume without loss of generality that $k$ is algebraically closed.

The left action of $\op{H}$ on $\op{G}/\op{G}'$ has, by assumption, stabilizer scheme at the identity coset $\op{H}'$.  Therefore, $\op{H}/\op{H}'$ maps isomorphically onto the (locally closed) $\op{H}$-orbit through the identity coset.  In other words, the morphism $f$ is a locally closed immersion.  However, any locally closed immersion may be written as the composition of a closed immersion followed by an open immersion.  Since both the source and target of $f$ are smooth of the same dimension, it follows immediately that $f$ is an open immersion.

To conclude, observe that if $f$ is a surjective open immersion, then it is bijective and hence a homeomorphism.  In particular it is quasi-compact.  Therefore, by \cite[Lemma 34.22.1]{stacks-project} we may conclude that $f$ is an isomorphism.
\end{proof}

\begin{proof}[Proof of Theorem \ref{thm:g2sl3andspin7}]
Since all group schemes in the statement are defined over $\Z$, we establish the results in that case; the results in general follow from that case by pullback.\newline

\noindent {\bf Point 1}.  All group schemes appearing in the statement of Proposition \ref{prop:spin7g2} are split reductive and are thus pulled back from reductive group schemes over $\Spec \Z$.  Then, if $\op{H} \to \op{G}$ is a closed immersion of reductive $\Z$-group schemes, then the quotient $\op{G}/\op{H}$ exists as a smooth $\Z$-scheme, e.g., by appeal to \cite[Th\'eor\`eme 4.C]{Anantharaman}.  Furthermore, the morphism $\op{G} \to \op{G}/\op{H}$ is an $\op{H}$-torsor and therefore, smoothness of $\op{G}/\op{H}$ is inherited from smoothness of $\op{G}$ (as mentioned in the proof of Lemma \ref{lem:reductiontogeometricfibers}).  Finally, the fact that the quotient $\op{G}/\op{H}$ is affine is standard from geometric invariant theory; see, e.g., \cite[Corollary 9.7.5]{Alper}.\newline

\noindent {\bf Point 2}.  We want to show that the induced map $\op{G}_2/\op{SL}_3 \to \op{Spin}_7/\op{Spin}_6$ is an isomorphism of schemes.  by appeal to Lemma \ref{lem:reductiontogeometricfibers}, it suffices to check this upon passing to geometric fibers, and without loss of generality we may assume that we are working over an algebraically closed field and make arguments at the level of $k$-points.  Then, we appeal to a more invariant presentation of the quotient.  Consider the space $X$ parameterizing split complex subalgebras of the split octonion algebra $O$.  We already identified $\op{SL}_3$ as $\op{Aut}_{O/C}$ and therefore, the stabilizer of a point in $X$ is isomorphic to $\op{SL}_3$.    We claim that $\op{G}_2$ acts transitively on $X$.  To see this, we treat two cases, depending on the characteristic $p$ of the base field. \newline

\noindent {\bf Case $p \neq 2$}. If $k$ has characteristic unequal to $2$, then the choice of any element $a \in O$ with $\op{Tr}(a) = 0$ and $\op{N}(a) \neq 0$ determines a split complex sub-algebra.  By fixing any other element $b \in O$ with $\langle a,b \rangle = 0$, $\op{Tr}(b) = 0$ and $\op{N}(b) \neq 0$, we obtain a special $(\op{N}(a),\op{N}(b))$-pair in the sense of \cite[Definition 1.7.4]{SpringerVeldkamp}.  By \cite[Corollary 1.7.5]{SpringerVeldkamp}, $\op{G}_2$ acts transitively on the set of such pairs.  In particular, $\op{G}_2$ acts transitively on the set of $a \in O$ such that $\op{Tr}(a) = 0$ and $\op{N}(a) = 1$, which, unwinding the formulas for norm and trace, forms a quadric isomorphic to $\op{Q}_6$ (note: the resulting quadric is isomorphic to $\op{Q}_6$ assuming $2$ is invertible, which holds by the assumption on the characteristic).  Observe that the elements $1_O$ and $a$ span a $2$-dimensional non-singular subspace of $O$.  In particular, $a$ lies in the orthogonal complement to the line through $1_O$, and $\op{Spin}_7$ acts on this orthogonal complement with stabilizer of the line through $a$ isomorphic to $\op{Spin}_6$.\newline

\noindent {\bf Case $p = 2$}. If $k$ has characteristic equal to $2$, then the choice of any element $a \in O$ with $\op{Tr}(a) \neq 0$ determines a split complex subalgebra of $O$ by \cite[Lemma 1.6.1]{SpringerVeldkamp}.  Fix such an element $a$ with $\op{Tr}(a) = 1$ and $\op{N}(a) = 0$.  By fixing any other element $b \in O$ with $\langle a,b \rangle = 0$, $\op{Tr}(b) = 0$, $\op{N}(b) = 1$, we obtain a special $(0,1)$-pair in the sense of \cite[Definition 1.7.4]{SpringerVeldkamp}.  Again, by \cite[Corollary 1.7.5]{SpringerVeldkamp}, $\op{G}_2$ acts transitively on the set of such pairs.  In particular, $\op{G}_2$ acts transitively on the set of $a \in O$ such that $\op{Tr}(a) = 1$ and $\op{N}(a) = 0$, which, unwinding the formulas for the norm and trace, defines the quadric $\op{Q}_6$ (on the nose!).  In particular, we conclude that $\op{G}_2/\op{SL}_3$ is isomorphic to $\op{Q}_6$, even over fields having characteristic $2$.  Note that, since $k$ has characteristic $2$, $\langle 1_O,1_O \rangle = 0$.  On the other hand, $1_O$ and $a$ span a $2$-dimensional subspace of $O$ with a well-defined orthogonal complement that is a free module of rank $6$ equipped with a non-degenerate form.  The space spanned by this $6$-dimensional space together with $a$ is a $7$-dimensional quadratic subspace of $O$ (with respect to the norm form) and $\op{Spin}_7$ acts transitively on the set of such subspaces, with stabilizer of the line through $a$ isomorphic to $\op{Spin}_6$.\newline

\noindent {\bf Point 3.} We want to show that the induced map $\op{Spin}_6/\op{SL}_3 \to \op{Spin}_7/\op{G}_2$ is an isomorphism.  Again, by appeal to Lemma \ref{lem:reductiontogeometricfibers}, it suffices to check this upon passing to geometric fibers, and without loss of generality we may assume that we are working over an algebraically closed field and make arguments at the level of $k$-points.  To this end, recall that $\op{G}_2$ begins life as a subgroup of $\op{Spin}_8$, which acts on $O$, which we think of as an $8$-dimensional $k$-vector space equipped with the structure of a split non-degenerate quadratic space.  In particular, $\op{Spin}_8$ preserves a split quadric hypersurface in $O$: the hypersurface of norm $1$ octonions, which is isomorphic to $\op{Q}_7$.

We claim that both $\op{Spin}_6$ and $\op{Spin}_7 \subset \op{Spin}_8$ act transitively on $\op{Q}_7$.  It suffices to prove this for $\op{Spin}_6$.  In that case, the isomorphism $\op{Spin}_6 \isomt \op{SL}_4$ equips $\op{Q}_7$ with an $\op{SL}_4$-action.  Moreover, the inclusion $\op{SL}_4 \to \op{Spin}_8$ is conjugate under this isomorphism to the hyperbolic embedding $\op{SL}_4 \to \op{Spin}_8$.  By appeal to Proposition \ref{prop:hyperbolicvsstabilization}, we conclude that the inclusion $\op{SL}_3 \to \op{Spin}_6 \to \op{SL}_4$ coincides with the usual stabilization map.  In that case, $\op{SL}_4/\op{SL}_3$ is isomorphic to the quadric defined by setting the hyperbolic form equal to $1$, which is precisely the copy of $\op{Q}_7$ considered above.  Thus, it follows that $\op{Spin}_6$ acts transitively on $\op{Q}_7$ with stabilizer at a point isomorphic to $\op{SL}_3$.  Since $\op{Spin}_6 \subset \op{Spin}_7$ it also follows that the latter acts transitively on $\op{Q}_7$.

Identity $\op{Q}_7$ with the closed subscheme of $O$ consisting of elements of norm $1$.  The element $1_O$ defines a point on $\op{Q}_7$; write $\op{H}$ for the stabilizer scheme of this point.  Clearly $\op{G}_2 \subset \op{H}$ since any algebra automorphism of $O$ preserves the norm, the trace and the unit.  Since $\op{Q}_7$ is affine, it follows that the identity connected component $\op{H}^{\circ}$ of $\op{H}$ is a reductive group by Matsushima's theorem (see, e.g., \cite[Corollary 9.7.7]{Alper}).  Since $\op{Spin}_7$ has dimension $21$, this stabilizer scheme must have dimension $14$.  In other words, $\op{H}^{\circ} = \op{G}_2$.  One may check by direct computation that $\op{H} = \op{H}^{\circ}$ and thus $\op{H} = \op{G}_2$.\newline

\noindent {\bf Point 4}.  Finally, we need to establish Zariski-local triviality of all the relevant torsors.  This is clear for the torsors $\op{G}_2 \to \op{G}_2/\op{SL}_3$ and $\op{Spin}_6 \to \op{Spin}_6/\op{SL}_3$ since $\op{SL}_3$ is a special group.  Likewise, the Zariski-local triviality of the torsor $\op{Spin}_7 \to \op{Spin}_7/\op{Spin}_6$ is immediate upon making the exceptional identification $\op{Spin}_6 \cong \op{SL}_4$.

It remains to establish Zariski-local triviality of the $\op{G}_2$-torsor $\op{Spin}_7 \to \op{Spin}_7/\op{G}_2$.  For this, we appeal to Theorem \ref{thm:g2zarloctriv}, which shows that it suffices to know that the quadratic space underlying the octonion algebra is split over local rings.  Now, if $X$ is a scheme, then specifying a morphism $X \to \op{Spin}_7/\op{G}_2$ is equivalent to specifying a $\op{G}_2$-torsor on $X$ together with a trivialization of the associated $\op{Spin}_7$-torsor via the homomorphism $\op{G}_2 \to \op{Spin}_7$ (pull back the universal $\op{G}_2$-torsor $\op{Spin}_7 \to \op{Spin}_7/\op{G}_2$ along the given morphism).  Therefore, given a morphism $X \to \op{Spin}_7/\op{G}_2$, one also obtains a $\op{G}_2$-torsor on $X$ equipped with a trivialization of the associated $\op{O}_7$ and $\op{O}_8$-torsors coming from the homomorphisms $\op{Spin}_7 \to \op{O}_7 \to \op{O}_8$.  Now, if $X = \Spec R$ is an affine scheme, by Theorem \ref{thm:g2torsorsandoctonionalgebras}, specifying a $\op{G}_2$-torsor on $X$ is equivalent to specifying an octonion algebra on $R$ and the associated $\op{O}_8$-torsor is identified as the torsor of automorphisms of the quadratic space underlying the octonion algebra.  In particular, saying that this $\op{O}_8$-torsor is trivial is the same as saying that the norm form of the octonion algebra is split.  If $X = \Spec R$ is furthermore a local scheme, it follows from Theorem \ref{thm:g2zarloctriv} that an octonion algebra over a local ring with split associated norm form is necessarily itself split.  In particular, this holds for local rings of $\op{Spin}_7/\op{G}_2$ and the result follows.
\end{proof}

\subsubsection*{The homogeneous space $\op{Spin}_8/\op{G}_2$}
As observed above, the inclusion $\op{G}_2 \to \op{SO}_8$ factors through $\op{G}_2 \to \op{SO}_7$ and lifts through a map $\op{G}_2 \to \op{Spin}_7$.  Alternatively, the homomorphism $\op{G}_2 \to \op{SO}_8$ lifts uniquely through a morphism $\op{G}_2 \to \op{Spin}_8$ by \cite[Exercise 6.5.2]{Conrad}.  Either way, one may construct a closed immersion group homomorphism $\op{G}_2 \to \op{Spin}_8$ as the composite
\[
\op{G}_2 \longrightarrow \op{Spin}_7 \longrightarrow \op{Spin}_8,
\]
where the second map is the standard map, corresponding to the standard ``stabilization" embedding $\op{SO}_7 \to \op{SO}_8$ as block matrices of suitable form.  We now identify the quotient $\op{Spin}_8/\op{G}_2$ as a product of quadrics; this result essentially goes back to Jacobson \cite[p. 93]{JacobsonJordanII} and provides an analog of the diffeomorphism $Spin(8)/G_2 \cong S^7 \times S^{7}$ (see, e.g., \cite[Theorem 14.69]{Harvey}).

\begin{prop}
\label{prop:spin8g2}
If $R$ is any base ring, then there is an isomorphism of schemes of the form
\[
\op{Spin}_7/\op{G}_2 \times \op{Spin}_8/\op{Spin}_7 \isomto \op{Spin}_8/\op{G}_2
\]
where the map $\op{Spin}_7/\op{G}_2 \to \op{Spin}_8/\op{G}_2$ is induced by the standard embedding.  In particular, $\op{Spin}_8/\op{G}_2 \cong \op{Q}_7 \times \op{Q}_7$.  The quotient map $\op{Spin}_8 \to \op{Spin}_8/\op{G}_2$ is Zariski locally trivial.
\end{prop}

\begin{rem}
Alsaody and Gille \cite[Theorem 4.1]{AlsaodyGille} also establish an identification $\op{Spin}_8/\op{G}_2 \cong \op{Q}_7 \times \op{Q}_7$ and Zariski local triviality of the torsor $\op{Spin}_8 \to \op{Spin}_8/\op{G}_2$.
\end{rem}

\begin{proof}
Repeating the arguments at the beginning of the proof of Theorem \ref{thm:g2sl3andspin7}, we may first assume without loss of generality that $R = \Z$ and the result in general follows by base-change.  In that case, the relevant quotients exist by appeal, once again, to \cite[Th\'eor\`eme 4.C]{Anantharaman}.

There are induced maps of homogeneous spaces
\[
\op{Spin}_7/\op{G}_2 \longrightarrow \op{Spin}_8/\op{G}_2 \longrightarrow \op{Spin}_8/\op{Spin}_7.
\]
As the associated fiber bundle of the Zariski-locally trivial $\op{Spin}_7$-torsor $\op{Spin}_8 \to \op{Spin}_8/\op{Spin}_7$, the map
\[
\op{Spin}_8/\op{G}_2 \longrightarrow \op{Spin}_8/\op{Spin}_7
\]
is a Zariski locally trivial fiber bundle with fibers isomorphic to $\op{Spin}_7/\op{G}_2$. Since the map $\op{Spin}_7 \to \op{Spin}_7/\op{G}_2$ is Zariski-locally trivial by point (3) of Theorem \ref{thm:g2sl3andspin7}, the Zariski-local triviality statement follows.

In fact, we claim this Zariski fiber bundle is split.  To see this, we identify $\op{Spin}_8/\op{Spin}_7 \cong \op{Q}_7$ and identify $\op{Q}_7$ with the unit norm elements in a split octonion algebra $O$.  By the discussion of Remark \ref{rem:leftmultinspin8}, left multiplication by a unit norm octonion defines a morphism of schemes $\op{Q}_7 \to \op{Spin}_8$ and one checks that this morphism provides a section of the $\op{Spin}_7$-torsor $\op{Spin}_8 \to \op{Spin}_8/\op{Spin}_7$.  Therefore, the torsor $\op{Spin}_8 \to \op{Spin}_8/\op{Spin}_7$ is trivial, and this yields the splitting of the statement.
\end{proof}

\section{\texorpdfstring{$\A^1$}{A1}-homotopy sheaves of \texorpdfstring{$\op{B}_{\Nis}\op{G}_2$}{BG2}}
\label{sec:a1computations}

In this section, we study the $\aone$-homotopy types of Nisnevich classifying spaces of the group schemes that appeared in Section \ref{sec:octonionprelims}.  In particular, we compute low-degree $\aone$-homotopy sheaves of $\op{B}_{\Nis}\op{G}_2$ (see Proposition \ref{prop:lowdegree} and Theorem \ref{thm:pi3bg2}).  Our computation follows the computation of the low degree homotopy groups of the classifying space of the compact Lie group $G_2$.  Indeed, we use (i) an algebro-geometric analog of the fiber sequence
\[
S^6 \longrightarrow BSU(3) \longrightarrow BG_2,
\]
and (ii) analogs of results about homotopy groups of $BSU(3)$ and $S^6$.  Along the way, we also compute low degree homotopy sheaves of spin groups (Proposition \ref{prop:pi3bspin7} and Corollary \ref{cor:pi3bspinn}).  For the classical computation, we refer the reader to, e.g., \cite{Mimura}, which essentially uses the Cartan--Serre method of killing homotopy groups together with the Serre spectral sequence of the above fibrations. Over fields that may be embedded in $\cplx$, we may make explicit comparison between our results and the classical results.

\subsection{Some \texorpdfstring{$\aone$}{A1}-fiber sequences}
\label{ss:aonefibersequences}
We develop algebro-geometric analogs of the classical fiber sequences mentioned above (see Proposition \ref{prop:sl3g2fibration}).  In Proposition \ref{prop:square} we establish some compatibility results among our fiber sequences which will be important in the homotopy sheaf computations that follow.

\subsubsection*{The Zorn algebra construction as a map of classifying spaces}
Here, we analyze the construction of Remark \ref{rem:zornisafunctor} in terms of classifying spaces.  In Remark \ref{rem:zornisafunctor} we mentioned the functor $\op{Zorn}$ from the groupoid of oriented projective $R$-modules (and isomorphisms) to the groupoid of octonion algebras over $R$ (and isomorphisms).  The image of this functor landed in the subgroupoid of generically split octonion algebras, since oriented projective $R$-modules are always generically free.  One knows that the groupoid of oriented projective $R$-modules is naturally equivalent to the groupoid of $\op{SL}_3$-torsors over $R$ by the functor sending an oriented projective $R$-module to the functor of isomorphisms with the trivial oriented projective $R$-module.  On the other hand, we saw in Theorem \ref{thm:g2torsorsandoctonionalgebras} that the groupoid of octonion algebras is naturally equivalent to the groupoid of $\op{G}_2$-torsors over $R$.

Write $\op{Tors}_{\tau}G(R)$ for the groupoid of $G$-torsors over $\Spec R$ which are locally trivial in the topology $\tau$ and $\op{B}\op{Tors}_{\tau}G$ for the simplicial presheaf obtained by taking the nerve of this groupoid (see \cite[\S 2.2]{AffineRepresentabilityII} for some discussion about how to make this discussion precise).  There is a morphism of simplicial presheaves
\[
\op{B}_{\Nis}\op{G} \longrightarrow \op{B} \op{Tors}_{\Nis}\op{G};
\]
this morphism is functorial in the group $\op{G}$ and has the following properties: (1) it is a weak equivalence of simplicial presheaves, (2) it induces, after taking connected components, a bijection $\pi_0(\op{B}_{\Nis}\op{G})(-) \cong \op{H}^1_{\Nis}(-,\op{G})$ (see \cite[Lemma 2.2.2]{AffineRepresentabilityII}).

Since $\op{Zorn}$ is a functor, there is an induced morphism
\[
\op{Zorn}: \op{B} \op{Tors}_{\Nis}\op{SL}_3 \longrightarrow \op{B} \op{Tors}_{\Nis}\op{G}_2.
\]
On the other hand, the homomorphism of group schemes $\op{SL}_3 \to \op{G}_2$ described in the proof of Theorem \ref{thm:g2sl3andspin7} yields a morphism $\op{B}_{\Nis}\op{SL}_3 \to \op{B}_{\Nis}\op{G}_2$ (we avoid giving this a name momentarily).  Tracing through the constructions, the following diagram commutes:
\[
\xymatrix{
\op{B}_{\Nis}\op{SL}_3 \ar[r]\ar[d] & \op{B}_{\Nis}\op{G}_2 \ar[d] \\
\op{B} \op{Tors}_{\Nis}\op{SL}_3 \ar[r]^-{\op{Zorn}}& \op{B} \op{Tors}_{\Nis}\op{G}_2
}
\]
where the vertical morphisms are those discussed in the previous paragraph, and the top horizontal map is that induced by the morphism of group schemes.  For this reason, we will abuse notation and write
\[
\op{Zorn}: \op{B}_{\Nis}\op{SL}_3 \longrightarrow \op{B}_{\Nis}\op{G}_2
\]
for the morphism of classifying spaces induced by the group homomorphism $\op{SL}_3 \to \op{G}_2$.

\subsubsection*{Fiber sequences involving $\op{G}_2/\op{SL}_3$}
Now, we show how the homogeneous spaces described in Section \ref{ss:homogeneousspaces} fit into $\aone$-fiber sequences.  In essence, these results build on the theory developed in \cite{WendtTorsors}, though the following results are direct consequences of results in \cite{AffineRepresentabilityII}.

\begin{prop}
\label{prop:sl3g2fibration}
Assume $k$ is an infinite field.  There are $\aone$-fiber sequences of the form:
\[
\xymatrix @R=.1pc{
\op{Q}_6 \ar[r] & \op{B}\op{SL}_3 \ar[r]^-{\op{Zorn}}&{\op{B}_{\Nis}} \op{G}_2, \\
\op{Q}_7 \ar[r] & \op{B}_{\Nis}\op{G}_2 \ar[r] & \op{B}_{\Nis}\op{Spin}_7, \text{ and} \\
\op{Q}_7 \times \op{Q}_7 \ar[r] & \op{B}_{\Nis}\op{G}_2 \ar[r]& \op{B}_{\Nis}\op{Spin}_8.  \\
}
\]
\end{prop}

\begin{proof}
The three cases are established similarly.  Consider the quotients $\op{G}_2/\op{SL}_3 \cong \op{Q}_6$, $\op{Spin}_7/\op{G}_2 \cong \op{Q}_7$ and $\op{Spin}_8/\op{G}_2 \cong \op{Q}_7 \times \op{Q}_7$ (the relevant isomorphism are found in points 1 and 2 of Theorem \ref{thm:g2sl3andspin7} as well as Proposition \ref{prop:spin8g2}).  In the first two cases, the torsors $\op{G}_2 \to \op{SL}_3$ and $\op{Spin}_7 \to \op{Spin}_7/\op{G}_2$ are Zariski-locally trivial by point (3) of Theorem \ref{thm:g2sl3andspin7}.  Likewise, the torsor $\op{Spin}_8 \to \op{Spin}_8/\op{G}_2$ is Zariski-locally trivial by appeal to Proposition \ref{prop:spin8g2}.  As a consequence, in each case, the Zariski sheaf quotient coincides with the scheme quotient \cite[Lemma 2.3.1]{AffineRepresentabilityII}.

The fact that
\[
\xymatrix{
\op{Q}_6 \ar[r] & {\op{B}}\op{SL}_3 \ar[r]^-{\op{Zorn}}&{\op{B}_{\Nis}} \op{G}_2.
}
\]
is a simplicial fiber sequence is immediate from the definitions (the other two statements are established in a completely analogous fashion).  The result then follows immediately from \cite[Theorem 2.1.5]{AffineRepresentabilityII} where the additional hypotheses follow immediately from \cite[Theorems 3.3.1 and 3.3.6]{AffineRepresentabilityII} in view of the fact that $\op{G}_2$, $\op{Spin}_7$ and $\op{Spin}_8$ are split groups, and thus isotropic by definition.
\end{proof}

\begin{rem}
\label{rem:spinorbundles}
By construction, the morphism $\op{B}_{\Nis}\op{G}_2 \to \op{B}_{\Nis}\op{Spin}_7$ can be interpreted in terms of octonion algebras as that sending an octonion algebra $O$ to the spinor bundle associated to the quadratic space obtained by restricting the norm form to trace zero matrices; this is usually referred to as the associated spinor bundle of the octonion algebra.  A similar description holds for the morphism $\op{B}_{\Nis}\op{G}_2 \to \op{B}_{\Nis}\op{Spin}_8$.
\end{rem}

Repeating the proof of Proposition \ref{prop:sl3g2fibration} with evident changes, one establishes the following result.

\begin{prop}
\label{prop:spinstabilizationsequences}
Assume $k$ is an infinite field having characteristic unequal to $2$.  For any integer $n \geq 3$, there are $\aone$-fiber sequences of the form
\[
\op{Q}_n \longrightarrow \op{B}_{\Nis}\op{Spin}_n \longrightarrow \op{B}_{\Nis}\op{Spin}_{n+1}.
\]
\end{prop}

\begin{proof}
Appealing to Lemma~\ref{lem:reductiontogeometricfibers} and \cite[Lemma 3.1.7]{AffineRepresentabilityII} we deduce the existence of isomorphisms $\op{Spin}_{n+1}/\op{Spin}_n \isomt \op{SO}_{n+1}/\op{SO}_n$.  Then, we may repeat the proof of Proposition \ref{prop:sl3g2fibration} and appeal to \cite[Lemma 3.1.7]{AffineRepresentabilityII} instead of Theorem \ref{thm:g2sl3andspin7}.  We leave the details to the reader.
\end{proof}

\subsubsection*{Compatibility of $\aone$-fiber sequences}
We now analyze compatibilities amongst the fiber sequences appearing in Proposition \ref{prop:sl3g2fibration}.  The next result is a straightforward combination of the results of Theorem \ref{thm:g2sl3andspin7} and Proposition \ref{prop:sl3g2fibration}.

\begin{prop}
\label{prop:square}
If $k$ is an infinite field, then the following diagram is $\aone$-homotopy Cartesian:
\[
\xymatrix{
\op{B}_{\Nis}\op{SL}_3 \ar[d]^-{\op{Zorn}}\ar[r] & \op{B}_{\Nis}\op{Spin}_6 \ar[d] \\
 \op{B}_{\Nis}\op{G}_2 \ar[r] & \op{B}_{\Nis}\op{Spin}_7
}
\]
\end{prop}

\begin{proof}
Note that for $G$ any sheaf of groups, the space $\op{B}_{\Nis}G$ is $\aone$-connected; this is an immediate consequence of the unstable $\aone$-connectivity theorem \cite[\S 2 Corollary 3.22]{MV} since ${\op{B}}G$ is a reduced simplicial presheaf and the map ${\op{B}}G \to \op{B}_{\Nis}G$ is a Nisnevich-local weak equivalence.  As a consequence, the diagram in the statement is a diagram of $\aone$-connected spaces.

To check if the diagram in the statement is $\aone$-homotopy Cartesian, it suffices to show that the induced map of vertical (or horizontal) $\aone$-homotopy fibers is an $\aone$-weak equivalence (this can be checked on stalks, and thus follows from the corresponding classical fact for simplicial sets).  Since all spaces in the diagram are $\aone$-connected, it suffices to check that the induced map on $\aone$-homotopy fibers over the base-point is an $\aone$-weak equivalence.  Then, Proposition \ref{prop:sl3g2fibration} shows that the map of vertical homotopy fibers is a map $\op{Q}_6 \to \op{Q}_6$.  Unwinding the definitions, this map is precisely the map $\op{G}_2/\op{SL}_3 \to \op{Spin}_7/\op{Spin}_6$ shown to be an isomorphism in Theorem \ref{thm:g2sl3andspin7}.
\end{proof}

\subsection{Characteristic maps and strictly \texorpdfstring{$\aone$}{A1}-invariant sheaves}
In the above, we considered the morphisms $\op{G}_2 \to \op{G}_2/\op{SL}_3$ and $\op{Spin}_7 \to \op{Spin}_7/\op{Spin}_6$.  The former yields a rank $3$ vector bundle on $\op{G}_2/\op{SL}_3$, while under the exceptional isomorphism $\op{Spin}_6 \cong \op{SL}_4$, the latter yields a rank $4$ vector bundle on $\op{Spin}_7/\op{Spin}_6$.  In this section, we develop some techniques that will be useful in providing different ways of interpreting the classifying maps associated with these vector bundles.  In particular, we establish Propositions \ref{prop:mapsoutofspheresI} and \ref{prop:mapsoutofquadricsascohomology} which allow us to interpret maps out of motivic spheres in various useful ways.

\subsubsection*{The structure of some strictly $\aone$-invariant sheaves}
Before proceeding to the computations of $\aone$-homotopy sheaves of classifying spaces, we recall some facts about the strictly $\aone$-invariant sheaves that will arise in the statements and proofs.  We write $\K^{\op{Q}}_i$ for the Nisnevich sheaf associated with the Quillen K-theory presheaf on $\Sm_k$.  Using $\aone$-representability of algebraic K-theory, one observes that $\bpi_i^{\aone}({\op B}_{\Nis}\op{SL}_n)$ is isomorphic to $\K^{\op{Q}}_i$ for $2 \leq i \leq n-1$ (see, e.g., \cite[Theorem 3.2]{AsokFaselSpheres} for more details).

One also defines unramified Milnor K-theory sheaves $\K^{\op{M}}_i$.  For an irreducible $U \in \Sm_k$, the group of sections of $\K^{\op{M}}_i(U)$ coincides with unramified Milnor K-theory of $U$, i.e., the subgroup of $\K^{\op{M}}_i(k(U))$ consisting of elements unramified along every codimension $1$ point of $U$.  These sheaves arise naturally as the Nisnevich sheafification of the motivic cohomology presheaves $U \mapsto \op{H}^{n,n}(U,\Z)$ via the Nesterenko--Suslin--Totaro theorem.

There are canonical isomorphisms $\K^{\op{M}}_i \to \K^{\op{Q}}_i$ for $i \leq 2$.  Indeed, both $\K^{\op{M}}_0$ and $\K^{\op{Q}}_0$ coincide with the constant sheaf $\Z$, while $\K^{\op{M}}_1$ and $\K^{\op{Q}}_1$ coincide with the sheaf of units.  The identification $\K^{\op{M}}_2 \isomt \K^{\op{Q}}_2$ follows from Matsumoto's theorem identifying the Quillen $\op{K}_2$ of a field.  Using these canonical isomorphisms, one defines a ``natural" homomorphism
\[
\mu_i: \K^{\op{M}}_i \longrightarrow \K^{\op{Q}}_i.
\]
More precisely, there are multiplication maps $\K^{\op{M}}_i \times \K^{\op{M}}_j \to \K^{\op{M}}_{i+j}$ and $\K^{\op{Q}}_i \times \K^{\op{Q}}_j \to \K^{\op{Q}}_{i+j}$ and since the Steinberg relation is a degree $2$ relation, the product maps ${\K^{\op{M}}_1}^{\times n} \isomt {\K^{\op{Q}}_1}^{\times n} \to \K^{\op{Q}}_n$ factor through the morphism $\mu_n$; see \cite[Lemma 3.7]{AsokFaselSpheres} for more details.

\begin{lem}
\label{lem:k3indsheaf}
If we set $\K^{\op{ind}}_3 := \op{coker}(\mu_3)$, then $(\K^{\op{ind}}_3)_{-i} = 0$ for $i \geq 1$.
\end{lem}

\begin{proof}
This follows by combining the following facts: (i) $(\mu_n)_{-j} = \mu_{n-j}$ by \cite[Lemma 3.7]{AsokFaselSpheres} and (ii) the map $\mu_2$ is an isomorphism by Matsumoto's theorem.
\end{proof}

One also defines a morphism $\psi_n: \K^{\op{Q}}_n \to \K^{\op{M}}_n$ for $n \geq 2$.  To see this consider the quotient map $\op{SL}_n \to \op{SL}_n/\op{SL}_{n-1}$.  \cite[Theorem 6.40]{MField} identifies the first non-vanishing $\aone$-homotopy sheaf of a motivic sphere in terms of Milnor--Witt K-theory.  Using this identification, the quotient map induces a morphism
\[
\bpi_{n-1}^{\aone}(\op{SL}_n) \longrightarrow \bpi_{n-1}^{\aone}(\op{Q}_{2n-1}) = \K^{\MW}_n.
\]
Assuming we work over a field $k$ that has characteristic unequal to $2$, \cite[Lemma 3.5]{AsokFaselSpheres} allows us to conclude that this morphism factors through the stabilization morphism $\bpi_{n-1}^{\aone}(\op{SL}_n) \to \bpi_{n-1}^{\aone}(\op{SL}_{n+1}) = \K^{\op{Q}}_n$ to yield a morphism $\K^{\op{Q}}_n \to \K^{\MW}_n$.  Furthermore, there is a canonical quotient morphism $\K^{\MW}_n \to \K^{\op{M}}_n$ and the composite morphism $\K^{\op{Q}}_n \to \K^{\MW}_n \to \K^{\op{M}}_n$ is $\psi_n$.  We define $\mathbf{S}_n := \op{coker}(\psi_n)$.  The following result summarizes the properties of the sheaf $\mathbf{S}_n$ that we use in the sequel.

\begin{lem}[{\cite[Corollary 3.11]{AsokFaselSpheres}}]
\label{lem:structureofs4}
If $k$ is a field having characteristic unequal to $2$, then the epimorphism $\K^{\op{M}}_4 \to \mathbf{S}_4$ factors through an epimorphism $\K^{\op{M}}_4/6 \to \mathbf{S}_4$; this epimorphism becomes an isomorphism after $2$-fold contraction.
\end{lem}

\subsubsection*{Maps out of motivic spheres}
In analyzing the fiber sequences of Proposition \ref{prop:sl3g2fibration}, we will frequently need precise information about the maps on homotopy sheaves induced by a map from a motivic sphere to some other target.  We now give various different ways of capturing the required information: a map out of a motivic sphere can be realized (1) as a section of a suitable $\aone$-homotopy sheaf of the target, (2) in terms of sheaf cohomology on the source motivic sphere with coefficients in a suitable homotopy sheaf of the target, or, depending on the structure of the target, (3) in terms of geometric constructions (e.g., vector bundles) on the motivic sphere itself.

\begin{prop}
\label{prop:mapsoutofspheresI}
Assume $k$ is a field, and suppose $(\mathscr{X},x)$ is a pointed $\aone$-simply-connected space.  For any integer $m \geq 1$, there are identifications of the form
\[
\begin{split}
[\op{Q}_{2m},\mathscr{X}]_{\aone} \isomto &\bpi_{m,m}^{\aone}(\mathscr{X})(k) \isomto \bpi_m^{\aone}(\mathscr{X})_{-m}(k) \cong \hom(\K^{\MW}_m,\bpi_m^{\aone}(\mathscr{X})), \text{ and } \\
[\op{Q}_{2m+1},\mathscr{X}]_{\aone} \isomto &\bpi_{m,m+1}^{\aone}(\mathscr{X})(k) \isomto \bpi_m^{\aone}(\mathscr{X})_{-m-1}(k) \cong \hom(\K^{\MW}_{m+1},\bpi_m^{\aone}(\mathscr{X}))
\end{split}
\]
Moreover, the composite isomorphism is induced by applying $\bpi_m^{\aone}(-)$.
\end{prop}

\begin{proof}
On the left hand side it makes no difference whether we use free or pointed $\aone$-homotopy classes because the forgetful map from free to pointed homotopy classes is a bijection under the hypothesis that $\mathscr{X}$ is $\aone$-$1$-connected \cite[Lemma 2.1]{AsokFaselSpheres}.
Since $\op{Q}_{2m} \cong \Sigma^m \gm{\sma m}$ or $\op{Q}_{2m+1} \cong \Sigma^m \gm{\sma m+1}$ the first isomorphism is then immediate.

Next, for any integer $m \geq 1$, there is an isomorphism of functors $\hom(\K^{\MW}_m,-) \cong (-)_{-m}$ where the hom is taken in the category of strictly $\aone$-invariant sheaves (though this hom coincides with that taken in the category of Nisnevich sheaves of abelian groups).  With this observation, we may rewrite the above identifications as:
\[
\begin{split}
[\op{Q}_{2m},\mathscr{X}]_{\aone} &\isomto \hom(\K^{\MW}_m,\bpi_m^{\aone}(\mathscr{X})), \text{ and } \\
[\op{Q}_{2m+1},\mathscr{X}]_{\aone} &\isomto \hom(\K^{\MW}_{m+1},\bpi_m^{\aone}(\mathscr{X})).
\end{split}
\]
The final statement follows by analyzing the proof of Proposition \ref{prop:mapsoutofquadricsascohomology}, which will be of independent interest.
\end{proof}

\begin{prop}
\label{prop:mapsoutofquadricsascohomology}
Assume $k$ is a field.  For any pointed $\aone$-simply-connected space $(\mathscr{X},x)$, there are canonical isomorphisms of the form
\[
\begin{split}
[\op{Q}_{2n},\mathscr{X}]_{\aone} &\isomto \op{H}^n(\op{Q}_{2n},\bpi_n^{\aone}(\mathscr{X})), \text{ and }\\
[\op{Q}_{2n+1},\mathscr{X}]_{\aone} &\isomto \op{H}^n(\op{Q}_{2n+1},\bpi_n^{\aone}(\mathscr{X})).
\end{split}
\]
\end{prop}

\begin{proof}
As above, the assumption that $\mathscr{X}$ is $\aone$-simply-connected allows us to not have to distinguish between free and pointed $\aone$-homotopy classes of maps.  The argument proceeds by obstruction theory.  The arguments for $\op{Q}_{2n}$ and $\op{Q}_{2n+1}$ are essentially identical, so we give the former and indicate the changes necessary for the latter.  Because $\op{Q}_{2n}$ is $\aone$-$(n-1)$-connected, any pointed morphism $\op{Q}_{2n} \to \op{\mathscr{X}}$ factors uniquely through the $\aone$-$(n-1)$-connective cover of $\mathscr{X}$, for which we shall write $\mathscr{X}\langle n \rangle$.  The first stage of the $\aone$-Postnikov tower of $\mathscr{X}\langle n \rangle$ is $K(\bpi_n^{\aone}(\mathscr{X}),n)$.  Therefore, composing these two maps we get the map in the statement.  To see that this map is an isomorphism, we use obstruction theory.  We want to show that any map from $\op{Q}_{2n}$ to the $n$-th stage of the $\aone$-Postnikov tower of $\mathscr{X}$ extends uniquely to $\mathscr{X}\langle n \rangle$.

To this end, suppose ${\mathbf M}$ is any strictly $\aone$-invariant sheaf.  Since $\op{Q}_{2n} \cong \Sigma \op{Q}_{2n-1}$, by the suspension isomorphism and appeal to \cite[Lemma 4.5]{AsokFaselSpheres} we know that
\[
\op{H}^i(\op{Q}_{2n},\mathbf{M}) = \begin{cases} {\mathbf M}(k) & \text{ if } i = 0 \\ {\mathbf M}_{-n}(k) & \text{ if } i = n\\ 0 & \text{ otherwise.}
\end{cases}
\]
Now, the groups housing obstructions to lifting and parameterizing lifts are both of the form $\op{H}^i(\op{Q}_{2n},\mathbf{M})$ for $i \geq n+1$ and therefore vanish by the above result.

The only change that one need make for the case $\op{Q}_{2n+1}$ is that one may appeal directly to \cite[Lemma 4.5]{AsokDoranFasel}.
\end{proof}

\subsection{Realization and characteristic maps}
In this section, we analyze the behavior of the maps $\op{Q}_6 \to \op{B}_{\Nis}\op{SL}_3$ and $\op{Q}_6 \to \op{B}_{\Nis}\op{Spin}_6$ from Proposition \ref{prop:sl3g2fibration} under complex and \'etale realization.  The key results are Lemmas \ref{lem:q6chernclass} and \ref{lem:nontrivialq7toq6}.

\subsubsection*{Complex realization}
For a discussion of real and complex realization we refer the reader to \cite[\S 3.3]{MV} or to \cite{DuggerIsaksen}.  In the first case, suppose $k$ is a field equipped with an embedding $\iota: k \to \cplx$.  Given a smooth $k$-scheme $X$, we may consider the assignment $X \mapsto X(\cplx)$ sending $X$ to $X(\cplx)$ with its usual structure of a complex analytic space.  This assignment extends to a functor $\hop{k} \to \mathscr{H}_{\bullet}$ and this functor preserves homotopy colimits.

\begin{lem}
\label{lem:complexrealizationofspheres}
Assume $k \subset \cplx$.  For any $n \geq 0$, the complex realization of $\op{Q}_n$ is homotopy equivalent  to $S^n$.
\end{lem}

\begin{lem}
\label{lem:complexrealizationofclassifyingspaces}
Assume $k \subset \cplx$.  If $\op{G}$ is a split reductive $k$-group scheme, and $G$ is a maximal compact subgroup of $\op{G}(\cplx)$, then the complex realization of $\op{B}_{\Nis}\op{G}$ is homotopy equivalent to $BG$.
\end{lem}

\subsubsection*{\'Etale realization}
For a discussion of \'etale realization we refer the reader to \cite{Isaksen}.  If $\ell$ is prime, then we may define the $\ell$-complete \'etale realization functor on the category of schemes.  Given a scheme $X$, its \'etale realization is an $\ell$-complete pro-simplicial set that we will denote by $\op{Et}(X)$.  This has the property that a morphism of schemes $f\colon X\to Y$ induces a weak equivalence $\op{Et}(X)\to\op{Et}(Y)$ if and only if $f^*\colon \op{H}^*_{\et}(Y;\Z/\ell)\to \op{H}^*_{\et}(X;\Z/\ell)$ is an isomorphism.  By \cite{Isaksen}, if $k$ is a field and and $\ell$ is different from the characteristic of $k$, the assignment $X \mapsto \op{Et}(X)$ on smooth $k$-schemes extends to a functor on $\hop{k}$ that we will also denote by $\op{Et}$.  If $k$ is furthermore separably closed, it follows from the K\"unneth isomorphism in \'etale cohomology with $\Z/\ell\Z$-coefficients that the functor $\op{Et}$ preserves finite products (and smash products of pointed spaces).

We now analyze the behavior of $\op{Q}_n$ and Nisnevich classifying spaces of split reductive groups under \'etale realization.  To this end, assume furthermore that $k$ is a separably closed field, and let $R$ be the ring of Witt vectors of $k$.  Choose an algebraically closed field $K$ and embeddings $R \hookrightarrow K$ and $\cplx \hookrightarrow K$.  These morphisms yield maps of the form:
\[
\op{G}_k \longrightarrow \op{G}_R \longleftarrow \op{G}_K \longrightarrow \op{G}_{\cplx}.
\]
We may use these maps to compare the \'etale realization of $\op{G}$ and the corresponding complex points.

\begin{lem}
\label{lem:etalerealizationofquadrics}
Assume $k$ is a separably closed field having characteristic $p$ and fix $\ell \neq p$.  For any integer $n \geq 0$, $\op{Et}(\op{Q}_n) \cong (S^n)^{\wedge}_{\ell}$.
\end{lem}

\begin{proof}
There are weak equivalences $\op{Q}_n \cong \op{S}^i \wedge \gm{\sma j}$ for integers $i,j \geq 0$, with $i + j = n$.  One checks $\op{Et}(\gm{}) \cong {\gm{}(\cplx)}^{\wedge}_{\ell}\cong(S^1)^{\wedge}_{\ell}$ using the comparison maps described above.  The result then follows because \'etale realization preserves smash products of pointed spaces.
\end{proof}

\begin{lem}
\label{lem:etalerealizationofclassifyingspaces}
Assume $\op{G}$ is a split reductive group over $\Z$.  Suppose $k$ is a separably closed field having characteristic $p$ and fix $\ell \neq p$.  If $G$ is a maximal compact subgroup of $\op{G}(\cplx)$ then there are canonical equivalences
\[
\op{Et}(\op{B}_{\Nis}\op{G}_k) \cong (\op{B}\op{Et}(\op{G}_k))^{\wedge}_{\ell} \cong (BG)^{\wedge}_{\ell}.
\]
\end{lem}

\begin{proof}
By \cite[Theorem 1]{FriedlanderParshall}, the maps
\[
\op{G}_k \longrightarrow \op{G}_R \longleftarrow \op{G}_K \longrightarrow \op{G}_{\cplx}.
\]
induce isomorphism on \'etale cohomology with $\Z/\ell\Z$ coefficients and thus also after \'etale realization.  In particular, there are canonical weak equivalences
\[
\op{Et}(\op{G}_k) \cong \op{Et}(\op{G}_{\cplx}) \cong \op{G}(\cplx)^{\wedge}_{\ell},
\]
naturally in $G$.  Since the map $G \to \op{G}(\cplx)$ is a weak equivalence, it follows that $G^{\wedge}_{\ell} \cong \op{G}(\cplx)^{\wedge}_{\ell}$ also.  Using the fact that \'etale realization preserves finite products and homotopy colimits we conclude that
\[
\op{Et}(\op{B}_{\Nis}\op{G}_k) \cong (\op{B}\op{Et}(\op{G}_k))^{\wedge}_{\ell}
\]
and the final statement follows immediately from Lemma \ref{lem:complexrealizationofclassifyingspaces}.
\end{proof}

\subsubsection*{An oriented rank $4$ bundles on $\op{Q}_6$}
Proposition \ref{prop:square} gives a homotopy commutative triangle of the form
\[
\xymatrix{
\op{Q}_6 \ar[r]\ar[dr] & \op{B}_{\Nis}\op{SL}_3 \ar[d]\\
 & \op{B}_{\Nis}\op{Spin}_6 \cong \op{B}_{\Nis}\op{SL}_4,
}
\]
where the horizontal map is the classifying map of the $\op{SL}_3$-torsor $\op{G}_2 \to \op{G}_2/\op{SL}_3$ and the diagonal map is the classifying map of the torsor $\op{Spin}_7 \to \op{Spin}_7/\op{Spin}_6$.  Since $\op{Q}_6 \cong \op{S}^3 \wedge \gm{\sma 3}$, the horizontal and diagonal maps yield elements of $\bpi_{3,3}^{\aone}(\op{B}_{\Nis}\op{SL}_3)$ and $\bpi_{3,3}^{\aone}(\op{B}_{\Nis}\op{SL}_4)$.  We now identify the class of the diagonal map.

\begin{lem}
\label{lem:diagonalmapgenerator}
Assume $k$ is a field.  There is an isomorphism $\bpi_{3,3}^{\aone}(\op{B}_{\Nis}\op{SL}_4) \cong \Z$ and the class of $\op{Q}_6 \to \op{B}_{\Nis}\op{SL}_4$ described above is a generator.
\end{lem}

\begin{proof}
By, e.g., \cite[Theorem 3.2]{AsokFaselSpheres}, we know that $\bpi_3^{\aone}(\op{B}_{\Nis}\op{SL}_4) = \K^{\op{Q}}_3$ and thus that $\bpi_{3,3}^{\aone}(\op{B}_{\Nis}\op{SL}_4) \cong (\K^{\op{Q}}_3)_{-3} = \Z$ (see, e.g., \cite[Lemma 2.7]{AsokFaselSpheres}).  Now, the diagonal map induces a morphism of sheaves
\[
\bpi_{3,3}^{\aone}(\op{Q}_6) \longrightarrow \bpi_{3,3}^{\aone}(\op{B}_{\Nis}\op{SL}_4).
\]
Since the sheaf on the right is the constant sheaf $\Z$, and the morphism is compatible with field extensions, we may assume without loss of generality that the base field $k$ is separably closed.  Thus, we may assume without loss of generality that $k$ is the algebraic closure of $\Q$ in $\cplx$ or the algebraic closure of a finite field.

First, we treat the characteristic $0$ case.  Then, by appeal to Lemmas~\ref{lem:complexrealizationofspheres} and \ref{lem:complexrealizationofclassifyingspaces} there is a commutative diagram of the form:
\[
\xymatrix{
\Z = \bpi_{3,3}^{\aone}(\op{Q}_6)(\cplx) \ar[r]\ar[d] & \bpi_{3,3}^{\aone}(\op{B}_{\Nis}\op{SL}_4) = \Z \ar[d] \\
\Z = \pi_6(S^6) \ar[r]& \pi_6(BSU(4)) = \Z;
}
\]
where the vertical maps are isomorphisms and the bottom horizontal map is an isomorphism.  It follows that the top horizontal map is an isomorphism as well.

Now, assume $k$ has non-zero characteristic and that $\ell$ is a prime different from the characteristic of $k$.  In that case, after evaluation at $k$ and appeal to Lemmas~\ref{lem:etalerealizationofquadrics} and \ref{lem:etalerealizationofclassifyingspaces} we conclude that there is a commutative square of the form
\[
\xymatrix{
\Z = \bpi_{3,3}^{\aone}(\op{Q}_6)(k) \ar[r]\ar[d] & \bpi_{3,3}^{\aone}(\op{B}_{\Nis}\op{SL}_4) = \Z \ar[d] \\
\Z_{\ell} = \pi_6(S^6)^{\wedge}_{\ell} \ar[r]& \pi_6(BSU(4))^{\wedge}_{\ell} = \Z_{\ell};
}
\]
both vertical maps are injective, while the bottom horizontal map is that induced by the topological fiber sequence $S^6 \to BSU(4) \to BSpin(7)$ and thus is an isomorphism.  It follows that the top horizontal map must send a generator to a generator, which is what we wanted to show.
\end{proof}

\begin{rem}[Canonical reduction of Suslin matrices]
Vector bundles on $\op{Q}_{2n}$ were studied in detail in \cite{AsokDoranFasel}.  In particular, there we described an algebro-geometric variant of the clutching construction providing a bijection between rank $n$ oriented vector bundles on $\op{Q}_{2n}$ and $\aone$-homotopy classes of maps $\op{Q}_{2n-1} \to \op{SL}_n$.  One knows that $\widetilde{K}_0(\op{Q}_{2n}) = \Z$, and \cite[Theorem 4.3.4]{AsokDoranFasel} shows that a generator of this group can be defined in terms of Suslin matrices see \cite[\S 5]{SuslinStablyFree}.  In particular, Suslin constructed a map $\op{Q}_5 \to \op{SL}_4$ and this morphism corresponds to a generator of $\widetilde{K}_0(\op{Q}_{6})$.

Suslin shows that this matrix is $\aone$-weakly equivalent to one factoring through the inclusion $\op{SL}_3 \hookrightarrow \op{SL}_4$, and thus, the rank $4$ vector bundle splits off a free rank $1$ summand, however this factorization is {\em not} unique in general.  Indeed, $\bpi_{3,3}^{\aone}(\op{B}_{\Nis}\op{SL}_3) = \Z \oplus k^{\times}/k^{\times 6}$ and thus over non-algebraically closed fields, there can be many inequivalent reductions.  The discussion before Lemma~\ref{lem:diagonalmapgenerator} gives, however, a {\em canonical} such lift.  The following result summarizes these observations.
\end{rem}

\begin{lem}
\label{lem:q6chernclass}
Work over an infinite base field $k$. Write $\mathscr{E}$ for the oriented rank $4$ vector bundle on $\op{Q}_6$ obtained as the associated vector bundle to the $\op{SL}_4$-torsor $\op{Spin}_7 \to \op{Spin}_7/\op{SL}_4 \cong \op{Q}_6$.
\begin{enumerate}[noitemsep,topsep=1pt]
\item The class $[\mathscr{E}] \in \tilde{K}_0(\op{Q}_6) = \Z$ is a generator.
\item The bundle $\mathscr{E}$ splits as $\mathscr{E}' \oplus \mathscr{O}_{\op{Q}_6}$; the rank $3$ bundle $\mathscr{E}'$ is oriented and is the associated vector bundle to $\op{G}_2 \to \op{G}_2/\op{SL}_3$.
\end{enumerate}
\end{lem}

\subsubsection*{An oriented rank $3$ bundle on $\op{Q}_6$}
We now further analyze the map $\op{Q}_6 \to \op{B}_{\Nis}\op{SL}_3$ classifying the $\op{SL}_3$-torsor $\op{G}_2 \to \op{G}_2/\op{SL}_3$.  As above, there is an induced morphism
\[
\bpi_{3,4}^{\aone}(\op{Q}_6) \longrightarrow \bpi_{3,4}^{\aone}(\op{B}_{\Nis}\op{SL}_3).
\]
The group on the right is determined in \cite[Proposition 3.15]{AsokFaselSpheres} and is isomorphic to the constant sheaf $\Z/6\Z$ if we work over a base field having characteristic unequal to $2$.  Similarly, the sheaf on the right is isomorphic to $(\K^{\MW}_3)_{-4} \cong \mathbf{W}$.

\begin{lem}
\label{lem:nontrivialq7toq6}
Assume $k$ is a field having characteristic unequal to $2$.  The map
\[
\mathbf{W} := \bpi_{3,4}^{\aone}(\op{Q}_6) \longrightarrow \bpi_{3,4}^{\aone}(\op{B}_{\Nis}\op{SL}_3) \cong \Z/6\Z
\]
factors as the rank map $\mathbf{W} \to \Z/2\Z$ followed by the inclusion $\Z/2\Z \subset \Z/6\Z$.
\end{lem}

\begin{proof}
Since the sheaf $\Z/6\Z$ is a constant sheaf, and since the homomorphism of the statement is compatible with field extensions, we may assume without loss of generality that $k$ is separably closed.  In that case, the extension of scalars map $\mathbf{W}(k) \to \mathbf{W}(k^{\op{sep}}) = \Z/2\Z$ coincides with the rank map.  Therefore, it suffices to prove the second statement.

Assume first that $k \subset \cplx$.  In that case, by appeal to Lemmas ~\ref{lem:complexrealizationofspheres} and \ref{lem:complexrealizationofclassifyingspaces}, complex realization yields a diagram of the form:
\[
\xymatrix{
\bpi_{3,4}^{\aone}(\op{Q}_6)(k) \ar[r]\ar[d] & \bpi_{3,4}^{\aone}(\op{B}_{\Nis}\op{SL}_3)(k) \ar[d] \\
\pi_7(S^6) \ar[r] & \pi_7(BSU(3)).
}
\]
The map on the left sends the generator $\eta$ to the topological Hopf map $\eta$ and thus is an isomorphism.  The right vertical map is an isomorphism by \cite[Theorem 5.5]{AsokFaselSpheres}.  The bottom horizontal map is known classically to be the inclusion $\Z/2\Z \subset \Z/6\Z$ and the result follows.

If $k$ is now a separably closed field having positive characteristic unequal to $2$, then by appeal to Lemmas~\ref{lem:etalerealizationofquadrics} and \ref{lem:etalerealizationofclassifyingspaces}, we conclude there is a diagram as above with the bottom row replaced by its $2$-completion.  Once again, the leftmost vertical map is an isomorphism, and we deduce that the map
\[
\bpi_{3,4}^{\aone}(\op{Q}_6)(k) \longrightarrow \bpi_{3,4}^{\aone}(\op{B}_{\Nis}\op{SL}_3)(k)
\]
 is injective, as desired.
\end{proof}

\subsection{Some low degree \texorpdfstring{$\aone$}{A1}-homotopy sheaves of \texorpdfstring{${\op{B}_{\Nis}}\op{Spin}_n$}{BSpinn} and \texorpdfstring{${\op{B}_{\Nis}}\op{G}_2$}{BG2}}
In this section, we compute low degree $\aone$-homotopy sheaves of ${\op{B}_{\Nis}}\op{G}_2$ and deduce some consequences.  To do this, we appeal to Proposition \ref{prop:sl3g2fibration} and make use of connectivity estimates for the smooth affine schemes $\op{Q}_6$ and $\op{Q}_7$ appearing therein.

\subsubsection*{The first non-vanishing $\aone$-homotopy sheaf}
We begin by identifying the first non-vanishing $\aone$-homotopy sheaf of $\op{B}_{\Nis}\op{G}$ for $\op{G} = \op{SL}_3, \op{Spin}_6, \op{Spin}_7$ and $\op{G}_2$.  The following result is already essentially in \cite{WendtChevRep}, but we state and prove it here for convenience.
\begin{prop}
\label{prop:lowdegree}
Assume $k$ is an infinite field.  For $\op{G} = \op{SL}_3, \op{Spin}_6, \op{Spin}_7$ and $\op{G}_2$, $\op{B}_{\Nis}G$ is $\aone$-$1$-connected, and there are isomorphisms of the form
\[
\bpi_2^{\aone}(\op{B}_{\Nis}\op{G}) \isomto \K^{\op{M}}_2.
\]
\end{prop}

\begin{proof}
The result for $G = \op{SL}_3$ is due to F. Morel (it follows from \cite[Theorem 7.20]{MField}).  Now, we simply appeal to Proposition \ref{prop:sl3g2fibration} and the fact that both $\op{Q}_6$ and $\op{Q}_7$ are $\aone$-$2$-connected.  Indeed, as both $\op{Q}_6$ and $\op{Q}_7$ $3$-fold simplicial suspensions, they are necessarily simplicially $2$-connected and thus by appeal to \cite[Theorem 6.38]{MField} are $\aone$-$2$-connected as well; see \cite[Theorem 2.1.12]{AsokWickelgrenWilliams} for a detailed proof of the latter assertion.  We conclude that the maps $\op{B}_{\Nis}\op{SL}_3 \to \op{B}_{\Nis}\op{G}_2$, $\op{B}_{\Nis}\op{SL}_3 \to \op{B}_{\Nis}\op{Spin}_6$ and $\op{B}_{\Nis}\op{Spin}_6 \to \op{B}_{\Nis}\op{Spin}_7$ are $\aone$-$2$-equivalences, and the result follows.
\end{proof}

\subsubsection*{The sheaf $\bpi_3^{\aone}(\op{B}_{\Nis}\op{Spin}_n)$ }
In order to compute the next non-vanishing $\aone$-homotopy sheaf of $\op{B}_{\Nis}\op{G}_2$, we will need to analyze the maps on homotopy groups induced by the maps in Proposition \ref{prop:square}.  We make a systematic analysis of $\bpi_3^{\aone}(\op{B}_{\Nis}\op{Spin}_n)$.  Because of the exceptional isomorphisms $\op{Spin}_3 \cong \op{SL}_2$, $\op{Spin}_4 \cong \op{SL}_2 \times \op{SL}_2$, $\op{Spin}_5 \cong \op{Sp}_4$ and $\op{Spin}_5 \cong \op{SL}_4$, the third $\aone$-homotopy sheaf of $\op{B}_{\Nis}\op{Spin}_n$ has already been computed for $3 \leq n \leq 6$.  Indeed, $\bpi_3^{\aone}(\op{B}_{\Nis}\op{SL}_2)$ is computed in \cite[Theorem 3.3]{AsokFasel} and $\bpi_3^{\aone}(\op{B}_{\Nis}\op{Spin}_4) \cong \bpi_3^{\aone}(\op{B}_{\Nis}\op{SL}_2)^{\oplus 2}$.  Also, $\bpi_3^{\aone}(\op{B}_{\Nis}\op{Sp}_4) \cong \K^{\op{Sp}}_3$ by \cite[Theorem 2.6]{AsokFasel}, while $\bpi_3^{\aone}(\op{B}_{\Nis}\op{SL}_4) \cong \K^{\op{Q}}_3$, e.g., by \cite[Theorem 6.8(i)]{WendtChevRep} or \cite[Theorem 3.2]{AsokFaselSpheres}.

Therefore, we begin by analyzing $\bpi_3^{\aone}(\op{B}_{\Nis}\op{Spin}_7)$. The right vertical map $\op{B}_{\Nis}\op{Spin}_6 \to \op{B}_{\Nis}\op{Spin}_7$ in the diagram appearing in Proposition \ref{prop:square} is precisely the stabilization map in the Spin groups so we investigate the map on $\aone$-homotopy sheaves induced by that identification.

\begin{prop}
\label{prop:pi3bspin7}
If $k$ is an infinite base field, then there is an isomorphism $\K_3^{\op{ind}} \isomt \bpi_3^{\aone}(\op{B}_{\Nis}\op{Spin}_7)$.
\end{prop}

\begin{proof}
Consider the $\aone$-fiber sequence
\[
\xymatrix{
\op{Q}_6 \ar[r] & \op{B}_{\Nis}\op{Spin}_6 \ar[r] & \op{B}_{\Nis}\op{Spin}_7.
}
\]
Taking homotopy sheaves we obtain an associated long exact sequence; using the computations mentioned above, this exact sequence takes the form:
\[
\K^{\MW}_3 \longrightarrow \K^{\op{Q}}_3 \longrightarrow \bpi_3^{\aone}(\op{B}_{\Nis}\op{Spin}_7) \longrightarrow \bpi_2^{\aone}(\op{Q}_6) = 0.
\]
In particular, we conclude $\bpi_3^{\aone}(\op{B}_{\Nis}\op{Spin}_7) \cong \op{coker}(\K^{\MW}_3 \to \K^{\op{Q}}_3)$, and our goal is to identify this cokernel.  To this end, we make a general analysis of morphisms $\K^{\MW}_3 \to \K^{\op{Q}}_3$.  Observe that $\hom(\K^{\MW}_3,\K^{\op{Q}}_3) \cong (\K^{\op{Q}}_3)_{-3} = \Z$, e.g., by \cite[Lemma 5.1.3.1]{AsokWickelgrenWilliams}.  We give a geometric identification of the class of the morphism $\K^{\MW}_3 \to \K^{\op{Q}}_3$ in the statement.\newline

\noindent {\bf Step 1.} We begin by showing that any map $\op{Q}_6 \to \op{B}_{\Nis}\op{SL}_4$ is uniquely determined by a class in $\tilde{K}_0(\op{Q}_6)$; we do this in two stages.  \newline

\noindent {\bf Step 1a.} First we claim that any map $\op{Q}_6 \to \op{B}_{\Nis}\op{SL}_4$ is uniquely determined by its extension $\op{Q}_6 \to \op{B}_{\Nis}\op{SL}_{\infty}$.  Indeed, the $\aone$-homotopy fiber of the map $\op{B}_{\Nis}\op{SL}_n \to \op{B}_{\Nis}\op{SL}_{n+1}$ is ${\mathbb A}^{n+1} \setminus 0$, which is $\aone$-$(n-1)$-connected. When $n \geq 4$, the $\aone$-homotopy fiber is at least $\aone$-$3$-connected.  Therefore, any map $\op{Q}_6 \to {\mathbb A}^n \setminus 0$ is null $\aone$-homotopic.  It follows that the map
\[
[\op{Q}_6,\op{B}_{\Nis}\op{SL}_n]_{\aone} \longrightarrow [\op{Q}_6,\op{B}_{\Nis}\op{SL}_{n+1}]_{\aone}
\]
induced by the stabilization is bijective for every $n \geq 4$.  It follows immediately that there is an induced bijection $[\op{Q}_6,\op{B}_{\Nis}\op{SL}_4]_{\aone} \isomt [\op{Q}_6,\op{B}_{\Nis}\op{SL}]_{\aone}$ as claimed.  \newline

\noindent {\bf Step 1b.} Next, we claim that the class in $[\op{Q}_6,\op{B}_{\Nis}\op{SL}]_{\aone}$ determined by our given class in $[\op{Q}_6,\op{B}_{\Nis}\op{SL}_{4}]_{\aone}$ is equivalent to a class in $\tilde{K}_0(\op{Q}_6)$.  To this end, observe that there is a fiber sequence of the form
\[
\op{B}_{\Nis}\op{SL} \longrightarrow \op{B}_{\Nis}\op{GL} \stackrel{\det}{\longrightarrow} \op{B}_{\Nis}\gm{}
\]
and that this fiber sequence is split.  In particular,
\[
\tilde{K}_0(\op{Q}_6) = [\op{Q}_6,\op{B}_{\Nis}\op{GL}]_{\aone} \cong [\op{Q}_6,\op{B}_{\Nis}\op{SL}]_{\aone} \oplus [\op{Q}_6,\op{B}_{\Nis}\gm{}]_{\aone}.
\]
Since $[\op{Q}_6,\op{B}_{\Nis}\gm{}]_{\aone} = \op{Pic}(\op{Q}_6) = 0$, we obtain the required identification.  However, $\tilde{K}_0(\op{Q}_6) = \Z$ generated by the class of an explicit rank $3$ vector bundle on $\op{Q}_6$.  \newline

\noindent {\bf Step 2.} Granted the conclusion of Step 1, observe that there is a diagram of the form
\[
\tilde{K}_0(\op{Q}_6) \stackrel{\sim}{\longrightarrow} [\op{Q}_6,\op{B}_{\Nis}\op{SL}_4]_{\aone} \isomto \hom(\K^{\MW}_3,\K^{\op{Q}}_3),
\]
where the arrow on the left is the inverse of that from Step 1, while the arrow on the right is given to us by the conclusion of Proposition \ref{prop:mapsoutofspheresI}. \newline

\noindent {\bf Step 3.}  We claim that any morphism $\K^{\MW}_3 \to \K^{\op{Q}}_3$ factors uniquely through a morphism $\K^{\op{M}}_3 \to \K^{\op{Q}}_3$.  More precisely, in conjunction with the conclusion of Step 2, we obtain an identification
\[
\tilde{K}_0(\op{Q}_6) \stackrel{\sim}{\longrightarrow} [\op{Q}_6,\op{B}_{\Nis}\op{SL}_4]_{\aone} \isomto \hom(\K^{\op{M}}_3,\K^{\op{Q}}_3).
\]
Indeed, we claim that if ${\mathbf M}$ is any strictly $\aone$-invariant sheaf such that ${\mathbf M}_{-n-1} = 0$, then the canonical epimorphism $\K^{\MW}_n \to \K^{\op{M}}_n$ induces a bijection $\hom(\K^{\op{M}}_n,{\mathbf M}) \isomt \hom(\K^{\MW}_n,{\mathbf M})$; this follows from \cite[Lemma 5.1.5.2]{AsokWickelgrenWilliams}.  Applying this fact with $n = 3$ and ${\mathbf M} = \K^{\op{Q}}_3$, we conclude. \newline

\noindent {\bf Step 4.} Now, we claim that the morphism $\mu_3: \K^{\op{M}}_3 \to \K^{\op{Q}}_3$ discussed above Lemma \ref{lem:k3indsheaf} yields a generator of $\hom(\K^{\op{M}}_3,\K^{\op{Q}}_3) = \Z$. Indeed, by the conclusion of Step 3, the class of $\mu_3$ corresponds to a class in $\tilde{K}_0(\op{Q}_6)$.  By \cite[Theorem 4.3.4]{AsokDoranFasel}, a generator of $\tilde{K}_0(\op{Q}_6)$ is provided by applying the clutching construction to an explicit morphism $\op{Q}_5 \to \op{SL}_3$ defined using Suslin matrices.  By Lemma \ref{prop:mapsoutofspheresI} such a morphism yields a homomorphism $\K^{\MW}_3 \to \bpi_2^{\aone}(\op{SL}_3) \to \K^{\op{Q}}_3$.  By Step 3 again, this morphism factors through a morphism $\K^{\op{M}}_3 \to \K^Q_3$.  That this morphism coincides (up to a sign) with $\mu_3$ then follows by appeal to \cite[Lemma 3.8]{AsokFaselSpheres}.\newline

\noindent {\bf Step 5.} By Lemma \ref{lem:k3indsheaf}, the cokernel of $\mu_3$ is $\K^{\op{ind}}_3$.  Therefore, to conclude the proof it remains to observe that the $\op{SL}_4$-bundle over $\op{Q}_6$ corresponding to $\op{Spin}_7 \to \op{Spin}_7/\op{Spin}_6$ yields a generator of $\tilde{K}_0(\op{Q}_6)$; this follows immmediately from Lemma \ref{lem:q6chernclass}.
\end{proof}

\begin{cor}
\label{cor:pi3bspinn}
Assume $k$ is an infinite field having characteristic unequal to $2$.  For any integer $n \geq 7$, there are canonical isomorphisms
\[
\K^{\op{ind}}_3 \isomto \bpi_3^{\aone}(\op{B}_{\Nis}\op{Spin}_n).
\]
\end{cor}

\begin{proof}
For $n \geq 3$, under the hypotheses on the characteristic, by Proposition \ref{prop:spinstabilizationsequences} there are $\aone$-fiber sequences of the form $\op{Q}_n \to \op{B}_{\Nis}\op{Spin}_n \to \op{B}_{\Nis}\op{Spin}_{n+1}$.  For $n \geq 8$, $\op{Q}_n$ is at least $\aone$-$3$-connected, so the maps
\[
\bpi_3^{\aone}(\op{B}_{\Nis}\op{Spin}_n) \longrightarrow \bpi_3^{\aone}(\op{B}_{\Nis}\op{Spin}_{n+1})
\]
are isomorphisms for $n \geq 8$.  It remains to treat the case $n = 7$.  In that case, we have the exact sequence
\[
\bpi_3^{\aone}(\op{Q}_7) \longrightarrow \bpi_3^{\aone}(\op{B}_{\Nis}\op{Spin}_7) \longrightarrow \bpi_3^{\aone}(\op{B}_{\Nis}\op{Spin}_8) \longrightarrow 0.
\]
Since $\bpi_3^{\aone}(\op{Q}_7) = \K^{\MW}_4$ and since $\bpi_3^{\aone}(\op{B}_{\Nis}\op{Spin}_7) \cong \K^{\op{ind}}_3$ by Proposition \ref{prop:pi3bspin7}, the left hand arrow is a map $\K^{\MW}_3 \to \K^{\op{ind}}_3$.  However, $\hom(\K^{\MW}_4,\K^{\op{ind}}_3) = (\K^{\op{ind}}_3)_{-4} = 0$.  Thus, every such homomorphism is trivial, and we conclude that $\bpi_3^{\aone}(\op{B}_{\Nis}\op{Spin}_8) \cong \K^{\op{ind}}_3$ as well.
\end{proof}

The composite map
\[
\op{B}_{\Nis}\op{Spin}_n \longrightarrow \op{B}_{\Nis}\op{O}_n \longrightarrow \op{B}_{\et}\op{O}_n \cong \op{B}_{\op{gm}}\op{O}_n
\]
is compatible with stabilization and yields a map
\[
\op{B}_{\Nis}\op{Spin} \longrightarrow \op{B}_{\op{gm}}\op{O}.
\]
In \cite[Corollary 3]{AsokFaselKO}, we showed that $\bpi_3^{\aone}(\op{B}_{\op{gm}}\op{O}) = \K^{\op{ind}}_3$.  Thus, by appeal to Corollary \ref{cor:pi3bspinn}, the map on $\bpi_3^{\aone}(-)$ induced by the above map is an endomorphism of $\K^{\op{ind}}_3$.

\begin{cor}
\label{cor:computationscoincide}
The map
\[
\K^{\op{ind}}_3 = \bpi_3^{\aone}(\op{B}_{\Nis}\op{Spin}) \longrightarrow \bpi_3^{\aone}(\op{B}_{\op{gm}}\op{O}) = \K^{\op{ind}}_3
\]
is the identity.
\end{cor}

\begin{proof}
We unwind the definitions.  Indeed, \cite[Corollary 3]{AsokFaselKO} identifies $\bpi_3^{\aone}(\op{B}_{\op{gm}}\op{O})$ using the hyperbolic map $H: \op{B}_{\Nis}\op{SL} \to \op{B}_{\op{gm}}\op{O}$; this map evidently factors through $\op{B}_{\Nis}\op{O}$.  Now it suffices to observe that in Proposition \ref{prop:pi3bspin7} we identified $\bpi_3^{\aone}(\op{B}_{\Nis}\op{Spin}_7)$ as a quotient of $\bpi_3^{\aone}(\op{B}_{\Nis}\op{SL}_4) = \bpi_3^{\aone}(\op{B}_{\Nis}\op{SL})$.
\end{proof}

\begin{rem}
For an arbitrary algebraic group $\op{G}$, there is always the pointed map $\op{B}_{\Nis}\op{G} \to \op{B}_{\et}\op{G}$ (so the base-point in the latter corresponds to the trivial torsor).  This morphism is the inclusion of the connected compoment of the base-point in the Nisnevich local homotopy category by \cite[\S 4 Corollary 1.16]{MV}.  The results of this section suggest that in low degrees the induced maps on $\aone$-homotopy sheaves are isomorphisms for the orthogonal groups.  However, we do not know what the induced maps on $\aone$-homotopy sheaves look like in general.
\end{rem}

\subsubsection*{The third $\aone$-homotopy sheaf of $\op{B}_{\Nis}\op{G}_2$}
In order to describe $\bpi_3^{\aone}(\op{B}_{\Nis}\op{G}_2)$, we will need to recall a bit the computation of $\bpi_3^{\aone}({\op{B}_{\Nis}}\op{SL}_3)$ from \cite[Theorem 1.1]{AsokFaselSpheres}.  There, it was observed that there is an exact sequence of the form
\[
0 \longrightarrow \mathbf{S}_4 \longrightarrow \bpi_3^{\aone}(\op{B}_{\Nis}\op{SL}_3) \longrightarrow \K^{\op{Q}}_3 \longrightarrow 0,
\]
where the sheaf $\mathbf{S}_4$ is described in Lemma \ref{lem:structureofs4}.

\begin{thm}
\label{thm:pi3bg2}
If $k$ is an infinite field having characteristic unequal to $2$, then there is an exact sequence of the form:
\[
0 \longrightarrow \mathbf{S}_4/3 \longrightarrow \bpi_3^{\aone}(\op{B}_{\Nis}\op{G}_2) \longrightarrow \K^{\op{ind}}_3 \longrightarrow 0;
\]
furthermore, there is an epimorphism $\K^{\op{M}}_4/3 \to \mathbf{S}_4/3$ that becomes an isomorphism after $2$-fold contraction.
\end{thm}

\begin{proof}
The homotopy Cartesian square of Proposition \ref{prop:square} yields a commutative diagram of exact sequences of the form:
\[
\xymatrix{
                        & \K^{\MW}_3 \ar@{=}[r]\ar[d]                          & \K^{\MW}_3 \ar[d] & \\
\K^{\MW}_4 \ar[r] \ar@{=}[d] & \bpi_3^{\aone}(\op{B}_{\Nis}\op{SL}_3) \ar[r] \ar[d] & \bpi_3^{\aone}(\op{B}_{\Nis}\op{Spin}_6) \ar[r]\ar[d] & 0\\
\K^{\MW}_4 \ar[r]  & \bpi_3^{\aone}(\op{B}_{\Nis}\op{G}_2) \ar[r]\ar[d] & \bpi_3^{\aone}(\op{B}_{\Nis}\op{Spin}_7) \ar[r]\ar[d] & 0 \\
                  &   0                                                & 0.
}
\]
We now analyze this diagram more carefully.  \newline

\noindent {\bf Step 1.} The $\aone$-fiber sequence $\op{Q}_7 \to \op{B}_{\Nis}\op{G}_2 \to \op{B}_{\Nis}\op{Spin}_7$ in conjunction with the computation of Proposition \ref{cor:pi3bspinn} yields an exact sequence of the form
\[
\K^{\MW}_4 \longrightarrow \bpi_3^{\aone}(\op{B}_{\Nis}\op{G}_2) \longrightarrow \K^{\op{ind}}_3 \longrightarrow 0.
\]

\noindent {\bf Step 2.} By appeal to Proposition \ref{prop:hyperbolicvsstabilization}, under the exceptional isomorphism $\op{B}_{\Nis}\op{SL}_4 \cong \op{B}_{\Nis}\op{Spin}_6$ the map $\bpi_3^{\aone}(\op{B}_{\Nis}\op{SL}_3) \to \bpi_3^{\aone}(\op{B}_{\Nis}\op{Spin}_6)$ coincides with the map $\op{B}_{\Nis}\op{SL}_3 \to \op{B}_{\Nis}\op{SL}_4$ induced by stabilization.  As a consequence, the row
\[
\K^{\MW}_4 \longrightarrow \bpi_3^{\aone}(\op{B}_{\Nis}\op{SL}_3) \longrightarrow \bpi_3^{\aone}(\op{B}_{\Nis}\op{Spin}_6) \longrightarrow 0
\]
coincides with the exact sequence used to compute $\bpi_3^{\aone}(\op{B}_{\Nis}\op{SL}_3)$ in \cite[\S 3]{AsokFaselSpheres}.  In particular, we conclude that the map $\bpi_3^{\aone}(\op{B}_{\Nis}\op{SL}_3) \to \bpi_3^{\aone}(\op{B}_{\Nis}\op{Spin}_6) = \K^{\op{Q}}_3$ is surjective with kernel isomorphic to $\mathbf{S}_4$ and that the image of $\K^{\MW}_4$ in $\bpi_3^{\aone}(\op{B}_{\Nis}\op{SL}_3)$ is precisely $\mathbf{S}_4$. By a diagram chase, we conclude that $\bpi_3^{\aone}(\op{B}_{\Nis}\op{G}_2)$ is an extension of the sheaf $\K^{\op{ind}}_3$ by a quotient of $\mathbf{S}_4$. \newline

\noindent {\bf Step 3.} To determine the precise quotient of $\mathbf{S}_4$ that appears, we analyze the first column of the diagram.  The discussion of the two previous steps shows that the image of $\mathbf{S}_4$ in $\bpi_3^{\aone}(\op{B}_{\Nis}\op{G}_2)$ coincides with the image of $\K^{\MW}_4$ in the bottom row.  By appeal to Proposition \ref{prop:mapsoutofspheresI}, $\hom(\K^{\MW}_4,-)$ can be described by studying maps from $\op{Q}_7$ into the target.  Thus, we apply $\bpi_{3,4}^{\aone}(-)$ to the first column of the diagram to obtain an exact sequence of the form:
\[
\xymatrix{
\bpi_{3,4}^{\aone}(\op{Q}_6) \ar[r]\ar@{=}[d] & \bpi_{3,4}^{\aone}(\op{B}_{\Nis}\op{SL}_3) \ar[r]\ar@{=}[d] & \bpi_{3,4}^{\aone}(\op{B}_{\Nis}\op{G}_2) \ar[r] & 0 \\
\mathbf{W} & \Z/6\Z,
}
\]
where we have identified $\bpi_{3,4}^{\aone}(\op{B}_{\Nis}\op{SL}_3) \cong (\mathbf{S}_4)_{-4} = \Z/6\Z$.

By precomposition with endomorphisms of $\op{S}^3 \wedge \gm{\sma 4}$, this exact sequence is an exact sequence of $\K^{\MW}_0$-modules.  The homomorphism $\K^{\MW}_4 \to \bpi_3^{\aone}(\op{B}_{\Nis}\op{SL}_3)$ defining $\mathbf{S}_4$ is a generator of the latter as a $\K^{\MW}_0$-module.  On the other hand, $\bpi_{3,4}^{\aone}(\op{Q}_6) \cong \op{Q}_7$ is generated as a $\K^{\MW}_0$-module by a class $\eta$.  Thus, it suffices to understand the image of $\eta$ in $\Z/6\Z$.  Over a field having characteristic unequal to $2$, by appeal to Lemma \ref{lem:nontrivialq7toq6}, the image is precisely the $2$-torsion element $[3]$ of $\Z/6\Z$.  Therefore, the image of $\mathbf{S}_4$ in $\bpi_3^{\aone}(\op{B}_{\Nis}\op{G}_2)$ coincides with $\mathbf{S}_4/3$, which is what we wanted to show.
\end{proof}

\begin{rem}
One consequence of the above result is that the extension describing $\bpi_3^{\aone}(\op{B}_{\Nis}\op{SL}_3)$ must be non-split.
\end{rem}

\subsubsection*{Homotopy theory of low rank Lie groups}
Suppose $k$ is a field admitting an embedding $\iota$ in either the real or complex numbers.  If $G$ is a smooth $k$-group scheme, then we may consider the Lie groups $G(\real)$ or $G(\cplx)$ using the embedding.  The groups $G(\real)$ or $G(\cplx)$ have maximal compact subgroups, and it is well-known that the inclusion of a maximal compact subgroup is a homotopy equivalence.  The low-degree homotopy groups of compact Lie groups have been studied by many authors, but we will be particularly interested in the cases where $G = \op{G}_2, \op{Spin}_7, \op{SL}_3$ and $\op{SL}_4 \cong \op{Spin}_6$; the complex Lie groups $G(\cplx)$ are homotopy equivalent to, respectively, the compact Lie groups $G_2$, $Spin(7), SU(3)$ and $SU(4) = Spin(6)$.  The following result summarizes the known low-degree homotopy groups of each of these compact Lie groups; in very low degrees, the results are due to many authors, though we summarize the results from \cite{MimuraToda} \cite{Mimura}.

\begin{thm}
\label{thm:mimuratoda}
The homotopy groups of $SU(3), Spin(6) \cong SU(4), Spin(7)$, and $G_2$ in degrees $\leq 10$ are summarized in the following table:
\begin{center}
\begin{tabular}{ c | c | c | c | c | c | c | c | c }
$G \backslash i$   & 3 & 4 & 5 & 6 & 7 & 8 & 9 & 10 \\
\hline
$SU(3)$ & $\Z$ & $0$ & $\Z$ & $\Z/6\Z$  & 0  & $\Z/12\Z$  & $\Z/3\Z$  & $\Z/30\Z$ \\
\hline
$SU(4)$ & $\Z$ & $0$ & $\Z$ & $0$ & $\Z$ & $\Z/24\Z$ & $\Z/2\Z$ & $\Z/120\Z \oplus \Z/2\Z$\\
\hline
$Spin(7)$ & $\Z$ & $0$ & $0$ & $0$ & $\Z$ & $\Z/2\Z \oplus \Z/2\Z$ & $\Z/2\Z \oplus \Z/2\Z$ & $\Z/8\Z$ \\
\hline
$G_2$ & $\Z$ & $0$ & $0$ & $\Z/3\Z$ & $0$ & $\Z/2\Z$ & $\Z/6\Z$ & $0$.
\end{tabular}
\end{center}
\end{thm}

These results are established using a number of techniques, but among others, the careful analysis of fiber sequences involving the stated compact Lie groups.

\subsubsection*{Realizations of $\op{B}_{\Nis}\op{G}_2$}
\begin{thm}
\label{thm:complexrealizationofpi34bg2}
If $k$ is a separably closed field contained in $\cplx$ (resp. having characteristic unequal to $2$), then complex (resp. \'etale) realization induces isomorphisms of the form:
\[
\begin{split}
\bpi_{2,2}^{\aone}(\op{B}_{\Nis}\op{G}_2)(k) &\isomto \Z, \text{ and } \\
\bpi_{3,4}^{\aone}(\op{B}_{\Nis}\op{G}_2)(k) &\isomto \Z/3\Z.
\end{split}
\]
\end{thm}

\begin{proof}
For the first statement, observe that by appeal to Proposition \ref{prop:lowdegree}, we conclude that $\bpi_{2,2}^{\aone}(\op{B}_{\Nis}\op{G}_2) = \Z = \bpi_{2,2}^{\aone}(\op{B}_{\Nis}\op{SL}_3)$.  The result then follows by appeal to Bott periodicity; see the proof of \cite[Theorem 5.5]{AsokFaselSpheres} (we leave the details to the reader).  The second statement follows immediately from Lemma \ref{lem:nontrivialq7toq6} and functoriality of complex or \'etale realization.

\end{proof}

\begin{rem}
Note that our computations also see the {\em trivial} homotopy groups of ${\op{B}}\op{G}_2$ and ${\op{B}}\op{Spin}_7$ up to degree $7$.  Indeed, $\bpi_{2,3}^{\aone}(\op{B}_{\Nis}\op{G}_2) = (\K^{\op{M}}_2)_{-3} = 0$.  Likewise, one can show as above that $\bpi_{3,3}^{\aone}(\op{B}_{\Nis}\op{G}_2)(k) = 0$.  Similarly, $\bpi_{3,4}^{\aone}(\op{B}_{\Nis}\op{Spin}_7) = (\K_3^{\op{ind}})_{-4} = 0$ .
\end{rem}

\section{Classifying octonion algebras}
\label{sec:octonionclassification}

With the description of the first few $\A^1$-homotopy groups of ${\op{B}_{\Nis}}\op{G}_2$ in hand, we now turn to the question of classifying generically split octonion algebras over low-dimensional smooth affine varieties.  We begin by recalling a classification result for octonion algebras.

\subsection{Classifying generically split octonion algebras}
\label{ss:classificationinlowdimensions}
In this section, we give the general set-up for classifying octonion algebras: obstruction theory.  We begin by placing the problem in the context of $\aone$-homotopy theory by means of an affine representability theorem for octonion bundles.  Then, we explain how to describe $\aone$-homotopy classes of maps using the Postnikov tower.

\subsubsection*{Affine representability for octonion algebras}
The following result turns the problem of describing octonion algebras over smooth affine $k$-schemes over an infinite field into a problem of $\aone$-homotopy theory.

\begin{thm}
\label{thm:representability}
Assume $k$ is an infinite field.  If $R$ is any smooth affine $k$-algebra, then there is a bijection
\[
[\Spec R,\op{B}_{\Nis}\op{G}_2]_{\aone} \isomto \{ \text{generically split octonion algebras over } R\};
\]
this bijection is functorial in $R$.
\end{thm}

\begin{proof}
This result follows immediately by combining \cite[Theorem 1]{AffineRepresentabilityII} with Theorem \ref{thm:g2torsorsandoctonionalgebras}.
\end{proof}

\begin{rem}
If ``homotopy invariance" held for Nisnevich-locally trivial $\op{G}_2$-torsors over finite fields, then Theorem \ref{thm:representability} would hold over an arbitrary field.
\end{rem}

We may also slightly improve \cite[Theorem 4.2.2]{AffineRepresentabilityII}.

\begin{thm}
If $k$ is an infinite field, then the simplicial presheaf $\op{R}_{\Zar}\Singaone \op{Q}_6$ is Nisnevich local and $\aone$-invariant. If $R$ is any smooth $k$-algebra, then there is a functorial bijection
\[
\pi_0(\Singaone \op{Q}_6)(R) \isomto [\Spec R,\op{Q}_6]_{\aone}.
\]
\end{thm}

\begin{proof}
Combine point (2) of Theorem \ref{thm:g2sl3andspin7} and \cite[Theorems 2.3.2 and 3.3.6]{AffineRepresentabilityII}.
\end{proof}

\begin{rem}[Free vs. pointed]
\label{rem:freevspointed}
If $\mathscr{X}$ is a $k$-space, then we seek to describe $[\mathscr{X},\op{B}_{\Nis}\op{G}]$ where $\op{G}$ is a smooth $k$-group scheme which is $\mathbb{A}^1$-connected.  The usual techniques one uses in this situation naturally describe {\em pointed} $\aone$-homotopy classes and thus we quickly detour to describe the link between free and pointed homotopy classes.  We write $\mathscr{X}_+$ for $\mathscr{X}$ with a disjoint $k$-point attached.  Adjoining a disjoint base-point is left adjoint to the functor from pointed spaces to spaces given by forgetting the base-point.  This adjunction induces a bijection
\[
[\mathscr{X}_+,\op{B}_{\Nis}\op{G}]_{\aone} \longrightarrow [\mathscr{X},\op{B}_{\Nis}\op{G}]_{\aone}.
\]
We will use this identification routinely below to phrase our classification results in terms of {\em free} homotopy classes.
\end{rem}

\subsubsection*{The $\aone$-Postnikov tower of $\op{B}_{\Nis}\op{G}_2$}
By analogy with the situation in classical topology, we will use obstruction theory via the $\aone$-Postnikov tower to study classification problems.  The general theory of Postnikov towers in $\aone$-homotopy theory has been summarized in \cite[Theorem 6.1]{AsokFasel}; we specialize this discussion to the case of $\op{B}_{\Nis}\op{G}_2$.  The $\aone$-Postnikov tower of $\op{B}_{\Nis}\op{G}_2$ builds this space as a sequence of iterated principal fibrations, with fibers that are Eilenberg--Mac Lane spaces based on strictly $\aone$-invariant sheaves.  Briefly, there are spaces $\tau_{\leq i}(\op{B}_{\Nis}\op{G}_2)$ having $\aone$-homotopy sheaves concentrated in degrees $\leq i$, maps $\op{B}_{\Nis}\op{G}_2 \to \tau_{\leq i}(\op{B}_{\Nis}\op{G}_2)$ inducing isomorphisms on $\aone$-homotopy sheaves in degree $\leq i$ and maps $\tau_{\leq i}(\op{B}_{\Nis}\op{G}_2) \to \tau_{\leq i-1}(\op{B}_{\Nis}\op{G}_2)$ inducing isomorphisms on $\aone$-homotopy sheaves in degrees $\leq i-1$ that fit into a tower of the form
\[
\xymatrix{
& & \ar[dl] \op{B}_{\Nis}\op{G}_2 \ar[d] \ar[dr] & & \\
\cdots \ar[r]& \tau_{\leq i+1}(\op{B}_{\Nis}\op{G}_2) \ar[r] & \tau_{\leq i}(\op{B}_{\Nis}\op{G}_2) \ar[r] & \tau_{\leq i-1}(\op{B}_{\Nis}\op{G}_2) \ar[r] & \cdots.
}
\]
The following result, which is simply a reformulation of Proposition \ref{prop:lowdegree}, refines the observation that $\op{B}_{\Nis}\op{G}_2$ is $\aone$-$1$-connected.

\begin{prop}
\label{prop:secondstageBG2}
If $k$ is an infinite field, then there is an equivalence $\tau_{\leq 2}(\op{B}_{\Nis}\op{G}_2) \isomt \op{K}(\K^{\op{M}}_2,2)$.  The composite map
\[
\op{B}_{\Nis}\op{G}_2 \longrightarrow \tau_{\leq 2}(\op{B}_{\Nis}\op{G}_2) \isomto \op{K}(\K^{\op{M}}_2,2)
\]
is called the universal second Chern class.
\end{prop}

\begin{rem}
\label{rem:secondchernclass}
Suppose $X$ is a smooth affine $k$-scheme of dimension $2$ and $O$ is any generically split octonion algebra on $X$.  The underlying rank $8$ vector bundle has a second Chern class which is, in general, unequal to the second Chern class described above.  Indeed, suppose $\mathscr{E}$ is an oriented rank $3$ vector bundle on $X$ and let $\op{Zorn}(\mathscr{E})$ be the associated Zorn algebra.  The second Chern class of the underlying vector bundle is the second Chern class of $\O_X^{\oplus 2} \oplus \mathscr{E} \oplus \mathscr{E}^{\vee}$; since $c_2(\mathscr{E}^{\vee}) = c_2(\mathscr{E})$, the second Chern class of this rank $8$ module is $2c_2(\mathscr{E})$.  On the other hand, the map $\op{B}_{\Nis}\op{SL}_3 \to \op{B}_{\Nis}\op{G}_2$ is a $2$-equivalence and therefore the universal second Chern class of $\op{Zorn}(\mathscr{E})$ coincides with the second Chern class of $\mathscr{E}$ itself.

In general, the second Chern class of any associated vector bundle to a $\op{G}_2$-torsor is proportional to the second Chern class of the $\op{G}_2$-torsor, and the constant of proportionality may be described in entirely root-theoretic terms.  Note also that the universal second Chern class in $\op{H}^2(\op{B}_{\Nis}\op{G}_2,\K^{\op{M}}_2)$ corresponds to a central extension of the Nisnevich sheaf $\op{G}_2$ by $\K^{\op{M}}_2$ by the general theory of \cite{Giraud}.  As mentioned in \cite[Remark A.8]{MFM}, this central extension coincides with that constructed in \cite[\S 4]{BrylinksiDeligne} under our assumptions ($k$ is infinite).
\end{rem}

The next stage of the $\aone$-Postnikov tower yields an $\aone$-fiber sequence of the form
\[
\tau_{\leq 3}(\op{B}_{\Nis}\op{G}_2) \longrightarrow \tau_{\leq 2}(\op{B}_{\Nis}\op{G}_2)\stackrel{k_3}{\longrightarrow} \op{K}(\bpi_3^{\aone}(\op{B}_{\Nis}\op{G}_2),4),
\]
where the map $k_3$ is a so-called $k$-invariant.  Using this $\aone$-fiber sequence, we obtain the following result.

\begin{thm}
\label{thm:g2classification}
Assume $k$ is an infinite field and $X$ is a smooth affine $k$-scheme.
\begin{enumerate}[noitemsep,topsep=1pt]
\item If $X$ has dimension $\leq 2$, then the map $c_2: [X,\op{B}_{\Nis}\op{G}_2] \to \op{CH}^2(X)$ is bijective.
\item If $X$ has dimension $\leq 3$ and, furthermore, $k$ has characteristic unequal to $2$, then there is an exact sequence of the form
\[
\op{H}^1(X,\K^{\op{M}}_2) \xrightarrow{\Omega k_3} \op{H}^{3}(X,\K^{\op{M}}_4/3) \longrightarrow [X,\op{B}_{\Nis}\op{G}_2] \stackrel{c_2}{\longrightarrow} \op{CH}^2(X) \longrightarrow 0.
\]
In particular, if for every closed point $x \in X$, the residue field $\kappa_x$ is $3$-divisible (e.g., if $k$ is algebraically closed), then $c_2$ is bijective.
\end{enumerate}
\end{thm}

\begin{proof}
Mapping a smooth affine scheme $X := \Spec R$ into the fiber sequence from the $\mathbb{A}^1$-Postnikov tower yields an exact sequence of the form
\[
\begin{split}
\cdots \longrightarrow [X,\Omega \tau_{\leq 2}(\op{B}_{\Nis}\op{G}_2)]_{\aone} &\longrightarrow [X,\Omega \op{K}(\bpi_3^{\aone}(\op{B}_{\Nis}\op{G}_2),4)]_{\aone} \longrightarrow \\
[X,\tau_{\leq 3}(\op{B}_{\Nis}\op{G}_2)] \longrightarrow [X,\tau_{\leq 2}(\op{B}_{\Nis}\op{G}_2)]_{\aone} &\longrightarrow [X,\op{K}(\bpi_3^{\aone}(\op{B}_{\Nis}\op{G}_2),4)]_{\aone}
\end{split}
\]
This description may be simplified by appeal to Proposition \ref{prop:secondstageBG2} and using the fact that $\Omega \op{K}(\mathbf{A},n) = \op{K}(\mathbf{A},n-1)$.  In that case, we obtain an exact sequence of the form
\[
\begin{split}
\cdots \longrightarrow \op{H}^1(X,\K^{\op{M}}_2) &\longrightarrow \op{H}^3(X,\bpi_3^{\aone}(\op{B}_{\Nis}\op{G}_2)) \longrightarrow \\
[X,\tau_{\leq 3}(\op{B}_{\Nis}\op{G}_2)] &\longrightarrow \op{CH}^2(X) \longrightarrow [X,\op{K}(\bpi_3^{\aone}(\op{B}_{\Nis}\op{G}_2),4)]_{\aone}
\end{split}
\]
Since $X$ has dimension $\leq 3$, it follows that $[X,\op{K}(\bpi_3^{\aone}(\op{B}_{\Nis}\op{G}_2),4)]_{\aone} = \op{H}^4(X,\bpi_3^{\aone}(\op{B}_{\Nis}\op{G}_2)) = 0$.  On the other hand, by \cite[Proposition 6.2]{AsokFasel}, the map
\[
[X,\op{B}_{\Nis}\op{G}_2] \longrightarrow [X,\tau_{\leq 3}(\op{B}_{\Nis}\op{G}_2)]
\]
is a bijection.

It remains to describe $\op{H}^3(X,\bpi_3^{\aone}(\op{B}_{\Nis}\op{G}_2))$.  By Theorem \ref{thm:pi3bg2}, there is an exact sequence of the form
\[
0 \longrightarrow \mathbf{S}_4/3 \longrightarrow \bpi_3^{\aone}(\op{B}_{\Nis}\op{G}_2) \longrightarrow \K_3^{\op{ind}} \longrightarrow 0.
\]
Since $\K_3^{\op{ind}}$ is killed by a single contraction, it follows from the Gersten resolution that for {\em any} smooth $k$-scheme, the groups $\op{H}^i(X,\K^{\op{ind}}_3)$ vanish for $i \geq 1$.  Thus, the long exact sequence in cohomology attached to the above short exact sequence yields an isomorphism
\[
\op{H}^3(X,\bpi_3^{\aone}(\op{B}_{\Nis}\op{G}_2)) \isomto \op{H}^3(X,\mathbf{S}_4/3).
\]
Moreover, it follows from Theorem \ref{thm:pi3bg2} that the map $\K^{\op{M}}_4/3 \to \mathbf{S}_4/3$ induces an isomorphism after two contractions, and therefore from the Gersten resolution, we conclude that
\[
\op{H}^3(X,\K^{\op{M}}_4/3) \isomto \op{H}^3(X,\mathbf{S}_4/3).
\]
The group $\op{H}^3(X,\K^{\op{M}}_4/3)$ is trivial if
\begin{itemize}[noitemsep,topsep=1pt]
\item $X$ has dimension $\leq 2$
\item $X$ has dimension $3$ and the residue field $\kappa_x$ of any closed point $x \in X$ is $3$-divisible (e.g., if $k$ is algebraically closed).
\end{itemize}
The former statement is immediate, and the latter statement follows from \cite[Proposition 5.4]{AsokFasel}.
\end{proof}

\begin{rem}
The group $\op{H}^3(X,\K^{\op{M}}_4/3)$ is isomorphic to the motivic cohomology group $\op{H}^{7,4}(X,\Z/3\Z)$.  The operation $\Omega k_3$ even factors through a natural transformation of motivic cohomology groups.  However, we do not know if the map $\Omega k_3$ is {\em a priori} given by a stable motivic cohomology operation.
\end{rem}

Combining Theorem \ref{thm:g2classification} with Lemma \ref{lem:generictrivialitylowdimensions} one immediately deduces the following result.

\begin{cor}
\label{cor:splitinlowdimensions}
Let $X$ be a smooth affine variety of dimension $\leq 2$ over a field $k$.
\begin{enumerate}[noitemsep,topsep=1pt]
\item If $k(X)$ has $2$-cohomological dimension $\leq 2$, then an octonion algebra is uniquely determined by $c_2$.
\item If $X$ has the $\aone$-homotopy type of a smooth scheme of dimension $1$ and $k(X)$ has $2$-cohomological dimension $\leq 2$, then any octonion algebra over $X$ is split.
\end{enumerate}
\end{cor}

\subsubsection*{Examples}
\begin{ex}
\label{ex:c2examples}
Assume $k = \cplx$.  In that case, for any smooth affine $k$-surface $X$, one knows that $\op{CH}^2(X)$ is a uniquely divisible group by Roitman's theorem.  In particular, if $X$ admits a non-zero $2$-form (e.g., $X$ is an open subscheme of an abelian surface or a K3 surface), then $\op{CH}^2(X)$ is a non-trivial uniquely divisible group.  The manifold $X(\cplx)$ has the homotopy type of a CW-complex of dimension $\leq 2$ and thus $\op{H}^4(X,\Z) = 0$.  Therefore, the cycle class map $\op{CH}^2(X) \to \op{H}^4(X,\Z)$ is the trivial map.  One may give a topological classification of $\op{G}_2(\cplx)$-torsors over $X(\cplx)$ using obstruction-theoretic techniques.  In particular, since the map ${\op{B}}\op{SU}(3) \to {\op{B}}\op{G}_2$ is a $5$-equivalence, it follows that all $\op{G}_2(\cplx)$-torsors over $X(\cplx)$ are trivial.  By appeal to Grauert's Oka-principle, one concludes that there exist smooth complex affine surfaces with infinitely many pairwise non-isomorphic algebraic octonion algebras that become isomorphic upon passing to the holomorphic category (cf. Example \ref{ex:holomorphicsetting}).
\end{ex}

\begin{ex}
\label{ex:q6}
Assume we work over an infinite field having characteristic unequal to $2$.  It is easy to classify octonion algebras on $\op{Q}_6$.  We may always fix a base-point $\ast$ in $\op{Q}_6$.  Because $\op{B}_{\Nis}\op{G}_2$ is $\aone$-simply-connected by Proposition \ref{prop:lowdegree}, it follows, e.g., from \cite[Lemma 2.1]{AsokFaselSpheres} that the map from pointed to free homotopy classes
\[
[(\op{Q}_6,\ast),\op{B}_{\Nis}\op{G}_2]_{\aone} \longrightarrow [\op{Q}_6,\op{B}_{\Nis}\op{G}_2]_{\aone}
\]
is a bijection.  The left-hand side can be identified with the group of sections
\[
\bpi_{3,3}^{\aone}(\op{B}_{\Nis}\op{G}_2)(k) = \bpi_3^{\aone}(\op{B}_{\Nis}\op{G}_2)_{-3}(k).
\]
Using Theorem \ref{thm:pi3bg2} it remains to compute the $3$-fold contractions of $\K^{\op{ind}}_3$ and $\mathbf{S}_4/3$.  By Lemma~\ref{lem:k3indsheaf}, we conclude that $(\K^{\op{ind}}_3)_{-i} = 0$ for $i \geq 1$.  On the other hand, there is a map $\K^{\op{M}}_4/3 \to \mathbf{S}_4/3$ that becomes an isomorphism after $2$-fold contraction.  Therefore, $(\mathbf{S}_4/3)_{-3} \cong \K^{\op{M}}_1/3$ (see, e.g., \cite[Lemma 2.7]{AsokFaselSpheres}) and we conclude that there is a bijection
\[
k^{\times}/k^{\times 3} \isomto [\op{Q}_6,\op{B}_{\Nis}\op{G}_2]_{\aone}.
\]
\end{ex}

\begin{ex}
\label{ex:nontrivialiftingclass}
Assume we work over an infinite field having characteristic unequal to $2$.  It is easy to classify octonion algebras on $\op{Q}_7$.  As in Example~\ref{ex:q6} there are bijections $\bpi_{3,4}^{\aone}(\op{B}_{\Nis}\op{G}_2)(k) \isomt [\op{Q}_7,\op{B}_{\Nis}\op{G}_2]_{\aone}$. Appealing again to Theorem \ref{thm:pi3bg2} and the contraction computations in the preceding example, we conclude that there are bijections
\[
\Z/3\Z \isomto [\op{Q}_7,\op{B}_{\Nis}\op{G}_2]_{\aone}
\]
We may augment this example by observing that the ``lifting" coset in Theorem \ref{thm:g2classification} is actually non-trivial.  Indeed, by \cite[Lemma 4.5]{AsokFasel}, we see that $\op{H}^1(\op{Q}_7,\K^{\op{M}}_2) = 0$, while $\op{H}^3(\op{Q}_7,\K^{\op{M}}_4/3) \isomto \Z/3\Z$.  Thus, $\Omega k_3$ is necessarily the zero map and each lifting class determines an isomorphism class of octonion algebras over $\op{Q}_7$.  For example, the ``universal" octonion algebra $\op{Spin}_7 \to \op{Spin}_7/\op{G}_2$ is a non-trivial torsor and therefore corresponds to a non-trivial lifting class (see Proposition \ref{prop:notsplit} for a detailed proof of this last fast).
\end{ex}

\subsection{When are octonion algebras isomorphic to Zorn algebras?}
\label{sec:zorncomparison}
We now analyze the question of when a generically split octonion algebra over a smooth affine scheme lies in the image of $\op{Zorn}(-)$.  By appeal to the affine representability results discussed at the beginning of Section \ref{ss:classificationinlowdimensions}, and the analysis of the morphism $\op{Zorn}(-)$ in Section \ref{ss:aonefibersequences}, the question: ``when is a generically split octonion $R$-algebra a Zorn algebra" is equivalent to the question: ``when does a class in $[\Spec R,\op{B}_{\Nis}\op{G}_2]_{\aone}$ lift along the map $\op{Zorn}: \op{B}_{\Nis}\op{SL}_3 \to \op{B}_{\Nis}\op{G}_2$?

This lifting question is typically analyzed by appeal to the Moore--Postnikov factorization of the map $\op{Zorn}$.  In our setting, this can be viewed as an analysis of the $\aone$-Postnikov tower of the homotopy fiber of $\op{Zorn}$, which we know by Proposition \ref{prop:cartesiang2sl3}.  The theory of Moore--Postnikov factorizations of a map in $\aone$-homotopy theory in general is summarized in \cite[Theorem 6.1.1]{AsokFaselSplitting}, we specialize this discussion to the case of interest, which simplifies because all spaces under consideration are $\aone$-simply-connected.

\subsubsection*{The primary obstruction}
Fix an infinite field $k$ and a smooth affine $k$-scheme $X$.  Consider the lifting problem:
\[
\xymatrix{
& \op{B}_{\Nis}\op{SL}_3 \ar[d]^-{\op{Zorn}} \\
X \ar[r]\ar@{-->}[ur]^-{\exists ?} & \op{B}_{\Nis}\op{G}_2.
}
\]
Since Proposition \ref{prop:sl3g2fibration} identifies the $\aone$-homotopy fiber of $\op{Zorn}$ with $\op{Q}_6$ , by \cite[Theorem 6.1.1]{AsokFaselSplitting}, there is a well-defined {\em primary} obstruction to the existence of a lift.

\begin{thm}
\label{thm:primaryobstructiontobeingazornalgebra}
Assume $k$ is an infinite field and $X$ is a smooth affine $k$-scheme.
\begin{enumerate}[noitemsep,topsep=1pt]
\item If $X$ has the $\aone$-homotopy type of a smooth scheme of dimension $\leq 3$, then the map $[X,\op{B}_{\Nis}\op{SL}_3] \to [X,\op{B}_{\Nis}\op{G}_2]$ is a surjection. It is a bijection if $\dim X\leq 2$.
\item If $X$ has the $\aone$-homotopy type of a smooth scheme of dimension $\leq 4$, then there is well-defined primary obstruction to lifting a generically split octonion algebra classified by a map $X \to \op{B}_{\Nis}\op{G}_2$ to a map $X \to \op{B}_{\Nis}\op{SL}_3$. The primary obstruction takes values in
  \[
  \coker (\op{CH}^3(X) \longrightarrow \op{H}^4(X,\mathbf{I}^4)).
  \]
  In particular, if $k$ is algebraically closed, then the primary obstruction always vanishes.
\end{enumerate}
\end{thm}

\begin{proof}
The primary obstruction lives in $\op{H}^4(X,\bpi_3^{\aone}(\op{Q}_6)) \cong \op{H}^4(X,\K^{\op{MW}}_3)$.  This group always vanishes if $X$ has Nisnevich cohomological dimension $\leq 3$.  Now, appeal to the exact sequence
\[
0 \longrightarrow \mathbf{I}^4 \longrightarrow \K^{\op{MW}}_3 \longrightarrow \K^{\op{M}}_3 \longrightarrow 0.
\]
Note that $\op{H}^4(X,\K^{\op{M}}_3) = 0$ by appeal to the Gersten resolution for $\K^{\op{M}}_3$.  Then, it follows that there is an exact sequence
\[
\op{H}^3(X,\K^{\op{M}}_3) = \op{CH}^3(X) \longrightarrow \op{H}^4(X,\mathbf{I}^4) \longrightarrow \op{H}^4(X,\K^{\op{MW}}_3) \longrightarrow 0.
\]
If $k$ is algebraically closed, the relevant vanishing follows immediately from \cite[Proposition 5.2]{AsokFasel}.
\end{proof}

\begin{ex}
Using the techniques above, we may characterize those oriented rank $3$ projective modules over a commutative unital ring $R$ such that the octonion algebra by applying the Zorn construction is split.  Indeed, suppose $X = \Spec R$ and $P$ is an oriented rank $3$ projective $R$-module classified by an element $[P]$ of $[X,\op{B}_{\Nis}\op{SL}_3]_{\aone}$.  There is an exact sequence of the form
\[
[X,\op{Q}_6]_{\aone} \longrightarrow [X,\op{B}_{\Nis}\op{SL}_3]_{\aone} \longrightarrow [X,\op{B}_{\Nis}\op{G}_2].
\]
To say that $\op{Zorn}(P)$ is split is to say that the class in $[X,\op{B}_{\Nis}\op{G}_2]$ is the base-point.  In that case, there exists a class in $[X,\op{Q}_6]_{\aone}$ lifting $[P]$.  In other words, the oriented rank $3$ projective modules $P$ over $R$ such that $\op{Zorn}(P)$ is trivial are precisely those rank $3$ vector bundles that arise by pullback of the universal example along $X \to \op{Q}_6$.
\end{ex}

\begin{ex}
Let $\overline{X}$ be a smooth projective variety of dimension $3$ over $\mathbb{C}$ such that $\op{H}^0(\overline{X},\omega_{\overline{X}})\neq 0$, i.e., there is a global non-trivial holomorphic $3$-form. Let $X$ be the complement of a divisor in $\overline{X}$. By \cite[Theorem 2]{MurthySwan} and \cite[Proposition 2.1 and Corollary 5.3]{BlochMurthySzpiro}, the Chow group $\op{CH}^3(X)$ is a divisible torsion-free group of uncountable rank. In particular, there are uncountably many isomorphism classes of rank $3$ vector bundles over $X$ with $\op{c}_1=\op{c}_2=0$. These will all have trivial Zorn algebra by Theorem~\ref{thm:main1}.
\end{ex}

\subsubsection*{Non-triviality of the primary obstruction}
We now show that the obstruction to lifting can be realized at least on smooth affine schemes of ``homotopical dimension $4$'', coming from quadratic forms.  Indeed, one knows that $\op{H}^4(\op{Q}_8,\K^{\MW}_3) \cong (\K^{\MW}_3)_{-4} = \mathbf{W}(k)$.  The next result identifies precisely which elements of $\mathbf{W}(k)$ may be realized as the primary obstruction to lifting.

\begin{prop}
\label{prop:zornreductionex}
If $k$ is an infinite field having characteristic unequal to $2$, then there is an exact sequence of the form
\[
[\op{Q}_8,{\op{B}}\op{SL}_3]_{\A^1} \longrightarrow [\op{Q}_8,{\op{B}_{\Nis}}\op{G}_2]_{\A^1} \longrightarrow \mathbf{I}(k) \longrightarrow 0.
\]
\end{prop}

\begin{proof}
Mapping $\op{Q}_7$ into the loops on the $\aone$-fiber sequence
\[
\op{Q}_6 \longrightarrow \op{B}_{\Nis}\op{SL}_3 \longrightarrow {\op{B}_{\Nis}}\op{G}_2
\]
yields an exact sequence of the form
\[
[\op{Q}_7,\Omega \op{Q}_6]_{\aone} \longrightarrow [\op{Q}_7,\Omega\op{B}_{\Nis}\op{SL}_3 ]_{\aone} \longrightarrow [\op{Q}_7,\Omega {\op{B}_{\Nis}}\op{G}_2]_{\aone} \longrightarrow [\op{Q}_7,\op{Q}_6]_{\aone} \longrightarrow [\op{Q}_7,\op{B}_{\Nis}\op{SL}_3]_{\aone}.
\]
By adjunction, $[\op{Q}_7,\Omega \op{Q}_6]_{\aone} \cong [\op{Q}_8,\op{Q}_6]_{\aone}$, $[\op{Q}_7,\Omega\op{B}_{\Nis}\op{SL}_3 ]_{\aone} \cong [\op{Q}_8,\op{B}_{\Nis}\op{SL}_3 ]$, and $[\op{Q}_7,\Omega{\op{B}_{\Nis}}\op{G}_2]_{\aone} \cong [\op{Q}_8,{\op{B}_{\Nis}}\op{G}_2]_{\aone}$.

Now, $[\op{Q}_7,\op{Q}_6]_{\aone} \cong \bpi_{3,4}^{\aone}(\op{Q}_6)(k) \cong \mathbf{W}(k)$, and $[\op{Q}_7,\op{B}_{\Nis}\op{SL}_3]_{\aone} \cong \bpi_{3,4}^{\aone}(\op{B}_{\Nis}\op{SL}_3) \cong \Z/6\Z$ by appeal to \cite[Theorem 4.8]{AsokFaselSpheres}.  We claim the map $\mathbf{W}(k) \to \Z/6\Z$ is non-trivial and therefore factors $\mathbf{W}(k) \to \Z/2\Z \hookrightarrow \Z/6\Z$.  Indeed, this follows from the Proof of Theorem \ref{thm:pi3bg2}.  Since the kernel of the map $\mathbf{W}(k) \to \Z/2\Z$ is precisely the fundamental ideal $\mathbf{I}(k)$, we conclude by using the loop-suspension adjunction.
\end{proof}

\begin{rem}
To ``explain" the result above, we may appeal to real realization.  Indeed, $\op{SL}_3(\real)$ is homotopy equivalent to $SO(3)$ and $\op{G}_2(\real)$ is homotopy equivalent to $SU(2) \times SO(3)$.  The inclusion $\op{SL}_3(\real) \to \op{G}_2(\real)$ then corresponds to the inclusion of the second factor and up to homotopy a fiber sequence of the form
\[
SO(3) \longrightarrow SU(2) \times SO(3) \longrightarrow S^3
\]
Then, there is an exact sequence
\[
\pi_4(S^3) \longrightarrow \pi_4(BSO(3)) \longrightarrow \pi_4(BSU(2) \times BSO(3)) \longrightarrow \pi_3(S^3) \longrightarrow 0
\]
The left-hand map is necessarily trivial since $\pi_4(\op{S}^3) = \Z/2\Z$ and $\pi_4(BSO(3)) = \Z$, and the exact sequence in the middle becomes $0 \to \Z \to \Z \oplus \Z \to \Z \to 0$.  The real realization then induces a map from the exact sequence of the proposition to the above exact sequence.  In particular, $\mathbf{I}(\real) = \Z$.
\end{rem}

\begin{ex}
Proposition~\ref{prop:zornreductionex} provides a complete description of the set of octonion algebras over $\op{Q}_8$ which do not arise from the Zorn construction. In the case of number fields with several real embeddings, we actually see that different real embeddings give rise to algebraically nonisomorphic octonion algebras, evidenced by the signature morphisms from the Witt ring. The computation also implies that the abelian group of octonion algebras over $\op{Q}_8$ over a number field is not finitely generated. This demonstrates that in general the question of reducing an octonion algebra to a rank $3$ projective module is strongly related to the arithmetic of the base field. Note that Proposition~\ref{prop:zornreductionex} also implies that the non-reducible octonion algebras over $\op{Q}_8$ over a number field all become trivial upon base change to the algebraic (or quadratic) closure.
\end{ex}


\subsubsection*{The secondary obstruction}
Assuming the primary obstruction vanishes, we may continue our analysis of the lifting problem.  Indeed, return to the $\aone$-fiber sequence
\[
\op{Q}_6 \longrightarrow \op{B}_{\Nis}\op{SL}_3 \longrightarrow \op{B}_{\Nis}\op{G}_2.
\]
We know that $\bpi_3^{\aone}(\op{Q}_6) = \K^{\MW}_3$ and $\bpi_4^{\aone}(\op{Q}_6)$ has been determined under suitable additional hypotheses on the base field in \cite[Theorem 5.2.5]{AsokWickelgrenWilliams}.  The primary obstruction is determined by a $k$-invariant
\[
k_4: \tau_{\leq 3}(\op{B}_{\Nis}\op{SL}_3) \longrightarrow \op{K}(\bpi_3^{\aone}(\op{Q}_6),4)
\]
while the secondary obstruction to lifting is determined by a $k$-invariant
\[
k_5: \tau_{\leq 4}(\op{B}_{\Nis}\op{SL}_3) \longrightarrow \op{K}(\bpi_4^{\aone}(\op{Q}_6),5).
\]
Thus, we would like to analyze the groups $\op{H}^4(X,\bpi_3^{\aone}(\op{Q}_6))$ and $\op{H}^5(X,\bpi_4^{\aone}(\op{Q}_6))$ for $X$ a smooth affine scheme.

\begin{prop}
\label{prop:vanishingofsecondaryobstruction}
Assume $k$ is an algebraically closed field having characteristic $0$.  If $X$ is a smooth affine $k$-scheme of dimension $\leq 5$, then $\op{H}^5(X,\bpi_4^{\aone}(\op{Q}_6)) = 0$.
\end{prop}

\begin{proof}
By \cite[Theorem 5.2.5]{AsokWickelgrenWilliams} under the assumptions on $k$ we know there is an exact sequence of the form
\[
0 \longrightarrow \mathbf{F}_5' \longrightarrow \bpi_4^{\aone}(\op{Q}_6) \longrightarrow \mathbf{GW}^3_4 \longrightarrow 0,
\]
where $\mathbf{GW}^3_4$ is a Grothendieck--Witt sheaf and $\mathbf{F}_5'$ admits an epimorphism from $\K^{\op{M}}_5/24$ that becomes an isomorphism after $4$-fold contraction.  Therefore, by appeal to the long exact sequence in cohomology, it suffices to show that $\op{H}^5(X,\mathbf{GW}^3_4) = 0$ and $\op{H}^5(X,\K^{\op{M}}_5/24) = 0$ under the stated hypotheses.  By \cite[Proposition 3.4.3]{AsokFaselSplitting}, we know that $\mathbf{GW}^3_4$ is killed by $5$-fold contraction and thus the relevant vanishing follows immediately from the Gersten resolution for $\mathbf{GW}^3_4$.  On the other hand, $\op{H}^5(X,\K^{\op{M}}_5/24)$ vanishes by appeal to \cite[Proposition 5.4]{AsokFasel}.
\end{proof}

In contrast, on a smooth affine $5$-fold over an algebraically closed field, it is no longer necessarily the case that the primary obstruction always vanishes.  Indeed, we observed that there is an exact sequence
\[
\op{CH}^3(X) \longrightarrow \op{H}^4(X,\mathbf{I}^4) \longrightarrow \op{H}^4(X,\K^{\MW}_3) \longrightarrow 0.
\]
On the other hand, we know that $\op{H}^i(X,\mathbf{I}^5) = 0$ for $i \geq 4$ by \cite[Proposition 5.2]{AsokFasel}, and therefore, we conclude that the map
\[
\op{H}^4(X,\mathbf{I}^4) \longrightarrow \op{H}^4(X,\K^{\op{M}}_4/2)
\]
induced by the quotient map $\mathbf{I}^4 \to \K^{\op{M}}_4/2$ is an isomorphism.  Since $\op{H}^4(X,\K^{\op{M}}_4/2) = \op{CH}^4(X)/2$, one way to guarantee vanishing of the primary obstruction is to guarantee that $\op{Sq}^2: \op{CH}^3(X) \to \op{CH}^4(X)/2$ is surjective.

\begin{cor}
\label{cor:zornsecond}
Let $k$ be an algebraically closed field having characteristic $0$.  If $X$ is a smooth affine $k$-scheme of dimension $5$ for which $\op{CH}^4(X)/2$ is trivial, then any generically split octonion algebra over $X$ is isomorphic to the Zorn algebra of an oriented rank $3$ vector bundle on $X$.
\end{cor}

\begin{rem}
Note that $\op{H}^4_{\op{Nis}}(\op{Q}_6,\mathbf{K}^{\op{MW}}_3)\cong \mathbf{W}(k)$ is non-trivial, but this is not realizable as obstruction because $[\op{Q}_6,{\op{B}}_{\op{Nis}}\op{G}_2]_{\mathbb{A}^1}=0$.  In particular, non-trivial obstruction groups don't necessarily mean that lifting is impossible. Topologically, the first obvious example of a torsor which is non-liftable is the $2$-torsion element in $\pi_{10}({\op{B}}\op{G}_2)$ and then of course the non-torsion elements in $\pi_{12}({\op{B}}\op{G}_2)$. At the moment, it is not clear, what the expected range for the vanishing of obstructions should be over an algebraically closed field.
\end{rem}

\subsection{Octonion algebras with trivial spinor bundle}
\label{sec:spinorcomparison}
In this section, we discuss the relationship between generically split octonion algebras and their associated spinor bundles and norm forms.  As mentioned earlier, over a field, octonion algebras are determined up to isomorphism by their norm forms \cite[Theorem 33.19]{BookofInvolutions}.  Moreover, Bix showed that if $R$ is a local ring in which $2$ is invertible, then octonion algebras are determined up to isomorphism by their norm forms \cite[Lemma 1.1]{Bix}. The question of whether the same result holds over more general base rings was posed by H. Petersson in a lecture in Lens in 2012.  Gille answered Petersson's question negatively in \cite{GilleOctonion}, but his example is rather high-dimensional.  Here, we make a systematic analyis of Petersson's problem using our affine representability results.

First, the homomorphism $\op{G}_2 \to \op{O}_8$ discussed at the beginning of Section \ref{ss:homogeneousspaces} yields a map $\op{B}_{\Nis}\op{G}_2 \to \op{B}_{\Nis}\op{O}_8$.  If $X = \Spec R$, then the induced map
\[
[\Spec R,\op{B}_{\Nis}\op{G}_2]_{\aone} \to [\Spec R,\op{B}_{\Nis}\op{O}_8]_{\aone}
\]
is precisely the map sending a generically split octonion algebra to its associated norm form.

Because the morphism $\op{B}_{\Nis}\op{G}_2 \to \op{B}_{\Nis}\op{O}_8$ factors through $\op{B}_{\Nis}\op{Spin}_7$, the question: ``when is a generically split octonion algebra determined by its norm form" is equivalent to the question ``when is a generically split octonion algebra determined by its spinor bundle".  Moreover, the latter question is equivalent to: ``when is the map $[\Spec R,\op{B}_{\Nis}\op{G}_2]_{\aone} \to [\Spec R,\op{B}_{\Nis}\op{Spin}_7]_{\aone}$ injective?"  Using this translation, we are able to systematically analyze failure of injectivity and also to give a positive answer to Petersson's question over schemes of small dimension.

\subsubsection*{Analyzing the ``kernel"}
Consider the $\aone$-fiber sequence
\[
\op{Q}_7 \longrightarrow \op{B}_{\Nis}\op{G}_2 \longrightarrow \op{B}_{\Nis}\op{Spin}_7
\]
of Theorem \ref{thm:g2sl3andspin7}.  Given a smooth affine $k$-scheme $X$, mapping $X_+$ into this sequence yields an exact sequence of the form
\[
\cdots \longrightarrow [X,\op{Spin}_7]_{\aone} \longrightarrow [X,\op{Q}_7]_{\aone} \longrightarrow [X,\op{B}_{\Nis}\op{G}_2]_{\aone} \longrightarrow [X,\op{B}_{\Nis}\op{Spin}_7]_{\aone}.
\]
The kernel of the rightmost map of pointed sets consists (by construction) of generically trivial $\op{G}_2$-torsors whose associated $\op{Spin}_7$-torsor are trivial.  Moreover, by exactness, this kernel coincides with the quotient set
\[
[X,\op{Q}_7]_{\aone}/[X,\op{Spin}_7]_{\aone}.
\]
In particular, the kernel is trivial if the group $[X,\op{Spin}_7]_{\aone}$ acts transitively on $[X,\op{Q}_7]_{\aone}$.

The set $[X,\op{Q}_7]_{\aone}$ is $\aone$-weakly equivalent to the set of unimodular rows of length $n$ up to naive $\aone$-equivalence by \cite[Corollary 4.2.6 and Remark 4.2.7]{AffineRepresentabilityII}.  Moreover, if $X = \Spec R$, then the set $[X,\op{Spin}_7]_{\aone}$ coincides with the quotient of $\op{Spin}_7(R)$ by the normal subgroup of matrices $\aone$-homotopic to the identity by \cite[Theorem 4.3.1]{AffineRepresentabilityII}.  On the other hand, we observed before that there is an inclusion $\op{Spin}_6 \to \op{Spin}_7$.  Therefore, one way to guarantee that $[X,\op{Spin}_7]_{\aone}$ acts transitively on $[X,\op{Q}_7]_{\aone}$ is to require that the subgroup $[X,\op{Spin}_6]_{\aone} \cong [X,\op{SL}_4]_{\aone}$ acts transitively.

Mapping $X_+$ into the $\aone$-fiber sequence
\[
\op{Q}_7 \longrightarrow \op{B}_{\Nis}\op{SL}_3 \longrightarrow \op{B}_{\Nis}\op{SL}_4
\]
yields an exact sequence of the form
\[
[X,\op{SL}_4]_{\aone} \longrightarrow [X,\op{Q}_7]_{\aone} \longrightarrow [X,\op{B}_{\Nis}\op{SL}_3]_{\aone} \longrightarrow [\op{B}_{\Nis}\op{SL}_4]_{\aone}.
\]
Therefore, the quotient set $[X,\op{Q}_7]_{\aone}/[X,\op{SL}_4]_{\aone}$ coincides precisely with set of rank $3$ oriented vector bundles on $X$ that become trivial after addition of a trivial rank $1$ bundle.  Alternatively, if $X = \Spec R$, the set of rank $3$ oriented vector bundles that become trivial after addition of a trivial rank $1$ bundle coincides with the set of unimodular rows over $R$ of length $4$.  The next result uses these observations to show that under appropriate hypotheses on the base field, generically split octonion algebras over small dimensional smooth affine varieties with trivial spinor bundles are automatically split; this establishes Theorem \ref{thm:main3} stated in the introduction.

\begin{thm}
\label{thm:trivialnormform}
Assume $k$ is an infinite field.  If $X$ is a smooth affine $k$-variety, then generically split octonion algebras with trivial spinor bundles are split if all oriented rank $3$ vector bundles which become free after addition of a  free rank $1$ summand are already free.  In particular, generically split octonion algebras with trivial spinor bundle are always split if either
\begin{enumerate}[noitemsep,topsep=1pt]
\item $X$ has dimension $\leq 2$, or
\item $X$ has dimension $3$, (a) $k$ has characteristic unequal to $2$ or $3$ and has \'etale $2$ and $3$-cohomological dimension $\leq 2$, (b) $k$ is perfect of characteristic $2$ and has \'etale $3$-cohomological dimension $\leq 1$, (c) $k$ is perfect of characteristic $3$ and has \'etale $2$-cohomological dimension $\leq 1$, or
\item $k$ is algebraically closed, has characteristic unequal to $2$ or $3$ and $X$ has dimension $\leq 4$.
\end{enumerate}
\end{thm}

\begin{proof}
The first statement is immediate from discussion preceding the statement of the theorem.  For the final statements, we proceed as follows.  If $X$ has dimension $\leq 2$, then any oriented rank $3$ bundle that becomes free after addition of free rank $1$ summand is automatically stably free, and thus free by the Bass cancellation theorem \cite[Theorem 9.3]{Bass}.  If $X$ has dimension $3$, then it follows from Suslin's cancellation theorem that a rank $3$ oriented vector bundle that becomes free upon addition of a free rank $1$ summand is free \cite[Theorem 2.4]{SuslinCancellation} under any of the hypotheses above.  Finally, if $X = \Spec R$ has dimension $4$ and one works over an algebraically closed field having characteristic unequal to $2$ or $3$, then the Fasel--Rao--Swan theorem \cite{FaselRaoSwan} shows that any oriented rank $3$ vector bundle that is stably free is free.  Suppose $P$ is an oriented rank $3$ module.  By assumption, the rank $4$ modules $P \oplus R$ and $R^{\oplus 4}$ are stably isomorphic so by the cancellation theorem for rank $4$ modules over algebraically closed fields \cite{SuslinCancellationclosed}, we conclude that $P \oplus R$ and $R^{\oplus 4}$ are isomorphic.  Thus, every stably free oriented rank $3$ vector bundle which becomes trivial after addition of a trivial rank $1$ bundle is already free, and the result follows.
\end{proof}

The results above also yield a universal example of an octonion algebra with trivial spinor bundle that is itself non-trivial.  Indeed, in the $\aone$-fiber sequence
\[
\op{Q}_7 \longrightarrow \op{B}_{\Nis}\op{G}_2 \longrightarrow \op{B}_{\Nis}\op{Spin}_7,
\]
the composite map is null $\aone$-homotopic.  The map $\op{Q}_7 \to \op{B}_{\Nis}\op{G}_2$ classifies the $\op{G}_2$-torsor $\op{Spin}_7 \to \op{Spin}_7/\op{G}_2$.  The octonion algebra associated with this torsor is trivial if and only if the map $\op{Q}_7 \to \op{B}_{\Nis}\op{G}_2$ is null-homotopic.  The following result shows that this is not the case.

\begin{prop}
\label{prop:notsplit}
If $k$ is an infinite field having characteristic unequal to $2$, then the map $\op{Q}_7 \to \op{B}_{\Nis}\op{G}_2$ just described is not null-homotopic.  In particular, there is a non-trivial generically split octonion algebra on $\op{Q}_7$ with trivial spinor bundle.
\end{prop}

\begin{proof}
Consider the long exact sequence
\[
\bpi_3^{\aone}(\op{Q}_7) \longrightarrow \bpi_3^{\aone}(\op{B}_{\Nis}\op{G}_2) \longrightarrow \bpi_3^{\aone}(\op{B}_{\Nis}\op{Spin}_7) \longrightarrow 0.
\]
If the map in the statement was null $\aone$-homotopic, then the first map would be the zero map, and we would conclude that $\bpi_3^{\aone}(\op{B}_{\Nis}\op{G}_2) \to \bpi_3^{\aone}(\op{B}_{\Nis}\op{Spin}_7)$ is an isomorphism.  However, this contradicts the exact sequence in Theorem \ref{thm:pi3bg2}.
\end{proof}

\begin{rem}
Following an observation of Brion, Gille \cite[Concluding Remark (1)]{GilleOctonion} exhibited an example of an octonion algebra with hyperbolic norm form over the quadric $\op{Q}_7$, which corresponds to the $\op{G}_2$-torsor $\op{Spin}(7)\to \op{Spin}(7)/\op{G}_2\cong \op{Q}_7$.  Our result above shows that this example is, in an appropriate sense, the universal example of such an octonion algebra.  In a different direction, one might try to characterize those octonion algebras with a given associated spinor bundle.  As mentioned in the Remark \ref{rem:alsaodygilleremark} Alsaody and Gille have given precisely such a classification \cite{AlsaodyGille}.
\end{rem}

\subsubsection*{Characterizing generically split octonion algebras with trivial associated spinor bundle}
One also immediately deduces the following characterization of generically split octonion algebras whose associated spinor bundles are trivial.

\begin{prop}
\label{prop:explicitconstruction}
Assume $k$ is an infinite field, $X = \Spec R$ is a smooth affine $k$-scheme, and $O$ is a generically split octonion $R$-algebra.  The algebra $O$ has trivial associated spinor bundle if and only if there exists a unimodular row of length $4$ over $R$ with associated oriented projective $R$-module $P$ such that $O = \op{Zorn}(P)$.
\end{prop}

\begin{proof}
Assume we are given an octonion $R$-algebra $O$ as in the statement.  In that case, the classifying map of $O$ is a class in $[\Spec R,\op{B}_{\Nis}\op{G}_2]_{\aone}$ whose image in $[\Spec R,\op{B}_{\Nis}\op{Spin}_7]_{\aone}$ is the base-point.  As a consequence, the map lifts to a class in $[\Spec R,\op{Q}_7]_{\aone}$ (possibly non-uniquely).  However, elements of $[\Spec R,\op{Q}_7]_{\aone}$ are unimodular rows as in the statement.  The result then follows from the existence of the commutative diagram:
\[
\xymatrix{
[\Spec R,\op{Q}_7]_{\aone} \ar@{=}[d] \ar[r]& [\Spec R,\op{B}_{\Nis}\op{SL}_3]_{\aone} \ar[d]^-{\op{Zorn}} \\
[\Spec R,\op{Q}_7]_{\aone} \ar[r]            & [\Spec R,\op{B}_{\Nis}\op{G}_2]_{\aone},
}
\]
which arises from the $\aone$-homotopy Cartesian square in Proposition \ref{prop:square}.
\end{proof}

\begin{rem}
Gille asked about the existence of a mod $3$ invariant detecting octonion algebras with trivial norm form. Theorem \ref{thm:g2classification} in conjunction with Theorem \ref{thm:trivialnormform} provides one such manifestation of the existence of such an invariant.  From our point of view, the reason one might expect mod $3$ invariants to appear is a result of \cite[Theorem 3.1]{KumpelExceptional} giving a decomposition
\[
\pi_i(Spin(8))\cong\pi_i(G_2)\oplus\pi_i(Spin(8)/G_2)
\]
after inverting $3$.  This decomposition arises from the triality automorphism of $Spin(8)$ and, furthermore, implies that all topological octonion algebras with trivial norm form are related to $3$-torsion in the homotopy of $Spin(8)/G_2 \cong S^7 \times S^7$.
\end{rem}

\subsubsection*{Examples}
We discuss some examples of generically split octonion algebras with trivial associated spinor bundle; by appeal to Proposition~\ref{prop:explicitconstruction} such octonion algebras may be realized as Zorn algebras associated with oriented rank $3$ vector bundles associated with unimodular rows of length $4$.  It is then easy to show that the hypotheses in Theorem~\ref{thm:trivialnormform} are best possible by appeal to classical constructions in the theory of projective modules; the following example thus provides an answer to the question posed in \cite[Concluding Remarks 1]{GilleOctonion} about the minimal dimension of octonion algebras with split norm form.

\begin{thm}
\label{thm:mohankumar}
Assume $k$ is a field.  There exists a smooth affine $k$-scheme $X$ carrying a non-trivial generically split octonion algebra with trivial associated spinor bundle in either of the following situations:
\begin{enumerate}[noitemsep,topsep=1pt]
\item $k = F(t)$ for $F$ algebraically closed having characteristic unequal to $2$ and $X$ has dimension $\geq 4$, or
\item $k$ is algebraically closed having characteristic unequal to $2$ or $3$ and $X$ has dimension $\geq 5$.
\end{enumerate}
\end{thm}

\begin{proof}
In both cases, we are interested in building unimodular rows of length $4$ on smooth affine varieties whose associated rank $3$ vector bundles are non-trivial and for which the Zorn algebras of these rank $3$ vector bundles are non-split octonion algebras.  \newline

\noindent {\bf Point 1.} We begin by discussing a general procedure for building elements of $[X,\op{Q}_7]_{\aone}$, i.e., unimodular rows of length $4$ up to elementary equivalence and then we give a procedure to construct such unimodular rows whose associated rank $3$ vector bundles are non-trivial and whose Zorn algebras are non-trivial octonion algebras. \newline

\noindent {\bf Step 1.} Assume $X$ is a smooth affine $4$-fold over a field $k$ having $2$-cohomological dimension $\leq 1$.  In that case, we claim that the composite map
\[
[X,\op{Q}_7]_{\aone} \longrightarrow [X,\tau_{\leq 3}\op{Q}_7]_{\aone} = \op{H}^3(X,\K^{\MW}_4) \longrightarrow \op{H}^3(X,\K^{\op{M}}_4)
\]
is surjective.  Indeed, for any smooth affine $4$-fold, the first map is surjective by virtue of \cite[Proposition 6.2]{AsokFasel}.  For the surjectivity of the last map, observe that there is an exact sequence of the form
\[
\cdots \longrightarrow \op{H}^3(X,\K^{\MW}_4) \longrightarrow \op{H}^3(X,\K^{\op{M}}_4) \longrightarrow \op{H}^4(X,\mathbf{I}^5) \longrightarrow \cdots
\]
and $\op{H}^4(X,\mathbf{I}^5) = 0$ under the assumption that $k$ has $2$-cohomological dimension $\leq 1$ by appeal to \cite[Proposition 5.2]{AsokFasel}.  \newline

\noindent {\bf Step 2.} We now need a method to guarantee that the associated rank $3$ bundles or octonion algebras are actually non-trivial.  Equivalently, we want to know that the image of our class in $[X,\op{Q}_7]_{\aone}$ remains non-trivial in $[X,\op{B}_{\Nis}\op{SL}_3]_{\aone}$ or $[X,\op{B}_{\Nis}\op{G}_2]$.

We know that $[\op{Q}_7,\op{B}_{\Nis}\op{SL}_3]_{\aone} = \Z/6\Z$ and that this group surjects onto $[\op{Q}_7,\op{B}_{\Nis}\op{G}_2]_{\aone} = \Z/3\Z$.  Moreover, by Proposition \ref{prop:mapsoutofquadricsascohomology}, we know
\begin{eqnarray*}
[\op{Q}_7,\op{B}_{\Nis}\op{SL}_3]_{\aone} &\isomto& \op{H}^3(\op{Q}_7,\bpi_3^{\aone}(\op{B}_{\Nis}\op{SL}_3)), \text{ and } \\ {[\op{Q}_7,\op{B}_{\Nis}\op{G}_2]_{\aone}} &\isomto& \op{H}^3(\op{Q}_7,\bpi_3^{\aone}(\op{B}_{\Nis}\op{G}_2)).
\end{eqnarray*}
Given a map $\varphi: X \to \op{Q}_7$, the pullback of a class in $\op{H}^3(\op{Q}_7,\bpi_3^{\aone}(\op{B}_{\Nis}\op{G}_2))$ along this map yields a class in $\op{H}^3(X,\bpi_3^{\aone}(\op{B}_{\Nis}\op{G}_2))$ and to establish non-triviality of the associated octonion algebra, it suffices to establish non-triviality of this pulled back class.

We may use the exact sequence of Theorem \ref{thm:pi3bg2} to describe $\op{H}^3(X,\bpi_3^{\aone}(\op{B}_{\Nis}\op{G}_2))$ in more detail.  Indeed, the long exact sequence in cohomology associated with the short exact sequence of the statement takes the form:
\[
\cdots \longrightarrow \op{H}^2(X,\K^{\op{ind}}_3) \longrightarrow \op{H}^3(X,\op{S}_4/3) \longrightarrow \op{H}^3(X,\bpi_3^{\aone}(\op{B}_{\Nis}\op{G}_2)) \longrightarrow \op{H}^3(X,\K^{\op{ind}}_3) \longrightarrow \cdots
\]
Lemma \ref{lem:k3indsheaf} asserts that $\K^{\op{ind}}_3$ is killed by a single contraction.  Since the relevant cohomology groups may be computed by appeal to the Gersten resolution, we see that $\op{H}^2(X,\K^{\op{ind}}_3) = \op{H}^3(X,\K^{\op{ind}}_3) = 0$.  Furthermore, Theorem \ref{thm:pi3bg2} also asserts that the canonical epimorphism $\K^{\op{M}}_4/3 \to \mathbf{S}_4/3$ becomes an isomorphism after $2$-fold contraction and it follows that the induced map $\op{H}^3(X,\K^{\op{M}}_4/3) \to \op{H}^3(X,\op{S}_4/3)$ is an isomorphism.  Putting this all together, we conclude that $\op{H}^3(X,\K^{\op{M}}_4/3) \isomt \op{H}^3(X,\bpi_3^{\aone}(\op{B}_{\Nis}\op{G}_2))$.

Combining the observations of the previous two paragraphs, for any $\varphi: X \to \op{Q}_7$, the lifting coset of the octonion algebra we are trying to build is determined by the image of
\[
\varphi^*: \Z/6\Z \longrightarrow \op{H}^3(X,\K^{\op{M}}_4/6) \longrightarrow \op{H}^3(X,\K^{\op{M}}_4/3).
\]
(Note that, by construction, the resulting class does not lie in the image of $\op{H}^1(X,\K^{\op{M}}_2)$ under $\Omega k_3$).

The map $\bpi_3^{\aone}(\op{Q}_7) \to \bpi_3^{\aone}(\op{B}_{\Nis}\op{SL}_3)$ is the surjection onto $\mathbf{S}_4$ and the map $\bpi_3^{\aone}(\op{Q}_7) \to \bpi_3^{\aone}(\op{B}_{\Nis}\op{G}_2)$ is the surjection onto $\mathbf{S}_4/3$.  These quotient maps induce a commutative triangle:
\[
\xymatrix{
\op{H}^3(X,\K^{\MW}_4) \ar[r]\ar[dr]& \op{H}^3(X,\K^{\op{M}}_4/6) \ar[d]\\
& \op{H}^3(X,\K^{\op{M}}_4/3)
}
\]
where all the maps are the reduction maps.  Therefore, to check that a class in $[X,\op{Q}_7]_{\aone}$ corresponds to a non-trivial generically split octonion algebra, it suffices to show that the associated class in $\op{H}^3(X,\K^{\MW}_4)$ is non-trivial upon taking its reduction to $\op{H}^3(X,\K^{\op{M}}_4/3)$.  Equivalently, after Step 1, both horizontal maps factor through the surjection $\op{H}^3(X,\K^{\MW}_4) \to \op{H}^3(X,\K^{\op{M}}_4)$.  Thus, after the conclusion of Step 1, to build our unimodular rows, it suffices to write down classes in $\op{H}^3(X,\K^{\op{M}}_4)$ that remain non-trivial in the quotient $\op{H}^3(X,\K^{\op{M}}_4/6)$ or $\op{H}^3(X,\K^{\op{M}}_4/3)$.\newline

\noindent {\bf Step 3.} After the two preceding steps, we want to write down smooth affine varieties of dimension $4$ over a field $k$ of $2$-cohomological dimension $1$ together with a class in $\op{H}^3(X,\K^{\op{M}}_4)$ that remains non-trivial in $\op{H}^3(X,\K^{\op{M}}_4/3)$.  To this end, take $F$ to be an algebraically closed field having characteristic unequal to $2$, and let $k = F(t)$, which is $C_1$ and, in particular, has \'etale $2$-cohomological dimension $\leq 1$.

To write down a non-trivial class in $\op{H}^3(X,\K^{\op{M}}_4)$ that remains non-trivial after reduction mod $3$, let us suppose there exists a smooth affine $k$-variety $Y$ of dimension $4$ admitting an open cover $Y = U \cup V$ together with a non-trivial class $c \in \op{CH}^4(Y)$ which restricts to zero in $\op{CH}^4(U)$ and $\op{CH}^4(V)$.  In that case, the connecting homomorphism in the Mayer-Vietoris sequence for cohomology with coefficients in $\K^{\op{M}}_4$ takes the form
\[
\cdots \longrightarrow \op{H}^3(U \cap V,\K^{\op{M}}_4) \longrightarrow \op{CH}^4(Y) \longrightarrow \op{CH}^4(U) \oplus \op{CH}^4(V) \longrightarrow \cdots
\]
and exactness shows that $c$ lifts to a non-trivial class $\tilde{c} \in \op{H}^3(U \cap V,\K^{\op{M}}_4)$.  If, furthermore, $\op{CH}^4(Y) \to \op{CH}^4(Y)/3$ is an isomorphism, it follows that $\tilde{c}$ has non-trivial reduction mod $3$ by functoriality of Mayer-Vietoris sequences.
In that case, we may take $X = U \cap V$ and $\tilde{c}$ corresponds to a unimodular row of length $4$ with non-trivial associated octonion algebra by the discussion of the two previous steps.  The proof is completed by appeal to Lemma \ref{lem:explicitexample}.  \newline

\noindent {\bf Point 2.}  In the proof of (1), we constructed a unimodular row of length $4$ over a smooth affine $k$-variety $X$ of dimension $4$ with $k = F(t)$ and $F$ algebraically closed whose associated Zorn algebra is non-split.  By ``clearing the denominators" this unimodular row of length $4$ becomes a unimodular row over a smooth affine $F$-variety of dimension $5$.  Since the Zorn algebra of this unimodular row is non-split upon restriction to the open affine subscheme $X$ by construction, it must be non-split.
\end{proof}

\begin{lem}
\label{lem:explicitexample}
There exists a smooth affine variety $Y$ of dimension $4$ over a field $k$ and an open cover $Y = U \cup V$  with $U$ and $V$ affine such that $\op{CH}^4(Y) = \Z/3\Z$, and $\op{CH}^4(U) = \op{CH}^4(V) = 0$.
\end{lem}

\begin{proof}
We follow the construction of \cite{MohanKumarstablyfree}.  To start, suppose $k$ is an arbitrary field, and fix a polynomial $f$ of degree $3$ such that $f(0) \in k^{\times}$ and $f(x^{27})$ is irreducible.  Let $t_i = \frac{3^i-1}{3-1}$, and set $F(x_0,x_1) = x_1^3 f(\frac{x_0}{x_1})$.  Define polynomials $F_n(x_0,\ldots,x_n)$ inductively by the formula
\[
\begin{split}
F_1 &= F, \\
F_{i+1}(x_0,\ldots,x_{i+1}) &= F(F_i(x_0,\ldots,x_i),a^{t_i}x_{i+1}^{3^i})
\end{split}
\]
Let $Y$ be the complement in ${\mathbb P}^4$ of the variety defined by $F_4$.

We may cover $Y$ by the following two open sets.  Let $U$ be the intersection of $Y$ with the complement of the hypersurface defined $F_{3} = 0$.  Likewise, define $G = F_{3} - a^{t_{3}}x_{4}^{11}$ and set $V$ to be the complement of $G = 0$ intersected with $X$.  One checks that this is a Zariski open cover of $X$ (see \cite[p. 1441]{MohanKumarstablyfree}) .

As an open subscheme of ${\mathbb P}^4$, there is a degree-wise surjective ring map $\op{CH}^*({\mathbb P}^4) \to \op{CH}^*(Y)$; Mohan Kumar shows that $\op{CH}^4(Y) = \Z/p\Z$ generated by the class of a $k$-rational point of $Y$.  He also shows that the restriction map $\op{CH}^4(Y) \to \op{CH}^4(U)$ is zero and the repeating this argument one concludes also that $\op{CH}^4(Y) \to \op{CH}^4(U)$ is the zero map.
\end{proof}

\begin{rem}
The above argument essentially recasts \cite{MohanKumarstablyfree} (in the rank $3$ case) in the context of $\aone$-homotopy theory.  When working over $\cplx$, the rank $3$ vector bundles constructed above are actually holomorphically trivial as remarked in \cite{MohanKumarstablyfree}.  As a consequence the octonion algebras we describe are also holomorphically trivial.  In \cite{MKnote}, the third author will consider generalizations of the results above to constructing stably free non-free vector bundles of higher ranks; the arguments in these situations resemble the discussion above, but additional complications arise from the fact that $\bpi_n^{\aone}(\op{B}_{\Nis}\op{SL}_n)$ surjects onto $\K^{\op{Q}}_n$ rather than a quotient like $\K^{\op{ind}}_3$, which is cohomologically simpler.
\end{rem}

\begin{ex}
\label{ex:q6q7}
Any octonion bundle on either $\op{Q}_5$, $\op{Q}_6$ or $\op{Q}_7$ has trivial associated spinor bundle.  Indeed, the composite map of the classifying map $\op{Q}_i \to \op{B}_{\Nis}\op{G}_2$ of the octonion algebra with the map $\op{B}_{\Nis}\op{G}_2 \to \op{B}_{\Nis}\op{Spin}_7$ yields, by the same argument as in Example~\ref{ex:nontrivialiftingclass}, an element of the group $\bpi_{2,3}^{\aone}(\op{B}_{\Nis}\op{Spin}_7)(k)$, $\bpi_{3,3}^{\aone}(\op{B}_{\Nis}\op{Spin}_7)(k)$ or $\bpi_{3,4}^{\aone}(\op{B}_{\Nis}\op{Spin}_7)(k)$.  However, $\bpi_{2,3}^{\aone}(\op{B}_{\Nis}\op{Spin}_7) = (\K^{\op{M}}_2)_{-3}$ by Proposition \ref{prop:lowdegree} and the latter sheaf is trivial.  Similarly, $\bpi_{3,3}(\op{B}_{\Nis}\op{Spin}_7) = (\K^{\op{ind}}_3)_{-3}$ and $\bpi_{3,4}^{\aone}(\op{B}_{\Nis}\op{Spin}_7) = (\K^{\op{ind}}_3)_{-4}$ by Theorem \ref{thm:pi3bg2}, and both of these sheaves are trivial by Lemma~\ref{lem:k3indsheaf}.

It follows from Theorem \ref{thm:trivialnormform} that in each case above, non-trivial octonion algebras are Zorn algebras associated with oriented rank $3$ bundles arising from unimodular rows of length $4$.  Thus, by \cite[Corollary 4.7]{AsokFaselSpheres} since all oriented rank $3$ bundles on $\op{Q}_5$ are trivial, we conclude that all octonion algebras are split.  Similarly, on $\op{Q}_7$, representatives of all oriented rank $3$ vector bundles are described in \cite[Theorem 4.8]{AsokFaselSpheres}.  Indeed, they may be given by the unimodular rows $(x_1^n,x_2,x_3,x_4)$ for $1 \leq n \leq 6$.  Note, however, that not all of these unimodular rows yield non-trivial octonion algebras: indeed, the relevant octonion algebra will be non-trivial if and only if $n$ is not divisible by $3$.
\end{ex}

\subsubsection*{When are octonion algebras determined by their norm forms?}
Finally, we may analyze the question: when are octonion algebras determined by their norm forms.  We will show that this holds in sufficiently small dimensions as follows.  First, we show that under suitable dimension hypotheses on $X$, the set $[X,\op{B}_{\Nis}\op{G}_2]_{\aone}$ of generically split octonion algebras and $[X,\op{B}_{\Nis}\op{Spin}_7]_{\aone}$ of spinor bundles of rank $7$ are equipped with abelian group structures making the map $[X,\op{B}_{\Nis}\op{G}_2]_{\aone} \to [X,\op{B}_{\Nis}\op{Spin}_7]_{\aone}$ into a group homomorphism.  This group structure is the ``same" group structure one constructs in cohomotopy.  Then, we appeal to our determination of the kernel of this map above.

\begin{prop}
\label{prop:groupstructures}
Assume $k$ is a field.  If $X$ is a smooth affine $k$-scheme of dimension $d \leq 2$, then $[X,\op{B}_{\Nis}\op{G}_2]_{\aone}$ and $[X,\op{B}_{\Nis}\op{Spin}_7]_{\aone}$ admit functorial abelian group structures such that the map
\[
[X,\op{B}_{\Nis}\op{G}_2]_{\aone} \longrightarrow [X,\op{B}_{\Nis}\op{Spin}_7]_{\aone}
\]
is a homomorphism.
\end{prop}

\begin{proof}
This is an immediate consequence of \cite[Proposition 1.1.4]{AsokFaselcohomotopy} since both $\op{B}_{\Nis}\op{G}_2$ and $\op{B}_{\Nis}\op{Spin}_7$ are $\aone$-$1$-connected by Proposition \ref{prop:lowdegree}.  Note that the hypotheses on the base field used there are irrelevant since $\op{B}_{\Nis}\op{G}_2$ and $\op{B}_{\Nis}\op{Spin}_7$ are pulled back from spaces defined over the prime field.
\end{proof}

\begin{thm}
\label{thm:octonionalgebrasaredeterminedbytheirnormforms}
Assume $k$ is an infinite field and $X$ is an irreducible smooth affine $k$-scheme of dimension $d \leq 2$.
\begin{enumerate}[noitemsep,topsep=1pt]
\item Any generically split octonion algebra on $X$ is determined by its norm form.
\item If, furthermore, $k$ has characteristic unequal to $2$ and $k(X)$ has $2$-cohomological dimension $\leq 2$, then {\em any} octonion algebra on $X$ is determined by its norm form.
\end{enumerate}
\end{thm}

\begin{proof}
By Proposition \ref{prop:groupstructures} the map
\[
\op{H}^{1}_{\Nis}(X,\op{G}_2) = [X,\op{B}_{\Nis}\op{G}_2]_{\aone} \longrightarrow [X,\op{B}_{\Nis}\op{Spin}_7]_{\aone}
\]
is a homomorphism of abelian groups.  Therefore, this homomorphism is injective if and only if it has trivial kernel.  However, Theorem \ref{thm:trivialnormform} guarantees that the kernel is trivial in this situation.

For Point (2), it suffices to observe by appeal to Lemma \ref{lem:generictrivialitylowdimensions}, that the hypothesis on $k(X)$ ensures that the map
\[
\op{H}^{1}_{\Nis}(X,\op{G}_2) \longrightarrow \op{H}^1_{\et}(X,\op{G}_2)
\]
is a bijection.
\end{proof}

\begin{rem}
Theorem \ref{thm:octonionalgebrasaredeterminedbytheirnormforms} gives a positive answer to a question of Petersson and Gille, at least over schemes of small dimension.  In particular, if $X$ is a smooth affine curve over a field of $2$-cohomological dimension $1$, then octonion algebras are determined by their norm forms.  Similarly, if $X$ is a smooth affine surface over an algebraically closed field, then octonion algebras are determined by their norm forms.  We leave it to the interested reader to formulate an appropriate variant of Theorem \ref{thm:octonionalgebrasaredeterminedbytheirnormforms} in case $k$ has characteristic $2$ (see Remark \ref{rem:characteristic2generictriviality} for a hint).
\end{rem}

\begin{footnotesize}
\bibliographystyle{alpha}
\bibliography{octonion}

\begin{thebibliography}{KMRT98}

\bibitem[Ada96]{Adams}
J.~F. Adams.
\newblock {\em Lectures on exceptional {L}ie groups}.
\newblock Chicago Lectures in Mathematics. University of Chicago Press,
  Chicago, IL, 1996.
\newblock With a foreword by J. Peter May, Edited by Zafer Mahmud and Mamoru
  Mimura.

\bibitem[ADF16]{AsokDoranFasel}
A.~Asok, B.~Doran, and J.~Fasel.
\newblock Smooth models of motivic spheres and the clutching construction.
\newblock {\em Int. Math. Res. Not.}, 2016.
\newblock {\em Available online at}
  \url{http://dx.doi.org/10.1093/imrn/rnw065}.

\bibitem[AF14a]{AsokFaselSpheres}
A.~Asok and J.~Fasel.
\newblock Algebraic vector bundles on spheres.
\newblock {\em J. Topology}, 7:894--926, 2014.

\bibitem[AF14b]{AsokFasel}
A.~Asok and J.~Fasel.
\newblock A cohomological classification of vector bundles on smooth affine
  threefolds.
\newblock {\em Duke Math. J.}, 163:2561--2601, 2014.

\bibitem[AF15]{AsokFaselSplitting}
A.~Asok and J.~Fasel.
\newblock Splitting vector bundles outside the stable range and
  $\mathbb{A}^1$-homotopy sheaves of punctured affine space.
\newblock {\em J. Amer. Math. Soc.}, 28:1031--1062, 2015.

\bibitem[AF16]{AsokFaselcohomotopy}
A.~Asok and J.~Fasel.
\newblock Euler class groups and motivic stable cohomotopy.
\newblock 2016.
\newblock {\em Preprint}, available at \url{http://arxiv.org/abs/1601.05723}.

\bibitem[AF17]{AsokFaselKO}
A.~Asok and J.~Fasel.
\newblock An explicit {$\mathbf{KO}$-degree map and applications}.
\newblock {\em J. Topology}, 10:268--300, 2017.

\bibitem[AG17]{AlsaodyGille}
S.~Alsaody and P.~Gille.
\newblock Isotopes of octonion algebras and triality.
\newblock {\em Preprint}, 2017.

\bibitem[AHW15]{AffineRepresentabilityII}
A.~Asok, M.~Hoyois, and M.~Wendt.
\newblock Affine representability results in {$\mathbb{A}^1$}-homotopy theory
  {II}: principal bundles and homogeneous spaces.
\newblock 2015.
\newblock {\em Preprint}, available at \url{https://arxiv.org/abs/1507.08020}.

\bibitem[AHW17]{AffineRepresentabilityI}
A.~Asok, M.~Hoyois, and M.~Wendt.
\newblock Affine representability results in {$\mathbb{A}^1$}-homotopy theory
  {I}: vector bundles.
\newblock {\em {\em To appear} Duke Math. J.}, 2017.
\newblock {\em Preprint}, available at \url{https://arxiv.org/abs/1506.07093}.

\bibitem[Alp14]{Alper}
J.~Alper.
\newblock Adequate moduli spaces and geometrically reductive group schemes.
\newblock {\em Algebr. Geom.}, 1(4):489--531, 2014.

\bibitem[Ana73]{Anantharaman}
S.~Anantharaman.
\newblock {\em Sch{\'e}mas en groupes, espaces homog{\`e}nes et espaces
  alg{\'e}briques sur une base de dimension 1}, volume~33 of {\em M{\'e}moires
  de la Soci{\'e}t{\'e} Math{\'e}matique de France}.
\newblock 1973.

\bibitem[AWW15]{AsokWickelgrenWilliams}
A.~Asok, K.~Wickelgren, and B.~Williams.
\newblock The simplicial suspension sequence in $\mathbb{A}^1$-homotopy.
\newblock {\em to appear in G \& T}, 2015.

\bibitem[Bas64]{Bass}
H.~Bass.
\newblock {$K$}-theory and stable algebra.
\newblock {\em Inst. Hautes \'Etudes Sci. Publ. Math.}, (22):5--60, 1964.

\bibitem[BD01]{BrylinksiDeligne}
J.-L. Brylinski and P.~Deligne.
\newblock Central extensions of reductive groups by {$\bold K_2$}.
\newblock {\em Publ. Math. Inst. Hautes \'Etudes Sci.}, (94):5--85, 2001.

\bibitem[Bix81]{Bix}
R.~Bix.
\newblock Isomorphism theorems for octonion planes over local rings.
\newblock {\em Trans. Amer. Math. Soc.}, 266:423--439, 1981.

\bibitem[BMS89]{BlochMurthySzpiro}
S.~Bloch, M.P. Murthy, and L.~Szpiro.
\newblock Zero cycles and the number of generators of an ideal.
\newblock {\em M{\'e}m. S.M.F. (2)}, 38:51--74, 1989.

\bibitem[Bor50]{BorelSpheres}
A.~Borel.
\newblock Le plan projectif des octaves et les sph\`eres comme espaces
  homog\`enes.
\newblock {\em C. R. Acad. Sci. Paris}, 230:1378--1380, 1950.

\bibitem[CF15]{GroupesClassiques}
B.~Calm\`es and J.~Fasel.
\newblock Groupes classiques.
\newblock In {\em Autours des sch\'emas en groupes. {V}ol. {II}}, volume~46 of
  {\em Panor. Synth\`eses}, pages 1--133. Soc. Math. France, Paris, 2015.

\bibitem[Con14]{ConradReductive}
B.~Conrad.
\newblock Reductive group schemes.
\newblock In {\em Autour des sch\'emas en groupes. {V}ol. {I}}, volume 42/43 of
  {\em Panor. Synth\`eses}, pages 93--444. Soc. Math. France, Paris, 2014.

\bibitem[Con15]{Conrad}
B.~Conrad.
\newblock Non-split reductive groups over {${\bf Z}$}.
\newblock In {\em Autours des sch\'emas en groupes. {V}ol. {II}}, volume~46 of
  {\em Panor. Synth\`eses}, pages 193--253. Soc. Math. France, Paris, 2015.

\bibitem[DI04]{DuggerIsaksen}
D.~Dugger and D.~C. Isaksen.
\newblock Topological hypercovers and {$\Bbb A^1$}-realizations.
\newblock {\em Math. Z.}, 246(4):667--689, 2004.

\bibitem[FP80]{FriedlanderParshall}
E.~M. Friedlander and B.~Parshall.
\newblock Etale cohomology of reductive groups.
\newblock In E.~M. Friedlander and M.~R. Stein, editors, {\em Algebraic
  {K}-theory, Evanston 1980}, pages 127--140. Springer, 1980.

\bibitem[FRS12]{FaselRaoSwan}
J.~Fasel, R.~Rao, and R.G. Swan.
\newblock On stably free modules over affine algebras.
\newblock {\em Publ. Math. Inst. Hautes {\'E}tudes Sci.}, 116:223--243, 2012.

\bibitem[Gil14]{GilleOctonion}
P.~Gille.
\newblock Octonion algebras over rings are not determined by their norms.
\newblock {\em Bulletin Canadien de Math{\'e}matiques}, 57:303--309, 2014.

\bibitem[Gir71]{Giraud}
J.~Giraud.
\newblock {\em Cohomologie non ab\'elienne}.
\newblock Springer-Verlag, Berlin-New York, 1971.
\newblock Die Grundlehren der mathematischen Wissenschaften, Band 179.

\bibitem[Gro67]{EGAIV.4}
A.~Grothendieck.
\newblock {\'E}l\'ements de g\'eom\'etrie alg\'ebrique. {IV}. {\'e}tude locale
  des sch\'emas et des morphismes de sch\'emas {IV}.
\newblock {\em Inst. Hautes \'Etudes Sci. Publ. Math.}, (32):361, 1967.

\bibitem[Har90]{Harvey}
F.~R. Harvey.
\newblock {\em Spinors and calibrations}, volume~9 of {\em Perspectives in
  Mathematics}.
\newblock Academic Press, Inc., Boston, MA, 1990.

\bibitem[Isa04]{Isaksen}
D.~C. Isaksen.
\newblock Etale realization on the {$\Bbb A^1$}-homotopy theory of schemes.
\newblock {\em Adv. Math.}, 184(1):37--63, 2004.

\bibitem[Jac58]{Jacobson}
N.~Jacobson.
\newblock Composition algebras and their automorphisms.
\newblock {\em Rend. Circ. Mat. Palermo (2)}, 7:55--80, 1958.

\bibitem[Jac60]{JacobsonJordanII}
N.~Jacobson.
\newblock Some groups of transformations defined by {J}ordan algebras. {II}.
  {G}roups of type {$F_{4}$}.
\newblock {\em J. Reine Angew. Math.}, 204:74--98, 1960.

\bibitem[KMRT98]{BookofInvolutions}
M.-A. Knus, A.~Merkurjev, M.~Rost, and J.-P. Tignol.
\newblock {\em The book of involutions}, volume~44 of {\em American
  Mathematical Society Colloquium Publications}.
\newblock American Mathematical Society, Providence, RI, 1998.
\newblock With a preface in French by J. Tits.

\bibitem[Knu91]{Knus}
M.-A. Knus.
\newblock {\em Quadratic and {H}ermitian forms over rings}, volume 294 of {\em
  Grundlehren der Mathematischen Wissenschaften [Fundamental Principles of
  Mathematical Sciences]}.
\newblock Springer-Verlag, Berlin, 1991.
\newblock With a foreword by I. Bertuccioni.

\bibitem[KPS94]{KnusParimalaSridharan}
M.-A. Knus, R.~Parimala, and R.~Sridharan.
\newblock On compositions and triality.
\newblock {\em J. Reine Angew. Math.}, 457:45--70, 1994.

\bibitem[Kum65]{KumpelExceptional}
P.~G. Kumpel, Jr.
\newblock On the homotopy groups of the exceptional {L}ie groups.
\newblock {\em Trans. Amer. Math. Soc.}, 120:481--498, 1965.

\bibitem[LPR08]{LoosPeterssonRacine}
O.~Loos, H.P. Petersson, and M.L. Racine.
\newblock Inner derivations of alternative algebras over commutative rings.
\newblock {\em Algebra and Number Theory}, 2:927--968, 2008.

\bibitem[Mim67]{Mimura}
M.~Mimura.
\newblock The homotopy groups of {L}ie groups of low rank.
\newblock {\em J. Math. Kyoto Univ.}, 6:131--176, 1967.

\bibitem[MK85]{MohanKumarstablyfree}
N.~Mohan~Kumar.
\newblock Stably free modules.
\newblock {\em Amer. J. Math.}, 107(6):1439--1444 (1986), 1985.

\bibitem[Mor11]{MFM}
F.~Morel.
\newblock On the {F}riedlander-{M}ilnor conjecture for groups of small rank.
\newblock In {\em Current developments in mathematics, 2010}, pages 45--93.
  Int. Press, Somerville, MA, 2011.

\bibitem[Mor12]{MField}
F.~Morel.
\newblock {\em {$\mathbb{A}^1$}-algebraic topology over a field}, volume 2052
  of {\em Lecture Notes in Mathematics}.
\newblock Springer, Heidelberg, 2012.

\bibitem[MS76]{MurthySwan}
M.P. Murthy and R.G. Swan.
\newblock Vector bundles over affine surfaces.
\newblock {\em Invent. Math.}, 36:125--165, 1976.

\bibitem[MT64]{MimuraToda}
M.~Mimura and H.~Toda.
\newblock Homotopy groups of {${\rm SU}(3)$}, {${\rm SU}(4)$} and {${\rm
  Sp}(2)$}.
\newblock {\em J. Math. Kyoto Univ.}, 3:217--250, 1963/1964.

\bibitem[MV99]{MV}
F.~Morel and V.~Voevodsky.
\newblock {${\mathbf A}^1$}-homotopy theory of schemes.
\newblock {\em Inst. Hautes \'Etudes Sci. Publ. Math.}, (90):45--143 (2001),
  1999.

\bibitem[Nis84]{Nisnevich}
Y.~A. Nisnevich.
\newblock Espaces homog\`enes principaux rationnellement triviaux et
  arithm\'etique des sch\'emas en groupes r\'eductifs sur les anneaux de
  {D}edekind.
\newblock {\em C. R. Acad. Sci. Paris S\'er. I Math.}, 299(1):5--8, 1984.

\bibitem[Pet93]{Petersson}
H.~Petersson.
\newblock Composition algebras over algebraic curves of genus zero.
\newblock {\em Trans. Amer. Math. Soc.}, 337:473--493, 1993.

\bibitem[Ser95]{SerreUpdate}
J.-P. Serre.
\newblock Cohomologie galoisienne: progr\`es et probl\`emes.
\newblock {\em Ast\'erisque}, (227):Exp.\ No.\ 783, 4, 229--257, 1995.
\newblock S\'eminaire Bourbaki, Vol. 1993/94.

\bibitem[Ser02]{Serre}
J.-P. Serre.
\newblock {\em Galois cohomology}.
\newblock Springer Monographs in Mathematics. Springer-Verlag, Berlin, english
  edition, 2002.
\newblock Translated from the French by Patrick Ion and revised by the author.

\bibitem[{Sta}15]{stacks-project}
The {Stacks Project Authors}.
\newblock {\itshape Stacks Project}.
\newblock \url{http://stacks.math.columbia.edu}, 2015.

\bibitem[Sus77a]{SuslinCancellationclosed}
A.~A. Suslin.
\newblock A cancellation theorem for projective modules over algebras.
\newblock {\em Dokl. Akad. Nauk SSSR}, 236(4):808--811, 1977.

\bibitem[Sus77b]{SuslinStablyFree}
A.~A. Suslin.
\newblock Stably free modules.
\newblock {\em Mat. Sb. (N.S.)}, 102(144)(4):537--550, 632, 1977.

\bibitem[Sus82]{SuslinCancellation}
A.~A. Suslin.
\newblock Cancellation for affine varieties.
\newblock {\em Zap. Nauchn. Sem. Leningrad. Otdel. Mat. Inst. Steklov. (LOMI)},
  114:187--195, 222, 1982.
\newblock Modules and algebraic groups.

\bibitem[SV00]{SpringerVeldkamp}
T.~A. Springer and F.~D. Veldkamp.
\newblock {\em Octonions, {J}ordan algebras and exceptional groups}.
\newblock Springer Monographs in Mathematics. Springer-Verlag, Berlin, 2000.

\bibitem[Tot99]{Totaro}
B.~Totaro.
\newblock The {C}how ring of a classifying space.
\newblock In {\em Algebraic {K}-theory}, volume~67 of {\em Proc. Symp. Pure
  Math.}, pages 249--281. Amer. Math. Soc., 1999.

\bibitem[Wen10]{WendtChevRep}
M.~Wendt.
\newblock $\mathbb{A}^1$-homotopy groups of {C}hevalley groups.
\newblock {\em J. K-Theory}, 5:245--287, 2010.

\bibitem[Wen11]{WendtTorsors}
M.~Wendt.
\newblock Rationally trivial torsors in $\mathbb{A}^1$-homotopy theory.
\newblock {\em J. K-Theory}, 7:541--572, 2011.

\bibitem[Wen17]{MKnote}
M.~Wendt.
\newblock Variations in $\mathbb{A}^1$ on a theme of {M}ohan {K}umar.
\newblock Preprint, arXiv:1704.00141, 2017.

\bibitem[Zor33]{Zorn}
M.~Zorn.
\newblock Alternativk\"orper und quadratische {S}ysteme.
\newblock {\em Abh. Math. Sem. Univ. Hamburg}, 9(1):395--402, 1933.

\end{thebibliography}
\end{footnotesize}
\Addresses
\end{document}